\documentclass{amsart}
\usepackage{amscd}
\numberwithin{equation}{section}
\usepackage{amsmath}
\usepackage{latexsym}
\usepackage{graphicx}
\usepackage{subfigure}
\usepackage{color}
\usepackage{psfrag}
\usepackage{amssymb,amsfonts,amstext,graphicx}
\usepackage[curve]{xypic}
\usepackage[all]{xy}

\let\cal\mathcal

\newcommand{\cO}{\mathcal{O}}
\def\id{\text{id}}

\def\Sym{\mathop{\text{\upshape{Sym}}}\nolimits}

\def\diag{\operatorname {diag}}
\def\Ext{\operatorname {Ext}}

\def\gkdim{\operatorname {GK dim}}
\def\gldim{\operatorname {gl\,dim}}

\DeclareMathOperator{\Aut}{Aut}

\newcommand{\asX}{\langle X \rangle}

\newtheorem{lemma}{Lemma}[section]
\newtheorem{proposition}[lemma]{Proposition}
\newtheorem{theorem}[lemma]{Theorem}

\newtheorem{question}[lemma]{Question}
\newtheorem{corollary}[lemma]{Corollary}

\newtheorem{lemmas}{Lemma}[subsection]

\newtheorem{fact}[lemma]{Fact}

\newtheorem{convention}[lemma]{Convention}

\theoremstyle{definition}

\newtheorem{example}[lemma]{Example}
\newtheorem{definition}[lemma]{Definition}

\newtheorem{notation}[lemma]{Notation}

\newtheorem{problem}[lemmas]{Problem}

{

\newtheorem*{sketchofproof}{Sketch of the proof}

}

\theoremstyle{remark}

\newtheorem{remark}[lemma]{Remark}

\newcommand{\Dcal}{\mbox{$\cal D$}}

\newcommand{\Fcal}{\mbox{$\cal F$}}

\newcommand{\Lcal}{\mbox{$\cal L$}}

\newcommand{\Ncal}{\mbox{$\cal N$}}

\newcommand{\Rcal}{\mbox{$\cal R$}}

\newcommand{\stu}{\mathbin{\natural}}

\newcommand{\la}{{\triangleright}}

\newdimen\uboxsep \uboxsep=1ex
\def\uboxn#1{\vtop to 0pt{\hrule height 0pt depth 0pt\vskip\uboxsep
\hbox to 0pt{\hss #1\hss}\vss}}

\def\uboxs#1{\vbox to 0pt{\vss\hbox to 0pt{\hss #1\hss}
\vskip\uboxsep\hrule height 0pt depth 0pt}}

\title[Noninvolutive YBE]{A combinatorial approach to noninvolutive set-theoretic solutions of the Yang-Baxter equation
} \keywords{Yang-Baxter, Braided sets, Quadratic sets, Quadratic
algebras} \subjclass{Primary 81R50, 16W50, 16S36,
16S37}
\thanks{The author was partially supported
by The Max Planck Institute for Mathematics (MPIM), Bonn, by ICTP (Trieste), and by Grant KP-06 N 32/1 of 07.12.2019
of the Bulgarian National Science Fund}
\author{Tatiana Gateva-Ivanova}
\address{American University in
Bulgaria, 2700 Blagoevgrad, Bulgaria, and Max Planck Institute for Mathematics, Vivatsgasse 7, 53111 Bonn, Germany} \email{tatyana@aubg.edu}

\begin{document}

\date{\today}
\begin{abstract}
We study noninvolutive set-theoretic solutions $(X,r)$ of the
Yang-Baxter equations in terms of the properties of the
canonically associated braided monoid $S(X,r)$, the quadratic
Yang-Baxter algebra $A= A(\textbf{k}, X, r)$ over a field
$\textbf{k}$ and its Koszul dual, $A^{!}$. More generally, we
continue our systematic study of \emph{nondegenerate quadratic sets}
$(X,r)$ \emph{and their associated algebraic objects}. Next we investigate
the class of (noninvolutive) square-free solutions $(X,r)$. It
contains the self distributive solutions (quandles). We make a
detailed characterization in terms of various algebraic and
combinatorial properties each of which shows the contrast between
involutive and noninvolutive square-free solutions. We introduce
and study a class of finite square-free braided sets $(X,r)$ of
order $n\geq 3$ which satisfy "\emph{the minimality condition}",
that is $\dim_{\textbf{k}} A_2 =2n-1$. Examples are some simple
racks of prime order $p$. Finally, we discuss general extensions
of solutions and introduce the notion of \emph{a generalized
strong twisted union of braided sets}. We prove that if $(Z,r)$ is
a non-degenerate 2-cancellative braided set splitting as a
generalized strong twisted union $Z = X\stu^{\ast} Y$, of
$r$-invariant subsets, then its braided monoid $S_Z$ is a
generalized strong twisted union $S_Z= S_X\stu^{\ast} S_Y$ of the
braided monoids $S_X$ and $S_Y$. We propose a construction of a generalized strong twisted
union $Z = X\stu^{\ast} Y$ of braided sets $(X,r_X)$, and $(Y,
r_Y)$, where the map $r$ has a high, explicitly prescribed order.
\end{abstract}
 \maketitle

\setcounter{tocdepth}{1}
\tableofcontents

\section{Introduction}
\label{Intro} It was established in the last three decades that
solutions of the linear braid or Yang-Baxter equations (YBE) on
a vector space of the form $V^{\otimes 3}$ lead to remarkable
algebraic structures. Here $r : V\otimes V \longrightarrow
V\otimes V$ , $r^{12}= r\otimes \id$, $r^{23} = id\otimes r$ is a
notation and structures include coquasitriangular bialgebras
$A(r)$, their quantum group (Hopf algebra) quotients, quantum
planes and associated objects, at least in the case of specific
standard solutions, see \cite{MajidQG, RTF}. On the other hand,
the variety of all solutions on vector spaces of a given dimension
has remained rather elusive in any degree of generality. It was
proposed by V.G. Drinfeld \cite{D}, to consider the same equations
in the category of sets, and in this setting numerous  results
were found. It is clear that a set-theoretic solution extends to a
linear one, but more important than this is that set-theoretic
solutions lead to their own remarkable algebraic and combinatoric
structures, only somewhat analogous to quantum group
constructions. In the present paper we continue our systematic
study of set-theoretic solutions based on the associated quadratic
algebras and  monoids that they generate.

More generally, we study quadratic sets and their algebraic
objects. The notions of a quadratic set $(X,r)$ and its related
algebraic objects were introduced by the author and studied first
in \cite{GI04}, see also \cite{GIM08, GI11} for more results on
quadratic sets $(X,r)$. We shall use the terminology,  notation
and some results from \cite{GI04, GI11, GIC, GI18, GIM08}.

\begin{definition}\cite{GI04}
\label{def:quadraticsets_All} Let $X$ be a nonempty set (possibly
infinite) and let $r: X\times X \longrightarrow X\times X$ be a
bijective map. In this case we use notation $(X, r)$ and refer to
it as \emph{a quadratic set}. The image of $(x,y)$ under $r$ is
presented as
\[
r(x,y)=({}^xy,x^{y}).
\]
This formula defines a ``left action'' $\Lcal: X\times X
\longrightarrow X,$ and a ``right action'' $\Rcal: X\times X
\longrightarrow X,$ on $X$ as: $\Lcal_x(y)={}^xy$, $\Rcal_y(x)=
x^{y}$, for all $x, y \in X$. (i) $(X, r)$ is \emph{non-degenerate}, if
the maps $\Lcal_x$ and $\Rcal_x$ are bijective for each $x\in X$.
(ii) $(X, r)$ is \emph{involutive} if $r^2 = id_{X\times X}$. (iii)
$(X,r)$ is \emph{square-free} if $r(x,x)=(x,x)$ for all $x\in X.$
(iv) $(X,r)$ is \emph{quantum binomial} if it is nondegenerate,
square-free, and involutive. (v) $(X, r)$ is \emph{a set-theoretic
solution of the Yang--Baxter equation} (YBE) if  the braid
relation
\[r^{12}r^{23}r^{12} = r^{23}r^{12}r^{23}\]
holds in $X\times X\times X,$  where  $r^{12} = r\times\id_X$, and
$r^{23}=\id_X\times r$. In this case we  refer to  $(X,r)$ also as
\emph{a braided set}. A braided set $(X,r)$ with $r$ involutive is
called \emph{a symmetric set}.
\end{definition}

In this paper we always assume that $r$ is nondegenerate.
As a notational tool, we  shall often identify the sets $X^{\times
m}$ of ordered $m$-tuples, $m \geq 2,$  and $X^m,$ the set of all
monomials of length $m$ in the free monoid $\asX$.

As in our previous works (\cite{GI04, GI11, GIC, GI18, GIM08}) to
each quadratic set $(X, r)$  we associate canonically several
algebraic objects (see Definition \ref{def:algobjects}) generated
by $X$ and with quadratic defining relations naturally determined
as
\[
xy=y^{\prime} x^{\prime}\in \Re(r)\; \text{iff}\;
 r(x,y) = (y^{\prime}, x^{\prime})\; \text{and} \;
 (x,y) \neq (y^{\prime}, x^{\prime})\;\text{hold in}\;X \times
 X.\]
  Note that
in the case when X is finite, the set $\Re(r)$  of defining
relations is also finite, therefore the associated algebraic
objects are finitely presented.

We continue our systematic study of the close
relations between the combinatorial properties of the defining
relations, i.e. of the map $r$, and the structural properties of
the associated algebraic objects.

\begin{remark}
\label{rem:changes}
We thank Magdalena Wiertel  and Arne van Antwerpen who communicated the author an example of a finite non involutive square-free solution whose Yang-Baxter algebra is PBW and therefore Koszul, see Example \ref{ex:Arne} in Sec 3.
which demonstrates that \cite[Proposition 3.12, parts (1) and (2)]{GI21} is not true. 

The example inspired the corrections and improvements we made in the present most recent version of the paper \cite{GI21}.
 
A brief description of the changes follow.
Section 3 of the original version is revised. In particular, we replaced \cite[Remark 3.9]{GI21} with Subsection 3.3. which gives a (possibly new) set of reduced relations of the Yang-Baxter algebra $A= A(\textbf{k}, X, r)$ and  a set of correct explicit relations of the Koszul dual algebra $A^{!}$. We remark that parts (1) and (2) of \cite[Proposition 3.12]{GI21}  are not correct and erase this proposition. We give corrected and improved versions of the three results in the original paper 
\cite{GI21} whose proofs involved \cite[Proposition 3.12]{GI21}: these are Theorem 3.17 which is a corrected version \cite[Theorem 3.16]{GI21}, 
a corrected version of \cite[Theorem 5.5]{GI21} and its proof, see the present Theorem 5.5., and a corrected version of \cite[Corollary 5.6]{GI21}.
\end{remark}

In the first half of the paper we investigate non-degenerate
quadratic sets $(X,r)$ of finite order, their quadratic graded
algebras $A$, and the monoid $S(X,r)$. Section \ref{Preliminaries}
contains preliminary material on quadratic sets. In Section \ref{sec:quadratic} we study
non-degenerate quadratic sets $(X,r)$,  with 2-cancellation.
Proposition \ref{prop:mainthlemma1} provides upper and  lower
bounds for the dimension $\dim A_2$, and shows that the upper
bound is attained, whenever $r$ is involutive. The main result of
the section is Theorem \ref{thm:new_theorem}. It implies, in
particular, that a square-free non-degenerate quadratic set
$(X,r)$, with $|X|= n$ is a symmetric set if and only if its
quadratic algebra $A$ has Hilbert series $H_A(z)=
\frac{1}{(1-z)^n}$. The theorem improves an old result of the
author, see \cite{GI12}, Theorem 2. In Section
\ref{sec:cycliccondition} we pay special attention to square-free
quadratic sets with cyclic conditions. We find some new
combinatorial results, see Theorem \ref{thm:proposition1AS} and
use them to show that, surprisingly, a square-free quadratic set
$(X,r)$ of finite order $|X|=n$ which satisfies  the cyclic
conditions is a symmetric set if and only if
 $\dim_{\textbf{k}} A^{!}_3 = \binom{n}{3},$ see Proposition \ref{prop:A!_and_YBE}.
In Section \ref{sec:contrast} we study square-free braided sets
and the contrast between the involutive and noninvolutive cases.
We show that every square-free braided set (of arbitrary
cardinality) satisfies the cyclic conditions.  Theorem
\ref{mainth} characterizes the involutive braided sets $(X,r)$ in
terms of various equivalent properties of the algebra $A$, its
Koszul dual $A^{!}$ and the monoid $S(X,r)$. Corollary
\ref{mainthCor} provides a characterization of
\emph{noninvolutive} square-free braided sets. In Section
\ref{sec:minimalitycond} we introduce quadratic sets $(X,r)$ which
satisfy \emph{"the minimality condition"}, \textbf{M}, that is
$\dim_{\textbf{\textbf{k}}} A_2 =2n-1$, see Definition
 \ref{minimality_def}.
We first investigate (general) square-free 2-cancellative
quadratic sets $(X,r)$ with minimality condition and prove
Proposition \ref{prop:minimalityprop}. We make some initial steps
in the study of braided sets, and in particular, quandles
 with minimality condition, \textbf{M}.  Corollary
\ref{Mcorollary} implies that every square-free self distributive
solution,  $(X,r)$ (see Definition \ref{SDdef}), corresponding to
a dihedral quandle of prime order $|X|= p>2$ satisfies the
minimality condition \textbf{M}.
 In Sec. \ref{sec:specialext} we propose a construction
which generates noninvolutive extensions $(Z,r)$ of braided (or
symmetric) sets, where the map $r$ has high, explicitly prescribed
order, see Theorem \ref{theor:irregularext}. In Section
\ref{BraidedMonoidSec} are studied braided monoids $S(X,r)$ and
extensions of solutions. We consider "general" extensions of
braided sets.  In Subsec. \ref{subsec:StuBraidedsets} we introduce
"generalized strong twisted unions $Z = X\stu^{\ast} Y$ of
non-degenerate braided sets", see Definition \ref{STUgendef}. The
main result of the section is Theorem \ref{stuthm}.
Finally, in Section \ref{sec:questions} we give a list
of questions and problems. Some of these are still open questions, other were posed in earlier versions of our work and have
stimulated recent results of other authors.

\section{Preliminaries}
\label{Preliminaries}

During the last two decades the study of
set-theoretic solutions of the Yang-Baxter equation and related
structures notably intensified, a relevant selection of works for the
interested reader
is \cite{D, GIVB, ESS, LYZ, GI04, Carter, Catino, Rump, Takeuchi, GIM08, GI12, GIC, GI18, CJO14, V16,
GuaVendramin17,  BCV, LeVendramin18, SmokVendr18, Soloviev,  Kazhdan}, et all, and the references there in.
In this section we recall basic notions and
results which will be used in the paper.
 We shall use the terminology,  notation and some results from
\cite{GI04, GI11, GIC, GI18, GIM08}.

\begin{remark}
\label{rem_Invol-Sqfree}
Let $(X,r)$  be a  quadratic set, and let ${}^x\bullet$, and ${\bullet}^x$ be the associated left and right
actions. Then
\begin{enumerate}
\item
The map  $r$ is involutive   \emph{iff} the actions satisfy:
\begin{equation}
\label{involeq}
{}^{{}^uv}{(u^v)}= u, \; \text{and} \; ({}^uv)^{u^v} = v, \; \forall u, v \in X.
\end{equation}
\item $r$ is square-free if and only if $\; {}^xx=x$, and $x^x =
x, \; \forall\; x \in X$.
 \item If $r$ is nondegenerate and
square-free, then
\begin{equation}
\label{rightactioneq1}
\begin{array}{lllll}
{}^zt={}^zu &\Longrightarrow &t = u &\Longleftarrow & t^z= u^z\\
{}^zt = z&\Longleftrightarrow &t = z &\Longleftrightarrow &t^z =
z.
\end{array}
\end{equation}
\end{enumerate}
\end{remark}

\begin{remark} \label{ybe}
\cite{ESS}
Let $(X,r)$ be quadratic set.
Then $r$ obeys the YBE,
that is $(X,r)$ is a braided set {\em iff} the following conditions
hold for all $x,y,z \in X$:
\[
\begin{array}{lclc}
 {\bf l1:}\quad& {}^x{({}^yz)}={}^{{}^xy}{({}^{x^y}{z})},
 \quad\quad\quad
 & {\bf r1:}\quad&
{(x^y)}^z=(x^{{}^yz})^{y^z},
\end{array}\]
 \[ {\rm\bf lr3:} \quad
{({}^xy)}^{({}^{x^y}{z})} \ = \ {}^{(x^{{}^yz})}{(y^z)}.\]
\end{remark}

\begin{convention}
\label{conv:convention1} In this paper by "\emph{a solution}" we
mean "\emph{a non-degenerate braided set}" $(X,r)$, where $X$  is
a set of arbitrary cardinality. We shall also refer to it as
"\emph{a braided set}", keeping the convention that we consider
only non-degenerate braided sets. "\emph{An involutive solution}"
means "\emph{a non-degenerate symmetric set}". In most cases we
shall also assume that $r$ is 2-cancellative but this will be
indicated explicitly.
\end{convention}

\subsection{Quadratic sets and their algebraic objects}

Let $X$ be a non-empty set, and let $\textbf{k}$ be a field.
We denote by $\asX$,  and ${}_{gr}\asX$,
respectively the free monoid, and the free group
generated by $X,$ and by $\textbf{k}\asX$-
 the free associative $\textbf{k}$-algebra
generated by $X$. For a set $F
\subseteq \textbf{k}\asX$, $(F)$ denotes the two sided ideal
of $\textbf{k}\asX$, generated by $F$.

 For $m \geq 1$, the length of a monomial $u = x_1\cdots x_m \in X^m$
will be denoted by $|u|= m$.

As in our works \cite{GI04, GI04s,  GI11, GIM08, GIM11, GI18}, we
use the following.
\begin{definition}
\label{def:algobjects} To each quadratic set $(X,r)$ we associate
canonically algebraic objects generated by $X$ and with quadratic
relations $\Re =\Re(r)$ naturally determined as
\[
xy=y^{\prime} x^{\prime}\in \Re(r)\; \text{iff}\;
 r(x,y) = (y^{\prime}, x^{\prime})\; \text{and} \;
 (x,y) \neq (y^{\prime}, x^{\prime})\;\text{hold in}\;X \times X.
\]
 The monoid
$S =S(X, r) = \langle X ; \; \Re(r) \rangle$
 with a set of generators $X$ and a set of defining relations $ \Re(r)$ is
called \emph{the monoid associated with $(X, r)$}.
 The \emph{group $G=G(X, r)= G_X$ associated with} $(X, r)$ is
defined analogously.

For an arbitrary fixed field $\textbf{k}$,
\emph{the} \textbf{k}-\emph{algebra associated with} $(X ,r)$ is
defined as
\[\begin{array}{c}
A = A(\textbf{k},X,r) = \textbf{k}\langle X  \rangle /(\Re_0)
\simeq \textbf{k}\langle X ; \;\Re(r)
\rangle ,\;\;\text{where}\\
\Re_0 = \Re_0(r)
 = \{xy-y^{\prime}x^{\prime}\mid xy=y^{\prime}x^{\prime}\in \Re
(r) \}.
\end{array}
\]
Clearly, the quadratic algebra $A$ generated by $X$ and
 with defining relations  $\Re_0(r)$
is isomorphic to the monoid algebra $\textbf{k}S(X, r)$.
\end{definition}

\begin{definition}
\label{injectivedef}
We shall call a quadratic set  $(X,r)$ \emph{injective} if the set $X$ is embedded in $G(X,r)$.
\end{definition}

Recall that when $(X,r)$ is a braided set its monoid $S=S(X,r)$ is a graded braided monoid, \cite{GIM08}, and the group $G(X,r)$ is a braided group, \cite{LYZ}, see details
in section \ref{BraidedMonoidSec}.
Moreover, the associated quadratic algebra A = A(\textbf{k},X,r) is also called \emph{an Yang-Baxter algebra}, see \cite{Maninpreprint}.

\begin{remark}
\label{numberofrelationsrem}
\cite{GI11}, Proposition 2.3., If
$(X,r)$ is a nondegenerate and involutive quadratic set of finite order $|X| =n$
then the set $\Re(r)$
consists of precisely $\binom{n}{2}$ quadratic relations.
Clearly, in in this  case the associated algebra $A= A(\textbf{k}, X, r)$ satisfies
\[\dim A_2 = \binom{n+1}{2}.\]
Various equivalent conditions are given in Proposition \ref{prop:mainthlemma1}.
\end{remark}

\begin{remark}
\label{gradingRemark}
Suppose $(X,r)$ is a finite quadratic set.
Then $A$ is a quadratic algebra, generated by $X$ and
 with quadratic defining relations  $\Re(r).$
Clearly, $A$ is \emph{a connected
graded} $\textbf{k}$-algebra (naturally graded by length),
 $A=\bigoplus_{i\ge0}A_i$, where
$A_0=\textbf{k}$, $A$ is generated by $A_1=Span_{\textbf{k}}X,$ so
each graded component $A_i$ is finite dimensional.
Moreover, the associated monoid $S= S(X,r)$ \emph{is naturally graded by length}:
\[
 S = \bigsqcup_{m \geq 0}  S_{m}; \;\; \text{where}\;\;
S_0 = 1,\; S_1 = X,\; S_m = \{u \in S \mid\;   |u|= m \}, \; S_m.S_t \subseteq S_{m+t}. \]
In the sequel, by "\emph{a graded monoid}, $S$", we shall
mean that $S$ is generated  by $S_1=X$ and graded by length.
The grading of $S$ induces a canonical grading of its monoid algebra $\textbf{k}S(X, r).$
The isomorphism $A\cong \textbf{k}S(X, r)$ agrees with the canonical gradings, so there is an isomorphism of vector spaces $A_m \cong Span_{\textbf{k}}S_m$.
\end{remark}

\begin{remark}
\label{orbitsinG}
\cite{GI04s, GIM24} 
Let $(X,r)$ be a finite quadratic set. We define an equivalence relation on $X^2$ by
 \[xy \sim zt\; \text{if and only if}\; r^k(xy) = r^m(zt) \; \text{for some}\; m, k\geq 0.\]
 We define an $r$-orbit $\cO$ as an equivalence class with respect to this. Here $r^0$ denotes the identity map.

Clearly $X^2$ is a  disjoint union of $r$-orbits and we denote by 
$\cO(xy)$ the $r$-orbit containing $xy\in X^2$.  Also note that  the equality $xy=zt$ holds in $S=S(X,r)$ if and only if $xy\sim zt$, so each $r$-orbit corresponds to a distinct element of $S$ of grade 2.

 The relation
$\sim$ can be extended in a natural way to an equivalence relation on the free monoid $\asX$, such that
 $\sim$  is a congruence on $\asX$ which (by definition) satisfies:
 \[\text{if} \; u, v, a, b \in \asX, \text{ and}\;  u \sim v, \;\text{then}\;  aub \sim avb.\]
 It then follows from the definition that $S$ is isomorphic (as a monoid) to the quotient monoid modulo the congruence,
 $\asX /\sim$.
  In particular two words $w_1, w_2 \in \asX$ are equal in $S$ \emph{iff} $w_1$, and $w_2$ belong to the same congruence
  class in $\asX$.
\begin{enumerate}
\item
 There is a practical way given $w_1, w_2 \in \asX$, how to determine if $w_1 = w_2$ in $S$. Namely,
 $w_1, w_2 \in \asX$ are equal
 in $S$
  if they have equal lengths  $|w_1| = |w_2|\geq 2$ and there exists a monomial $w_0$ such that each  $w_i, i=1,2$ can
  be
  transformed to
  $w_0$ by a finite sequence of
  replacements (they are also called \emph{reductions} in the literature) each of the form
  \[a(xy)b \longrightarrow a(zt)b,\]
where $xy=zt$ is an equality in $S$, $xy>zt$  in $X^2$ and $a,b \in \asX$.  Clearly, every such replacement preserves monomial length, which therefore descends to $S(X,r)$. Furthermore, replacements coming from the defining relations are possible only
on monomials of length $\geq 2$, hence $X \subset S(X,r)$ is an inclusion.


\item When $r$ is involutive it is  convenient for each $m \geq 2$ to refer to the subgroup $D_m$ of the symmetric group $\Sym (X^m)$
generated concretely by the maps
\begin{equation}
\label{eq:r{ii+1}}
r^{ii+1}: X^m \longrightarrow X^m, \; r^{ii+1} =\id_{X^{i-1}}\times r\times \id_{X^{m-i-1}}, \; i = 1, \cdots, m-1.
\end{equation}
In the general case, we consider the free groups
\[\Dcal_m (r)=  \: _{\rm{gr}} \langle r^{i i+1}\mid i = 1, \cdots, m-1 \rangle,\]
where the
$r^{ii+1}$ are treated as abstract symbols, as well as various quotients depending on the further type of $r$ of interest.
 In particular, $\Dcal_2(r) = \langle r \rangle \subset \Sym (X^2)$
 is the cyclic group generated by $r$.
 It follows straightforwardly from part (1) that $w, w^{\prime} \in \langle X\rangle$ are equal as words in $S(X,r)$
  \emph{iff} they have the same length, say $m$, and belong to the same orbit of $\Dcal_m(r)$ in $X^m$. Clearly, in this case the equality $w=w^{\prime}$ holds in the group $G(X,r)$ and in the algebra $A(\textbf{k}, X, r)$.

An effective part of our combinatorial approach is the exploration of the actions of the groups
$\Dcal_2(r) = \langle r\rangle$ on $X^2$, the group $\Dcal_3 (r)=  \: _{\rm{gr}} \langle r^{12}, r^{23}\rangle$ on $X^3$,
and in particular, the properties of the corresponding orbits.  In the literature a $\Dcal_2(r)$-orbit $\cO$ in $X^2$ is often called
"\emph{an $r$-orbit}" and we shall use this terminology.

If  $r$ is involutive, the bijective maps $r^{12}$
and $r^{23}$ are involutive as well, so in this case $\Dcal_3(r)$ is
\emph{the infinite dihedral group},
\[
\Dcal_3(r)= \Dcal(r)=\: _{\rm{gr}} \langle r^{12}, r^{23}\mid (r^{12})^2= e,\quad  (r^{23})^2=e \rangle.
\]
\end{enumerate}
\end{remark}
\begin{remark}
 \label{dimA2rem}
 In notation and assumption as above, let $(X,r)$ be a finite quadratic set, $S=S(X,r)$ graded by length.
 Then the order of $S_2$ equals the number of $r$-orbits in $X^2$.
 \end{remark}
For positive integers $i<n$, the maps $r^{ii+1}: X^n \longrightarrow X^n$ are defined by (\ref{eq:r{ii+1}}).
Recall that the braid group $B_n$ is generated by elements $b_i,1 = i = n-1$, with defining relations
\begin{equation}
 \label{braidgrouprel}
 b_ib_j =b_jb_i, |i-j|> 1,\quad b_ib_{i+1}b_i = b_{i+1}b_ib_{i+1},
\end{equation}
 and the symmetric group $S_n$ is the quotient of $B_n$ by the relations $b_i^2 =1$.
 It is well known (and straightforward)  that for every $n \geq 3$ the following hold.
(i) The assignment $b_i \mapsto r^{ii+1}$ extends to a (left) action of $B_n$ on $X^n$ if and only if $(X,r)$ is a braided set.
(ii) The assignment $b_i \mapsto r^{ii+1}$ extends to an action of $S_n$ on $X^n$ if and only if $(X,r)$ is a symmetric set.

\section{Nondegenerate quadratic sets with 2-cancellation and their quadratic algebras}
\label{sec:quadratic}

\subsection{Basics on quadratic algebras}
 Our main reference for this subsection is \cite{PP}.

A quadratic  algebra is an associative graded algebra
 $A=\bigoplus_{i\ge 0}A_i$ over a ground field
 $\textbf{k}$  determined by a vector space of generators $V = A_1$ and a
subspace of homogeneous quadratic relations $R= R(A) \subset V
\otimes V.$ We assume that $A$ is finitely generated, so $\dim A_1 <
\infty$. Thus $ A=T(V)/( R)$ inherits its grading from the tensor
algebra $T(V)$. The Koszul dual algebra of $A$, denoted by $A^{!}$
is the quadratic algebra $T(V^{*})/( R^{\bot})$, see
\cite{Maninpaper, Maninpreprint}. The algebra  $A^{!}$ is also referred to as \emph{the quadratic dual algebra to a quadratic algebra} $A$, see \cite{PP}, p.6.

Following the classical tradition (and a recent trend), we take a
combinatorial approach to study $A$. The properties of $A$ will be
read off a presentation $A= \textbf{k} \langle X\rangle /(\Re)$,
where by convention $X$ is a fixed finite set of generators of
degree $1$, $|X|=n,$ $\textbf{k}\langle X\rangle$ is the unital free
associative algebra generated by $X$, and $(\Re)$ is the two-sided
ideal of relations, generated by a {\em finite} set $\Re$ of
homogeneous polynomials of degree two.


A quadratic algebra $A$ is \emph{a
PBW algebra} if there exists an enumeration of $X,$ $X= \{x_1,
\cdots x_n\}$ such that the quadratic relations $\Re$ form a
(noncommutative) Gr\"{o}bner basis with respect to the
degree-lexicographic order $<$ on $\langle X\rangle$ extending
$x_1 < x_2< \cdots <x_n$. In this case   the set of normal monomials
(mod $\Re$) forms a $\textbf{k}$-basis of $A$ called a \emph{PBW
basis}
 and $x_1,\cdots, x_n$ (taken exactly with this enumeration) are called \emph{ PBW-generators of $A$}.
 The notion of a \emph{PBW} algebra was introduced by Priddy, \cite{priddy}. His
 \emph{PBW basis}  is a generalization of the classical Poincar\'{e}-Birkhoff-Witt basis for the universal enveloping of a  finite dimensional Lie algebra.
 PBW algebras form an important class of Koszul algebras.
  The interested reader can find information on quadratic algebras and, in particular, on Koszul algebras and PBW algebras in
  \cite{PP}.

There are various equivalent definitions of a Koszul algebra, see  for example \cite{PP}, p. 19,
we recall one of them.
 A graded $\textbf{k}$-algebra $A$ is \emph{Koszul} if $A$ is quadratic and $\Ext^{*}_A(\textbf{k}, \textbf{k})\simeq A^{!}$.
It is known that if $(X,r)$ is a finite square-free braided set
with $r$ involutive then its quadratic algebra $A(\textbf{k}, X,r)$ is Koszul,
see \cite{GIVB}. 

Notice that the converse is not true- Example \ref{ex:Arne} gives a finite  nondegenerate square-free braided sets $(X,r)$ whose algebra  $A =A(\textbf{k}, X,r)$ is PBW (and therefore $A$ is Koszul), but $r$ is not involutive. 


Therefore in \cite[p 754] {GI21}, the phrase
"We shall prove that, conversely, if $(X, r)$ is a (general) square-free nondegenerate braided set
and its algebra  $A(\textbf{k}, X,r)$ is Koszul, then $r$ is involutive. " must be discarded  since it is wrong.

The following results can be used to test whether a quadratic algebra is Koszul.
\begin{fact}
\label{factKoszulvip}
\begin{enumerate}
\item[(1)] (Theorem of S. Priddy, \cite{priddy})  Every quadratic PBW algebra is Koszul.

\item[(2)] (\cite{PP}, Corollary 2.22.)
If $A$ is a quadratic Koszul algebra, with Koszul dual $A^{!}$, then their Hilbert series satisfy
\begin{equation}
\label{HilbKoszuleq}
H_A(z).H_{A^{!}}(-z) = 1.
\end{equation}
\end{enumerate}
Condition  (\ref{HilbKoszuleq}) is necessary but not sufficient for Koszulity of $A$, \cite{PP}.
\end{fact}

\subsection{Quadratic set with 2-cancellation and their quadratic algebras}
To proceed further, we require
some
cancellation conditions.

\begin{definition}\cite{GIM08}, Def.2.10.
\label{def:cancellsol} A quadratic set $(X,r)$ is  \emph{2-cancellative} if for
every positive integer $k$, less than the order of $r$, the
following two conditions hold:
\[
\begin{array}{ll}
r^k(x,y)= (x,z) \Longrightarrow y=z, \quad &
r^k(x,y)= (t,y) \Longrightarrow x=t.
\end{array}
\]
\end{definition}

The monoid $S=S(X,r)$  has
cancellation on monomials of
length 2 {\em  if and only if}  $r$ is
2-cancellative, see \cite[Prop.2.11(1)]{GIM08}.
Moreover, every injective quadratic set $(X,r)$ (see Definition \ref{injectivedef}) is 2-cancellative.
Note that if $x,y,z \in X, y\neq z$ each of the equalities $r^k(x,y) = (x,z)$, or $r^k(y, x) = (z,x) $ implies
$y = z$ in $G(X,r).$

\begin{remark}
\label{rem:nondeg_invol_implies2-cancel}
\begin{enumerate}
\item
 Every nondegenerate involutive quadratic set  $(X,r)$ is 2-cancellative, see \cite[Corollary 2.13]{GIM08}. 
Recall that when $X$ is a finite (nondegenerate) symmetric set the monoid $S=S(X,r)$ is embedded in the group $G(X,r)$ and therefore $S$ is a monoid with cancellation.
\item A non-degenerate braided set $(X,r)$ may fail to be 2-cancellative, see Example \ref{nocancellative_ex}, and Example \ref{ex:Arne}.
\item There exist noninvolutive non-degenerate braided sets $(X,r)$, where $r$ is 2-cancellative but
the corresponding monoid $S(X,r)$  fails to be 3-cancellative.
\end{enumerate}
\end{remark}

We shall prove that if $(X,r)$ is a finite square-free braided set, then the monoid $S(X,r)$ is cancellative \emph{iff} $r$ is involutive,
see Proposition \ref{prop3}.

\begin{notation}
\label{notation:Delta}
Denote  by $\Delta_m$ the diagonal of $X^{\times m}, \; m\geq 2$:
\[\Delta_m:=\diag(X^m) = \{x^m \mid x \in X\}.\]
One has $\Delta_3 = (\Delta_2 \times X) \bigcap (X \times \Delta_2)$.
\end{notation}

\begin{notation}
\label{def:fixedpts}
Suppose $(X,r)$ is a quadratic set. The element $xy\in X^2$ is \emph{an $r$-fixed point} if $r(x,y) =(x,y)$. \emph{The set of $r$-fixed points} in $X^2$ will be denoted by $\Fcal (X,r)$,
\begin{equation}
\label{eq:fixedpts}
\Fcal (X,r) = \{xy \in X^2\mid r(x,y) = (x,y)\}.
\end{equation}
\end{notation}
\begin{example}
\label{nocancellative_ex}
\begin{enumerate}
\item[(1)] \cite[Example 2.14]{GIM08}.
Let $X=\{x, y, z\},$ and let $\rho=(x\;y\;z)$ be a cycle of length three in $\Sym (X)$. Define $r(a,b):=(\rho(b), a),$ so $X^2$ splits into two $r$-orbits:
\[
\begin{array}{l}
(x, x)\longrightarrow^r (y, x)\longrightarrow^r (y,
y)\longrightarrow^r (z, y)\longrightarrow^r (z,
z)\longrightarrow^r (x, z)\longrightarrow^r (x, x),\\
(x, y)\longrightarrow^r (z, x)\longrightarrow^r (y,
z)\longrightarrow^r (x, y).
\end{array}
\]
This is 
permutation solution and it is easy to check that $(X, r)$ is a non-degenerate braided set.

Note that $r(a,b) \neq (a,b), \forall a,b \in X,$
so $\Fcal (X,r) = \emptyset,$ and that
$r$ is not
2-cancellative, since $xx=yx$ in $S$.
Moreover, $(X,r)$ is not injective, since in the group $G(X,r)$ all generators are equal: $x=y=z.$
\item[(2)]  \cite[Example 2.17]{GIM08} Let $X=\{x,y,z\},$ and let $r$ be the map
\[
\begin{array}{l}
(x, y)\longrightarrow^r (x, z)\longrightarrow^r (y,
z)\longrightarrow^r (y, y)\longrightarrow^r (x, y),\\
(x, x)\longrightarrow^r (z, z)\longrightarrow^r (y,
x)\longrightarrow^r (z, y)\longrightarrow^r (x, x),\quad r(z,x) = (z,x).
\end{array}
\]
The bijective map $r$ is nondegenerate, $\Lcal_x=\Lcal_y=\Lcal_z=(x\; z \; y);  \Rcal_x=\Rcal_y=\Rcal_z=(x\; z),$
$r$ is not
2-cancellative, conditions {\bf l1} and {\bf r1} are satisfied, but $(X,r)$ is not a braided set.
Here $\Fcal (X,r) = \{zx\}.$
\end{enumerate}
\end{example}

   \begin{lemma}
   \label{lem:2cancellativeprop}
   Suppose $(X,r)$ is a  nondegenerate quadratic set (possibly infinite). Then
   \begin{enumerate}   \item
   \label{lem:2cancellativeprop1}
   If $X$ has finite order, $|X|= n$, then
   \begin{equation}
\label{eq:fixedpts1}
0 \leq |\Fcal (X,r)|\leq |X|= n.
\end{equation}
   \item
   \label{lem:2cancellativeprop2}
   If $(X,r)$ is square-free, then $\Fcal(X,r) = \Delta_2$, the diagonal of $X^2$.
   In particular, if $X$ has finite order,  then $|\Fcal(X,r)|=|X|= n$.
   \item
   \label{lem:2cancellativeprop3}
   Suppose $(X,r)$ is nondegenerate and 2-cancellative. Then the following conditions hold.
   \begin{enumerate}
   \item
For every $y \in X$ there exists a unique $x \in X$ such that $r(x,y) = (x,y).$
In other words there exist a bijective map $t: X \longrightarrow X$ such that $r(t(y), y) = (t(y), y),$ for every $y \in X.$
\item For every $x \in X$ there exists a unique $y \in X$ such that $r(x,y) = (x,y).$
\item
\label{lem:2cancellativeprop3c}
If $X$ is finite, $X = \{x_1, \cdots, x_n\}$, then
\begin{equation}
\label{eq:fixedwords}
\Fcal = \Fcal (X,r) = \{xy \in X^2\mid r(x,y) = (x,y)\}= \{x_1y_1, \cdots , x_ny_n\},
\end{equation}
where $y_i\in X,$ is the unique element with $r(x_i, y_i) =(x_i, y_i), 1 \leq i \leq n$. In particular, $|\Fcal|= |X|= n.$
\end{enumerate}
\end{enumerate}
 \end{lemma}
\begin{proof}
(1).
The equality $r(x,y) = ({}^xy, x^y)$ implies
\[(x,y) \in \Fcal \; \text{if and only if} \; x = {}^xy,\;\text{and}\;
  x^y= y. \]
  It follows from the non-degeneracy that for each $y \in X$ there exists unique $x \in X,$ with
$x^y = y.$  In general, it is possible that ${}^xy \neq x$, and in this case $r(x,y) = ({}^xy, y)\neq (x,y)$.
This implies (\ref{eq:fixedpts1}).

(2). Suppose $(X,r)$ is square-free, then, by definition $\Delta_2  \subseteq \Fcal(X,r)$.
Assume $xy \in\Fcal(X,r)$.
Then ${}^xy =x = {}^xx$, hence by the non-degeneracy, $y=x$. This gives $\Fcal(X,r)\subseteq \Delta_2,$ and therefore $\Fcal(X,r)= \Delta_2.$
Moreover, $|\Fcal(X,r)|= |X|$, whenever $X$ is a finite set.

(3). Assume $(X,r)$ is 2-cancellative and nondegenerate.
Suppose $y \in X$, then, by the nondegeneracy there exists unique $x \in X$ such that $x^y = y.$  Consider the equality
$r(x,y) = ({}^xy, x^y)= ({}^xy, y)$ then by the $2$-cancellation law, one has ${}^xy =x$, hence $r(x, y)= (x,y)$, as desired.
Assume now that $r(z,y) = (z,y)$ for some $z \in X$. Then $r(z,y) = ({}^zy, z^y)= (z, y)$ implies $z^y = y = x^y$, which by the non-degeneracy gives $z = x.$
This proves part (3.a). Part (3.b) is analogous. Part (3.c) follows straightforwardly from (3.a) and (3.b).
\end{proof}

\begin{corollary}
 \label{sqfreerem}
 A  nondegenerate quadratic set is square-free if and only if for every pair $x,y \in X$ one has
 \[ r(x,y) = (x,y)\Longleftrightarrow  x=y .\]
 \end{corollary}

We replace the whole text of  \cite[Remark 3.9, p 757-758]{GI21} with (a new) subsection 3.3. below. 
\subsection{The set of reduced relations $R(r)$ of $A=A(\textbf{k}, X, r)$ and the relations of the Koszul dual algebra $A^{!}$}
\label{remark:Groebner1}

In this subsection 
 $(X,r)$ is a 
 nondegenerate quadratic set of finite order $|X|= n$. Initially we do not impose any restrictions such as "$(X,r)$ is 2-cancellative", or "$(X,r)$ is square-free". 
We shall apply the theory of noncommutative Gr\"{o}bner bases. We enumerate $X$, as  $X = \{x_1 < x_2< \cdots < x_n\}$ and consider the induced degree-lexicographic order $< $ on $\asX$. Let $\cO_j, 1\leq j \leq q,$ be the set of all nontrivial $r$-orbits in $X^2$. Each $r$-orbit $\cO_j, 1\leq j \leq q,$ has length $l_j = |\cO_j|\geq 2$
and
contains unique monomial $x_1^jy_1^j \in \cO_j$, which is minimal (in  $\cO_j$) with respect to the order $<$ on $\langle X\rangle$. 
Suppose 
\begin{equation}
\label{eq:Oj}
\cO_j = \{x_1^jy_1^j, x_2^jy_2^j, \cdots, x_{l_j}^jy_{l_j}^j\}, \; \text{where}\; x_k^jy_k^j  > x_1^jy_1^j, \; 2 \leq k \leq l_j.
\end{equation}
Then the subset of defining relations determined by $\cO_j$, namely \[\{xy -r(xy) | xy \in \cO_j\},\] reduces to precisely $l_j-1$ relations $f_{k}^j$, where 
 \[f_{k}^j=x_k^jy_k^j - x_1^jy_1^j, \; 2 \leq k \leq l_j, \; \text{with highest monomial}\; HM(f_{k}^j) = x_k^jy_k^j. \]
 Note that all  highest monomials $HM(f_{k}^j) = x_k^jy_k^j, \; 2 \leq k \leq l_j,\;  1 \leq j \leq q$, are pairwise distinct.

\emph{The set of reduced relations} $\textbf{R}(r)$ is defined as
 \begin{equation}
    \label{eq:reducedrel}
        \begin{array}{ll}
 \textbf{R}(r) &= \bigcup _{1 \leq j \leq q}\{f_{k}^j= x_k^jy_k^j - x_1^jy_1^j\mid 2 \leq k \leq l_j\}, \\&\\
 |\textbf{R}(r)|&= \sum_{1\leq j \leq q} (l_j -1) = (\sum_{1\leq j \leq q} l_j) -q \geq q.
     \end{array}
    \end{equation}
    We shall see below that the exact number of the reduced relations is \[|\textbf{R}(r)|= n^2-|\cal F|-q,\]
     where $\Fcal$ is the set of $r$-fixed points, and $q$ is the number of nontrivial $r$-orbits in $X^2$. Recall that each of the conditions (a) $(X,r)$ is 2-cancellative,
    or (b) $(X,r)$ is square-free implies $|\Fcal|=n$.

    Note that there is an equality of sets $\Re_0(r)= \textbf{R}(r)$ if and only if $r$ is involutive.
 Moreover, the two sets $\Re_0(r)$ and $\textbf{R}(r)$ may be different (in general), but  they always generate the same two-sided ideal 
 \[I = (\Re_0(r))=(\textbf{R}(r)) \; \text{of}\; \textbf{k}\langle X\rangle.\]
Hence the algebra $A= A(\textbf{k}, X, r)\cong \textbf{k}\langle X\rangle/(\Re_0(r))$ has a finite presentation as
 \[A = \textbf{k}\langle X\rangle/(\textbf{R}(r)), \;\text{with}\; \textbf{R}(r)\;\text{as in}\; (\ref{eq:reducedrel}).  \]
 The set of reduced relations $\textbf{R}(r)$ is exactly the quadratic part of the  \emph{reduced Gr\"{o}bner basis of} $I$, denoted $\textbf{GR}(I)$ (with respect to the degree-lexicographic
 order $< $ on $\asX$). Moreover, $\textbf{R}(r)$ is a linearly independent set, since  all relations $f_{k}^j, \; 1 \leq j \leq q, \; 2 \leq k \leq l_j,$ have pairwise distinct highest monomials. The graded component $I_2$ of the graded ideal $I$ is spanned by  $\textbf{R}(r)$, hence
 \begin{equation}
 \label{eq:dimI21} 
 \dim I_2 = |\textbf{R}(r)|.
 \end{equation}
 In general,  $\textbf{R}(r) \subset \textbf{GR}(I)$, and the reduced Gr\"{o}bner basis $\textbf{GR}(I)$ may be infinite.

 Denote by $\Ncal$ the set of all normal words in $\asX$ modulo $I$, (it is also the set of all normal words modulo $\textbf{GR}(I)$).  The normal words of length $m$ in $X^m$ will be denoted by $\Ncal_m$, $m \geq 1$.
 It is known that 
 \[
 \textbf{k}\asX = \textbf{k}\Ncal \oplus I 
 \]
is a direct sum of vector spaces.
 By the Diamond Lemma \cite{Bergman}, $\Ncal$ projects to a basis of $A$ and there is an isomorphism of vector spaces 
 $A \cong \textbf{k}\Ncal$.
 Similarly,  $\Ncal_2$ projects to a basis of $A_2.$
 This implies the following equalities of vector spaces
 \begin{equation}
 \label{eq:directsum2}
 \textbf{k}X^2 = \textbf{k}\Ncal_2 \oplus I_2, \; A_2 \cong \textbf{k}\Ncal_2.
 \end{equation}
 Moreover, $A_2 \simeq \textbf{k} S_2$ and two words of length $2$ in $X^2$ are distinct elements of $S_2$ \emph{iff} they belong to distinct $r$-orbits in $X^2$. The set $X^2$ splits into $|\cal F|+q$ distinct $r$-orbits, namely $\Fcal,\; \cO_j, 1 \leq j \leq q.$ 
 It follows that
 \begin{equation}\label{eq:dimA2}
   \dim A_2 = |S_2| = |\Fcal|+q.
 \end{equation}
 This together with (\ref{eq:dimI21}) and (\ref{eq:directsum2}) imply
 \begin{equation}\label{eq:dimI22}
 |\textbf{R}(r)|=  \dim I_2 = n^2- \dim A_2 = n^2- |\Fcal|-q.
 \end{equation}


  One has
\[
\Ncal_2 = \Fcal \cup \{x_1^jy_1^j\mid 1 \leq j \leq q\}, \quad |\Ncal_2|= \dim A_2= |\Fcal|+q.
\]
For every integer $m \geq 2$ denote by $\Ncal_m$ the set of all monomials in $X^m$ which are normal modulo $I=(\textbf{R}(r))$.
Then
\[
\begin{array}{ll}
\dim_{\textbf{k}} A_m &= |\Ncal_m| = |S_m| \\
                      &= \text{the number of all disjoint}\; \Dcal_m(r)-\text{orbits in}\;  X^m.
                      \end{array}
 \]

 The general definition of the \emph{Koszul dual algebra} $A^{!}$ of a quadratic algebra
  $A$ can be found in \cite{Maninpaper, Maninpreprint}, see also see \cite{PP}, p.6. where $A^{!}$ is also referred to as \emph{the quadratic dual algebra to a quadratic algebra} $A$.
       Here we need a detailed description of the Koszul dual algebra $A^{!}$ of  $A = A(\mathbf{k}, X, r)$ in terms of generators and relations. 

Suppose $(X,r)$ is a 
 nondegenerate quadratic set, where  $X= \{x_1, \cdots, x_n\}$. The Koszul dual algebra $A^{!}$ of  $A = A(\mathbf{k}, X, r)$ is presented in terms of generators and relations 
as
\[
A^{!}  = \textbf{k}\langle \xi_i, \cdots \xi_n\rangle/(\textbf{R}^{\bot}),
\]
 where $\xi_i$, $1\leq i \leq n$, are the linear functions defined as $\xi_i (x_j) = \delta_{ij}$, the
Kronecker delta, and 
$\textbf{R}^{\bot}$ consists of $n+q$ quadratic relations in the $\xi_i$'s  whose explicit form is given in  (\ref{eq:Rbot}).  
For each $y \in X$, we denote by $\xi_y$ the linear function determined as $\xi_y (z) = 1$ \emph{iff} $y=z$, and $\xi_y (z) = 0$, else.
Each nontrivial orbit $\cO_j= \{x_1^jy_1^j, \cdots, x_{l_j}^jy_{l_j}^j \}$, $1 \leq j \leq q$, determines exactly one relation $\textbf{R}_j^{\bot}\in \textbf{R}^{\bot}$: 
\begin{equation}
 \label{eq:Rjbot}
 \textbf{R}_j^{\bot} = \xi_{x_1^j}\xi_{y_1^j}+ \xi_{x_2^j}\xi_{y_2^j}+\cdots +\xi_{x_{l_j}^j}\xi_{y_{l_j}^j}.
 \end{equation} 
 The set of fixed points $\Fcal$ (one element orbits) determines $n$ monomial relations in $\textbf{R}^{\bot}$:
\begin{equation}
 \label{eq:RFbot}
 (\textbf{R}_{\Fcal})^{\bot} = \{\xi_{x_i}\xi_{y_i} \mid x_iy_i \in \Fcal(X,r)\}, \quad |(\textbf{R}_{\Fcal})^{\bot}|= n.
 \end{equation}
 
 The set $\textbf{R}^{\bot}$ of quadratic relations of the \emph{Koszul dual algebra} $A^{!}$ is determined as 
 \begin{equation}
 \label{eq:Rbot}
 \textbf{R}^{\bot}= \{\textbf{R}_j^{\bot}\mid 1 \leq j \leq q\}\cup (\textbf{R}_{\Fcal})^{\bot}.
 \end{equation}
 
There are equalities
\begin{equation}
    \label{dimA2eq}
    \dim A_2 + \dim A_2^{!} = n^2, \quad \quad \dim A_2= |\Fcal| +q, \quad\quad \dim A_2^{!} = n^2-|\Fcal| -q .
    \end{equation}
        Suppose, now that $(X,r)$ is a nondegenerate square-free
quadratic set (we do not assume 2-cancellativity).
Then $\Fcal(X,r) = \Delta_2$, and
$|\Fcal(X,r)|= n$.
Moreover,
\[(\textbf{R}_{\Fcal})^{\bot}= \{\xi_i\xi_i \mid 1 \leq i \leq n \}.\]

\begin{remark} (Erratum)
Notice that the description of the relations of the \emph{Koszul dual algebra} $A^{!}$ 
given above in (\ref{eq:Rjbot}), (\ref{eq:RFbot}), and  (\ref{eq:Rbot}) is precise for any kind of quadratic sets $(X,r)$, including the involutive ones.
In contrast, the relations $R^{\bot}$ of the Koszul dual
$A^{!}$  described on \cite[p. 758]{GI21} are correct  only if  $(X,r)$ is involutive, but that description is not correct for noninvolutive $r$.  Moreover, equalities (3.5) on page 758,  \cite{GI21} are not precise and should be corrected. The corrected equalities are given now in (\ref{dimA2eq}).
\end{remark}
\subsection{More results on nondegenerate quadratic sets }
\label{subsec:more_results}
\begin{proposition}
  \label{prop:mainthlemma1}
   Let $(X,r)$ be a
   nondegenerate quadratic set of finite order $|X|=n\geq 3$, and
   let $A=  A(\textbf{k}, X, r)$
 be its associated quadratic $\textbf{k}$-algebra, naturally graded by length.
Suppose that $X^2$ contains exactly $q$ nontrivial $r$-orbits,  $\cO_1, \cdots, \cO_q$, $|\cO_j|= l_j \geq 2$, $1 \leq j \leq q$.

 \begin{enumerate}
\item
\label{prop:mainthlemma1_1}
If $(X,r)$ is
   2-cancellative then the following conditions hold.
\begin{enumerate}
\item[(1.a)]
\label{prop:mainthlemma1_1a}
$X^2$ has exactly $n$ one-element $r$-orbits, that is $|\Fcal(X,r)|=n$.
  \item[(1.b)]
    \label{prop:mainthlemma1_1b}
    The following inequalities hold:
   \begin{equation}
\label{eq:dim1}
2n-1\leq \dim_{\textbf{k}} A_2 = n + q \leq \binom{n+1}{2}, \quad\text{and}\quad n-1 \leq q\leq \binom{n}{2},
\end{equation}
where the upper
bounds for $\dim_{\textbf{k}} A_2$, and for $q$ are exact,
 for all $n \geq 3$,
 and the lower bounds are exact, whenever $n = p>2$ is a prime number.
\end{enumerate}
\item
\label{prop:mainthlemma1_2}
Suppose $|\Fcal(X,r)|=n$. Then the following conditions are equivalent
\begin{enumerate}
\item[(i)]
the map $r$  is involutive;
\item[(ii)] $\dim_{\textbf{\textbf{k}}} A_2 = \binom{n+1}{2}$;
\item[(iii)]
 $q = \binom{n}{2}$;
\item[(iv)] $\dim I_2 = |\Re (r)|= \binom{n}{2}$;
\item[(v)] $\dim_{\textbf{\textbf{k}}} A^{!}_2 = \binom{n}{2}$.
 \end{enumerate}
Each of these conditions implies that $(X,r)$ is 2-cancellative.
\item
\label{prop:mainthlemma1_3}
For every integer $m \geq 2$,  $\dim_{\textbf{k}} A_m =|S_m|=$
the number of all disjoint $D_m(r)$-orbits in $X^m$.
\end{enumerate}
\end{proposition}
\begin{proof}

(\ref{prop:mainthlemma1_1}).
Suppose $(X,r)$ is
   2-cancellative.
Part (1.a) follows from Lemma \ref{lem:2cancellativeprop}.

(1.b).  The action of the cyclic group $\langle r\rangle$ on $X^2$ splits $X^2$ into disjoint $r$-orbits $\cO$.
We shall analyze the possible number of orbits and their lengths.
It is clear that the map $r$ is involutive \emph{iff} every nontrivial orbit $\cO(xy)$ has precisely two elements.

Let $X = \{x_1, \cdots , x_n\}$ be an arbitrary enumeration on
$X$. By Lemma \ref{lem:2cancellativeprop} (c) $X^2$ contains exactly
$n$ elements fixed under $r$, see
(\ref{eq:fixedwords}), so there are exactly $n$ one-element
$r$-orbits, $\cO (x_iy_i) = \{x_iy_i\},   1 \leq i \leq n$.

By assumption the complement $X^2 \setminus  \Fcal (X,r)$ splits into $q$ disjoint nontrivial orbits:
\[X^2 \setminus \Fcal (X,r)= \bigcup_ {1 \leq j \leq q}\cO_j, \quad \text{where}\quad \mid \cO_j\mid =l_j \geq 2.\]
Then
\begin{equation}
\label{eq:cardorbits1}
|X^2 \setminus \Fcal (X,r)|= n^2-n = \sum _{1 \leq j \leq q} l_j.
\end{equation}
 By the 2-cancellativity of $(X,r)$, a nontrivial orbit $\cO_j$
 does not contain distinct monomials of the shape $xu,xv, u \neq v$, or $xu, yu, x \neq y$, hence
\begin{equation}
\label{eq:cardorbits2}
 2 \leq |\cO_j| = l_j \leq n,  \;\; \;\forall \; 1 \leq j \leq q.
\end{equation}
Therefore
 \[
2q  \leq \sum _{1 \leq j \leq q} l_j = n^2-n \leq nq.
\]
But $2q \leq n^2-n$ is equivalent with $q \leq n(n-1)/2 = \binom{n}{2},$ moreover
$n^2-n \leq nq$ implies $n-1 \leq q$.  This proves the right-hand side inequalities in
\ref{eq:dim1}.

Recall that $a, b \in \cO_j$ if and only if $a = b$ in the algebra $A$, or equivalently in the monoid $S.$
 We argue with the number of distinct words of length 2 in $S$ which is the same as the number of all $r$-orbits
 \[|S_2|= n +q.\]
 One has $A_2 = Span _{\textbf{k}}S_2$, and since
 every set of pairwise distinct words in $S_2$ is linearly independent, we yield
\begin{equation}
\label{dimeq1}
2n-1\leq \dim_{\textbf{k}} A_2 = |S_2| = n + q \leq n +\binom{n}{2} = \binom{n+1}{2},
\end{equation}
which proves the left-hand side inequalities in
(\ref{eq:dim1}). (One may also use the theory of Gr\"{o}bner basis for a detailed proof, see Subsection \ref{remark:Groebner1}).
We shall discuss the exactness of the bounds after the proof of part (\ref{prop:mainthlemma1_2}).

(\ref{prop:mainthlemma1_2}). The equality $\dim_{\textbf{k}} A_2=n+q$ implies the equivalence of (ii) and (iii).
The equivalence of (ii) and (v) follows from (\ref{dimA2eq}).

Each $w \in X^2\setminus \Fcal$ belongs to a nontrivial orbit $\cO(w)$, and $|X^2\setminus \Fcal|=n(n-1)$. It is clear that $q =\binom{n}{2}$ if and only if
each nontrivial orbit has exactly two elements, which is equivalent to $r^2 = 1$, thus (i) and (iii) are equivalent.
One has
\[\textbf{k}\asX _2 = \textbf{k}X^2=I_2\oplus A_2, \quad \dim (\textbf{k}\asX) _2 =\dim \textbf{k}X^2= \dim I_2 + \dim A_2,\] so
\[n^2 - \dim A_2 = \dim I_2 = \sum _{1 \leq j \leq q} (l_j -1)= (\sum _{1 \leq j \leq q} l_j)-q \]
and $\dim I_2= \binom{n}{2}$
if and only if $\dim A_2= \binom{n+1}{2}$, which gives the equivalence of  (ii) and (iv).

It follows from \cite[Corollary 2.13]{GIM08} that  every nondegenerate involutive quadratic set  $(X,r)$ is 2-cancellative. Therefore
each of the equivalent conditions  (i) through (v) implies  "$(X,r)$ is 2-cancellative".

Part (\ref{prop:mainthlemma1_2})  implies that the upper bounds in (\ref{eq:dim1}) are exact.
If $n=p>2$ is a prime number, then by Corollary \ref{Mcorollary} every square-free self distributive solution, $(X,r)$ corresponding to a dihedral quandle of prime order
$|X|= p>2$
satisfies what we call "the minimality condition" : $\dim_{\textbf{\textbf{k}}} A_2 =2n-1$, which is equivalent to $q = n-1$, see  Definition  \ref{minimality_def}.
This proves the exactness of the lower bound, whenever $n=p>2$ is a prime number.

(\ref{prop:mainthlemma1_3}).  The distinct elements of the monoid $S=S(X,r),$ form a $\textbf{k}$-basis of the
monoid algebra $\textbf{k}S \simeq A(\textbf{k}, X,r)$. In
particular $\dim A_m$ equals the number of distinct monomials of
length $m$ in $S$ which is exactly the number of $\Dcal_m(r)$-orbits in
$X^m$, see Remark  \ref{remark:Groebner1}.
\end{proof}

\begin{corollary}
  \label{cor:mainthlemma1}
   Let $(X,r)$ be a
   nondegenerate quadratic set of finite order $|X|=n\geq 3$. 
   \begin{enumerate}
   \item
Each of the conditions
(a) $(X,r)$ is 2-cancellative, or (b) $(X,r)$ is square-free implies the equality $|\Fcal(X,r)|= n.$
   \item
   Suppose some of the following two conditions holds:
   (a) $|\Fcal(X,r)|=n$; (b)  $(X,r)$ is a square-free quadratic set. Then
the map $r$  is involutive if and only if
$\dim_{\textbf{\textbf{k}}} A_2 = \binom{n+1}{2}$.
In this case $(X,r)$ is $2$-cancellative.
\end{enumerate}
\end{corollary}
\begin{remark} (Erratum)
\label{rmk:prop:Koszul}
Observe that the original paper \cite{GI21} contains Proposition 3.12 whose parts (1) and (2) \emph{are not correct.} 
In the present version we erase  Proposition 3.12 (1) and (2) and move part (3) (which is correct) in the text of the present Corollary  
\ref{cor:mainthlemma1}. 
There are three results in \cite{GI21} which involve Proposition 3.12 in their statements, and proofs. These are Theorem 3.16, on p. 763, Theorem 5.5. on p. 777, 
and Corollary 5.6. on p. 776. Here we propose corrected (and in some cases improved) versions of these results and give (different) proofs which do not involve the wrong Proposition 3.12. 
\end{remark}

We thank Magdalena Wiertel  and Arne van Antwerpen who communicated the author the following example which demonstrated that \cite[Proposition 3.12, parts (1) and (2)]{GI21} is not true.

\begin{example} (M. Wiertel and A. van Antwerpen, private communication)
\label{ex:Arne}
Consider the braided set $(X, r)$, where $X= \{x_1, x_2, x_3\}$ and the map $r$ is defined as
\[
\begin{array}{c}
x_1x_3 \longrightarrow x_3x_2 \longrightarrow x_2x_3\longrightarrow x_3x_1 \longrightarrow x_1x_3,\quad \quad x_1x_2\longleftrightarrow x_2x_1 \\
  x_1x_1\longrightarrow x_1x_1, \quad  x_2x_2\longrightarrow x_2x_2,\quad  x_3x_3\longrightarrow x_3x_3.
\end{array}
\]
There are $2$ nontrivial $r$-orbits in $X^2$. 
The set $\textbf{R}(r)$ of reduced relations is:
\[
\textbf{R}(r) = \{x_3x_2- x_1x_3, x_3x_1- x_1x_3, x_2x_3- x_1x_3, x_2x_1 - x_1x_2 \}.
\]  
It is easy to check that the set $\textbf{R}(r)$ is a  Gr\"{o}bner basis of the ideal $I= (\textbf{R}(r))$ w.r.t. deg-lex ordering on $\langle X \rangle$, so the Yang-Baxter algebra $A= A(\textbf{k}, X, r)$ is PBW, and hence $A$ is a Koszul algebra, but $r$ is not involutive.

Clearly, $(X,r)$ is not 2-cancellative, since $r^2(x_1x_3) = x_2x_3$, and $x_2x_3= x_1x_3$ in $S$.
Here $\dim A_2= 5$.
The set of normal words is
\[
\Ncal = \{x_1^{\alpha} x_2^{\beta}; \quad x_1^{\alpha} x_3^{\beta} \mid  \alpha, \beta \in \mathbb{N}\cup \{0\}\}.
\]
The Gelfand-Kirillov dimension of $A$ is $\gkdim A= 2 < |X|= 3,$ and $A$ has infinite global dimension.
\end{example}

The following is an improved version of \cite[Corollary 3.13]{GI21}.
\begin{corollary}
  \label{cor:implications}
    Let $(X,r)$ be a
   nondegenerate quadratic set of finite order $|X|=n$, and
   let $A=  A(\textbf{k}, X, r)= \textbf{k}\asX/(\Re(r))$, let $I = (\Re(r))$ be the corresponding ideal of $\textbf{k}\asX.$
    Consider the following conditions
    \begin{enumerate}
    \item
    \label{cor:implic0}
    $(X,r)$  is square-free.
    \item
    \label{cor:implic1}
    $(X,r)$  is involutive.
    \item
    \label{cor:implic2}
    $(X,r)$  is 2-cancellative.
    \item
    \label{cor:implic3}
    The set of fixed points $\Fcal(X,r)$ has cardinality $n$.
    \item
    \label{cor:implic4}
    The number $q$ of nontrivial $r$-orbits in $X^2$ is $q = \binom{n}{2}$.
    \item
    \label{cor:implic5}
    $\dim A_2 = \binom{n+1}{2}$.
    \item
    \label{cor:implic6}
    $\dim I_2 = \binom{n}{2}$.
    \item
    \label{cor:implic6a}
    The algebra $A$ has exactly $\binom{n}{2}$ defining relations, $|\Re(r)| = \binom{n}{2}$.
    \item
    \label{cor:implic7}
    $\dim A^{!}_2 = \binom{n}{2}$.
    
        \end{enumerate}
        The following implications are in force.
        \[
        \begin{array}{l}
        (\ref{cor:implic0}) \Rightarrow  (\ref{cor:implic3})\\
        (\ref{cor:implic1})  \Rightarrow (\ref{cor:implic2}),\quad (\ref{cor:implic3}), \quad (\ref{cor:implic4}),\quad (\ref{cor:implic5}),\quad
        (\ref{cor:implic6}), \quad (\ref{cor:implic6a}),  \quad (\ref{cor:implic7}).\\
        (\ref{cor:implic2})  \Rightarrow (\ref{cor:implic3})\quad \text{and}\quad n-1 \leq q \leq  \binom{n}{2}, \quad 2n-1 \leq \dim A_2\leq \binom{n+1}{2}.\\
        \text{Assume\quad (\ref{cor:implic2}), or (\ref{cor:implic3}). \quad Then}\\
        (\ref{cor:implic1}) \Leftrightarrow (\ref{cor:implic4}) \Leftrightarrow (\ref{cor:implic5})\Leftrightarrow
        (\ref{cor:implic6})\Leftrightarrow (\ref{cor:implic6a}) \Leftrightarrow (\ref{cor:implic7}).\\
        \end{array}
        \]
\end{corollary}

\begin{lemma}
\label{lem:n3}
Suppose $X = \{x_1, x_2, x_2 \}$, and $(X,r)$ is a nondegenerate, 2-cancellative, square-free quadratic set.
Then (up to isomorphism) there are exactly two non-isomorphic
quadratic algebras $A^{(i)}$ corresponding to quadratic sets $(X,
r_{i}), i = 1,2$ which satisfy the hypothesis. These are
\[\begin{array}{l}
A^{(1)}= \textbf{k}\langle X : x_3x_2 - x_2x_3,  \;x_3x_1 - x_1x_3, \;x_2x_1 - x_1x_2\rangle\\
A^{(2)}= \textbf{k}\langle X : x_3x_2 - x_1x_3,  \;x_3x_1 -
x_2x_3, \;x_2x_1 - x_1x_2\rangle.
\end{array}\]
The algebras $A^{(1)}$ and $A^{(2)}$ are \emph{PBW} algebras with
$\gkdim A =3$ (in fact these are binomial skew-polynomial rings).
The quadratic set $(X, r_{1})$ is the trivial solution of the YBE.
$(X, r_{2})$ is the unique (up to isomorphism) non-trivial
square-free solution of YBE of order $|X|=3$.
\end{lemma}

\begin{question}
 Let $(X,r)$ be a
   2-cancellative
   nondegenerate quadratic set of finite order $|X|=n$.
  Suppose $A=  A(\textbf{k}, X, r)$ is \emph{PBW}. 
  \begin{enumerate}
  \item Is it true that if ${X,r}$ is square-free then
  $A$ has polynomial growth?
\item Is it true that if $(X,r)$ is square-free then $A$ has finite global dimension?
\end{enumerate}
This is so  for $|X|=3,$ see Lemma \ref{lem:n3}.
An affirmative answer would imply that $(X,r)$ is a solution of YBE, and $A$ satisfies all conditions (1) through (10) in Theorem
\ref{thm:new_theorem}. We do not know a counterexample.
\end{question}

In \cite{GI12} we study the close relation between square-free nondegenerate symmetric sets $(X,r)$ and a class of Artin-Schelter regular algebras.
Our result \cite{GI12}, Theorem 2, investigates \emph{quantum binomial quadratic sets} $(X,r)$, that is square-free nondegenerate involutive quadratic set,
in terms of various algebraic, homological, and numerical properties of the algebra $A(\textbf{k}, X, r)$. We have proven that each of these properties of $A$ is equivalent to "$(X,r)$ is a solution of YBE".
Our new Theorem \ref{thm:new_theorem} is a similar but stronger result. We "weaken" the hypothesis, assuming only that $(X,r)$ is a square-free nondegenerate quadratic set (we do not assume that  $r$ is involutive) and give a list of similar algebraic and homological  properties of $A$ each of which is equivalent to
"$(X,r)$ is an involutive solution of YBE (that is a symmetric set).

Theorem \ref{thm:new_theorem} below is a corrected and improved version of \cite[Theorem 3.16, p.763]{GI21}.
Here we present a new proof of the theorem which does not involve the wrong 
\cite[Proposition 3.12]{GI21}.

\begin{theorem}
\label{thm:new_theorem} 
Let $(X,r)$ be a square-free
nondegenerate quadratic set of finite order $|X| = n$. Let $A$ be its associated quadratic algebra,
$A = \textbf{k}\asX/(\textbf{R}(r))$. The following conditions are equivalent.
\begin{enumerate}
\item
\label{new_theorem1}
The Hilbert series of $A$ is
\begin{equation}
\label{eq:correctHilbSeries}
H_A(z)= \frac{1}{(1-z)^n}.
\end{equation}
\item
\label{new_theorem5}
There exists an enumeration of $X$,
$X = \{x_1, x_2, \cdots,  x_n\}$, such that the set
\begin{equation}
\label{eq:ncal}
\Ncal = \{x_1^{\alpha_{1}}x_2^{\alpha_{2}}\cdots
x_n^{\alpha_{n}}\mid \alpha_{i} \geq 0 \text{ for } 1 \leq i
\leq n\}
\end{equation}
is a $\textbf{k}$-basis of $A.$
\item
\label{new_theorem6}
\[
\dim_{\textbf{k}}
A_2= \binom{n+1}{2},\quad\text{and}\quad  \dim_{\textbf{k}} A_3= \binom{n+2}{3}.
\]
\item \label{new_theorem7}
\[\dim_{\textbf{k}}A^{!}_2= \binom{n}{2}, \quad\text{and}\quad \dim_{\textbf{k}}A^{!}_3 = \binom{n}{3}.\]
\item
\label{new_theorem2}
$A$ is a  PBW algebra with a set of PBW generators $X=\{x_1, x_2, \cdots, x_n\}$ with polynomial growth and finite global dimension.
\item
\label{new_theorem3}
$A$ is a  PBW algebra, with a set of PBW generators $X=\{x_1, x_2, \cdots, x_n\}$ , with finite global dimension, $\gldim A < \infty$ and exactly $\binom{n}{2}$ relations.
\item
\label{new_theorem3a}
 $A$ is a  PBW algebra, with a set of PBW generators $X=\{x_1, x_2, \cdots, x_n\}$ with polynomial growth and exactly $\binom{n}{2}$ relations.
\item
\label{new_theorem4}
$A$ is a PBW Artin-Schelter regular algebra.
\item
\label{new_theorem8}
$A$ is a binomial skew polynomial ring in the sense of \cite{GI96}.
\item
\label{new_theorem9}
$(X,r)$ is \emph{a square-free symmetric set}, that is an involutive solution of YBE.
\end{enumerate}
 In this case $A$ is Koszul and a Noetherian domain. Moreover,
\begin{equation}
\label{eq:PBW}
\begin{array}{l}
 \gkdim A=n= \gldim A,\quad \dim _{\textbf{k}} A_m = \binom{n+m-1}{m}, m \geq 2.
\end{array}
\end{equation}
\end{theorem}

\begin{proof}
Note that the hypothesis of the theorem implies $|\Fcal(X,r)|=n$.

We claim that each of conditions (1) through (10) implies that $(X,r)$ is involutive, i.e. $r^2 = id$.

By Proposition \ref{prop:mainthlemma1}
part
(\ref{prop:mainthlemma1_2}),
the following are equivalent 
\[r^2 = id \Longleftrightarrow  \dim_{\textbf{k}} A_2 = \binom{n+1}{2} \Longleftrightarrow \dim_{\textbf{k}}A^{!}_2= \binom{n}{2} \]
In particular,
\[(\ref{new_theorem6}) \Longrightarrow   r^2 = id\]
\[(\ref{new_theorem7}) \Longrightarrow   r^2 = id\]

Each of the conditions  (\ref{new_theorem1}) and  (\ref{new_theorem5}) implies that the algebra $A$ and $P^n=\textbf{k}[x_1, \cdots, x_n]$, the algebra of polynomials in $n$ variables, 
 have the same Hilbert series,  $H_A(z)= H_{P^n}(z)$, see (\ref{eq:correctHilbSeries}).
Therefore $A$ and $P^n$ have the same Hilbert functions
\begin{equation}
\label{eq:Hilb_Fun}
\dim A_m = h_A(m) = h_{P^n}(m) = \binom{n+m-1}{m}, m \geq 1.
\end{equation}

In particular, $\dim A_2 = \binom{n+1}{2}$, and $\dim A_3 = \binom{n+2}{2}$.

Therefore 
\[(\ref{new_theorem1})  \Longrightarrow (\ref{new_theorem6}) \Longrightarrow r^2 = id\]
 \[(\ref{new_theorem5})  \Longrightarrow (\ref{new_theorem6}) \Longrightarrow r^2 = id,\]

By definition every Artin-Shelter regular algebra has polynomial growth and finite global dimension hence
\[ (\ref{new_theorem4})  \Longrightarrow (\ref{new_theorem2})\]

 Our result \cite[Theorem 1.1]{GI12}
implies the equivalence of the three conditions
\[(\ref{new_theorem2}) \Longleftrightarrow (\ref{new_theorem3}) \Longleftrightarrow (\ref{new_theorem3a})\]

It follows that each of the conditions (\ref{new_theorem2}), (\ref{new_theorem3}), (\ref{new_theorem3a})
and (\ref{new_theorem4}) implies that $A$ has exactly $\binom{n}{2}$ relations, so $\dim A_2 = \binom{n+1}{2} $,
and by Proposition \ref{prop:mainthlemma1}
part (2), again, $r$ is involutive.

Obviously, each of the conditions (9) and (10) implies that $r$ is involutive.

We have shown that each of the conditions (1) through (10) implies that
 $(X,r)$ is a square-free nondegenerate involutive quadratic set, that is $(X,r)$ is \emph{a quantum binomial quadratic set}, see \cite[Definition 2.4 (4)]{GI12}.
Now our result
\cite[Theorem 1.2]{GI12}
implies straightforwardly the equivalence of conditions (1) through (10), the equalities (\ref{eq:PBW}) and the fact that $A$ is Koszul and a Noetherian domain.
\end{proof}

The 3-generated \emph{PBW} algebras from Lemma  \ref{lem:n3} are particular cases of the class of \emph{PBW} algebras described by Theorem \ref{thm:new_theorem}.

\section{Square-free quadratic sets with cyclic conditions}
\label{sec:cycliccondition}
In this section we continue the study
of square-free nondegenerate quadratic sets $(X,r)$, the
associated algebra $A(\textbf{k}, X,r)$, and the monoid $S(X,r)$.
In a series of works, see \cite{GI04, GIM08, GIC, CGISm18}, we
have shown that the combinatorial properties of a solution of YBE,
$(X,r)$, are closely related to the algebraic and combinatorial
properties of its associated structures. Solutions satisfying some
of the conditions defined below are of particular interest.
\begin{definition} \cite{GIM08, GI04}
\label{lri&cl}
Let $(X,r)$ be a quadratic set.
\begin{enumerate}
\item  The following are called \emph{cyclic conditions on}  $(X, r)$.
\[\begin{array}{lclc}
 {\rm\bf cl1:}\quad&{}^{(y^x)}x= {}^yx, \quad\text{for all}\; x,y \in
 X;
 \quad\quad&{\rm\bf cr1:}\quad &x^{({}^xy)}= x^y, \quad\text{for all}\; x,y
 \in
X;\\
 {\rm\bf cl2:}\quad&{}^{({}^xy)}x= {}^yx,
\quad\text{for all}\; x,y \in X; \quad\quad &{\rm\bf cr2:}\quad
&x^{(y^x)}= x^y, \quad\text{for all}\; x,y \in X.
\end{array}\]
\item Condition \textbf{lri} on $(X,r)$ is defined as
 \[\begin{array}{lclc}
 \textbf{lri:}\quad
& ({}^xy)^x= y={}^x{(y^x)}, &\text{for all} &
x,y \in X.\quad
\end{array}
\]
In other words \textbf{lri} holds if and only if
$(X,r)$ is nondegenerate, and \[\Rcal_x=\Lcal_x^{-1}, \quad \text{and}\quad\Lcal_x =
\Rcal_x^{-1}.\]
\end{enumerate}
\end{definition}
The cyclic conditions were introduced by the author in \cite{GI94, GI96},
in the context of binomial skew
polynomial algebras and were crucial for the proof that every binomial skew
polynomial algebra defines canonically (via its relations)
a set-theoretic solution of YBE, see \cite{GIVB}.
It is known that  \emph{every  square-free nondegenerate symmetric set $(X,r)$
satisfies the cyclic conditions \textbf{cc} and condition
\textbf{lri}, so the map $r$ is  uniquely determined by the left action:
$r(x,y) = (\Lcal _x(y), \Lcal^{-1}_y(x))$}, see \cite{GI04,
GIM08}.
We shall prove that every  square-free nondegenerate braided set  $(X,r)$ (not necessarily finite, or involutive)
satisfies the cyclic conditions \textbf{cl1} and \textbf{cr1}, see Proposition \ref{prop:cc}.
The main result of this section is  Theorem \ref{thm:proposition1AS}.

\subsection{Combinatorics in square-free quadratic sets with cyclic conditions}
We recall the following useful result.
\begin{fact}
\label{remarkvip}
\cite{GIM08}, Proposition 2.25.
Suppose $(X,r)$ is a quadratic set.  (1) Any two of the following conditions
 imply the remaining third
 condition: (i) $(X,r)$ is involutive;  (ii) $(X,r)$ is nondegenerate and
 cyclic; (iii) $(X,r)$ satisfies \textbf{lri}.
(2)  In particular, if $(X,r)$ satisfies \textbf{cl1}, and \textbf{cr1} then $(X,r)$ is involutive \emph{iff} condition \textbf{lri} holds.
 \end{fact}
\begin{sketchofproof}
 For convenience of the reader we shall sketch the proof of (2).
Assume  \textbf{lri}. We shall prove that $r$ is involutive, or equivalently, (\ref{involeq}) holds.
We apply first \textbf{cr1} and \textbf{lri}, and next  \textbf{cl1} and \textbf{lri} to yield:
\[
{}^{{}^uv}{(u^v)}= {}^{{}^uv}{(u^{{}^uv})}= u,\quad ({}^uv)^{u^v} = ({}^{u^v} v)^{u^v} = v, \; \forall u, v \in X.
\]
Conversely, assume that $(X,r)$ is involutive. We shall prove that \textbf{lri} holds
Let $u,t \in X$, we have to show ${}^t{(u^t)} =u$, and $({}^tu)^t=u$.
By the nondegeneracy, there exist $v, w \in X,$ such that $t = {}^uv= w^u$. Then we use \textbf{cr1}, \textbf{cl1}, and (\ref{involeq})
to yield:
\[
{}^t{(u^t)} = {}^{{}^uv}{(u^{{}^uv})} = {}^{{}^uv}{(u^v)}=u,\quad ({}^tu)^t=  ({}^{w^u}u)^{w^u} = ({}^wu)^{w^u} = u.
\quad\quad\quad \Box\]
\end{sketchofproof}

Suppose $(X,r)$ is a finite nondegenerate quadratic set, $S = S(X,r)$.
As we  discussed in the preliminaries, for every integer $m \geq 2$
the group  $\Dcal_m(r)= {}_{gr} \langle
r^{ii+1}, 1 \leq i\leq m-1 \rangle$ acts on the left on $X^m$.
Each element $a \in S$ can be presented as a monomial
$a = \zeta_1\zeta_2 \cdots \zeta_n,\quad \zeta_i \in X$.
Two words $a, b\in \asX$ are equal in $S$ if they have the same length, say $a, b \in X^m,$ and belong to the same orbit of $\Dcal_m(r)= {}_{gr} \langle
r^{ii+1}, 1 \leq i\leq m-1 \rangle$.
 Clearly,  $(X,r)$ is square-free if and only if $\Dcal_m(r)$  acts trivially on $\Delta_m$, $m \geq 2$.

\begin{corollary}
 \label{diaglemma}
 Suppose $(X,r)$ is a square-free quadratic set. Let $x,y \in X$, let $m$ be an integer, $m\geq 2$. If $x^m = y^m$ is an equality in $S$, then $x = y.$
 \end{corollary}

 \begin{proposition}
 \label{prop:cc}
 Let $(X,r)$ be a square-free nondegenerate braided set of arbitrary cardinality. Then $(X,r)$ satisfies the cyclic conditions \textbf{cl1} and \textbf{cr1}.
Moreover, $(X,r)$ is involutive \emph{iff} condition \textbf{lri} is in force.
 \end{proposition}
 \begin{proof}
Let $a,x \in X$. Consider the "Yang-Baxter" diagram on monomials of length $3$ in $X^3$.

\begin{equation}
\label{ybediagram}
\begin{CD}
axx \quad\quad\quad @>r^{23}>> \quad\quad\quad axx\\
@V  r^{12} VV @VV r^{12} V\\
({}^ax) (a^x) x @. ({}^ax) (a^x) x \\
@V r^{23} VV @VV r^{23} V\\
({}^ax)  ({}^{(a^x)}x)(a^{xx}) \quad\quad\quad @>r^{12}>> \quad\quad\quad ({}^ax)  ({}^{(a^x)}x)(a^{xx}).
 \end{CD}
\end{equation}
It follows that $r({}^ax , {}^{(a^x)}x) = ({}^ax , {}^{(a^x)}x)$
 and therefore, by Corollary \ref{sqfreerem}, ${}^{(a^x)}x = {}^ax$, which proves \textbf{cl1}.

Similarly, a YB diagram starting with the monomial $xxa$ implies $r(x^{({}^xa)},x^a) = (x^{({}^xa)},x^a)$, hence $x^{({}^xa)}=x^a$, which proves \textbf{cr1}.
Now Fact \ref{remarkvip} implies straightforwardly that $(X,r)$ is involutive \emph{iff}  \textbf{lri} holds.
 \end{proof}

The action of the infinite dihedral group  $\Dcal$ on $X^3$ is of particular importance in this
section. Assuming that $(X,r)$ is a nondegenerate square-free quadratic set
we shall
find some counting formulae and inequalities involving the orders of
the $\Dcal$-orbits in $X^3$, and their number.
As usual, the orbit of a monomial $\omega \in X^3$ under the  action of $\Dcal$
will be denoted by  $\cO= \cO(\omega)$.

\begin{definition}
\label{squarefreeorbitdef}
We  call a $\Dcal$-orbit $\cO$ \emph{square-free} if
\[\cO\bigcap (\Delta_2 \times X \bigcup X \times \Delta_2) = \emptyset.\]
A monomial $\omega\in X^3$ is \emph{square-free} in $S$ if its
orbit $\cO(\omega)$ is square-free.
\end{definition}

\begin{notation}
\label{notation:E(O)}
Denote 
\[
E(\cO): = \cO\bigcap((\Delta_2 \times X\bigcup X \times \Delta_2)\backslash \Delta_3).
\]
\end{notation}

\begin{theorem}
\label{thm:proposition1AS}
Suppose $(X,r)$ is a nondegenerate square-free quadratic set, of finite order $|X|=n$.
  \begin{enumerate}
\item
\label{thm:proposition1AS1}
Let $\cO$ be a $\Dcal$-orbit in $X^3$. The following implications hold.
\[
\begin{array}{lllll}
\emph{\textbf{(i)}}&\quad \cO \bigcap \Delta_3 \neq \emptyset&\Leftrightarrow
& |\cO|=1.&\quad\quad\quad\quad\quad\quad\\
\emph{\textbf{(ii)}}&\quad E(\cO)
\neq \emptyset & \Longrightarrow& |\cO| \geq 3.&
\end{array}
 \]
In this case we say that \emph{$\cO$ is a $\Dcal$-orbit of type \textbf{(ii)}}.
\[
\begin{array}{llll}
\emph{\textbf{(iii)}}&\quad \cO
 \bigcap(\Delta_2 \times X\bigcup X \times \Delta_2) = \emptyset&
 \Longrightarrow & |\cO| \geq 6.
  \end{array}
\]
Recall that in this case $\cO$ is called \emph{a square-free orbit}, see Definition \ref{squarefreeorbitdef}.
\item
\label{thm:proposition1AS2}
The following two conditions are equivalent
\begin{enumerate}
\item $(X,r)$ is involutive and satisfies the cyclic conditions \textbf{cl1} and \textbf{cr1};
\item Every  orbit $\cO$ of type \textbf{(ii)}
contains exactly 3 distinct elements.
\end{enumerate}
\end{enumerate}
\end{theorem}
\begin{proof}
Condition (\ref{thm:proposition1AS1}) \textbf{(i)} is clear.

(\ref{thm:proposition1AS1}) \textbf{(ii)}.  Assume that $E(\cO)
\neq \emptyset$. Then $\cO$
contains an element of the shape $\omega = xxy,$ or $\omega =
xyy$, where $x, y \in X, x \neq y.$ Without loss of generality we
can assume $\omega = xxy \in \cO.$
   We look at a fragment of the "Yang-Baxter" diagram starting with $\omega$:
   \begin{equation}
\label{orb2}
\omega=\omega_1 = xxy \longrightarrow^{r^{23}} \omega_2 = x ({}^xy) (x^y) \longrightarrow^{r^{12}}  \omega_3= ({}^{x^2}y) (x^{{}^xy}) (x^y) \longrightarrow \cdots .
\end{equation}
Note that the first three elements $\omega_1, \omega_2,
\omega_3$ are distinct monomials in $X^3$. Indeed, $x \neq y$
implies $r(xy)\neq xy$ in $X^2$, (see Corollary \ref{sqfreerem}), so $\omega_2 \neq \omega_1$. By
assumption $(X,r)$ is square-free, so ${}^xx=x$, and
$y \neq x,$ implies  ${}^xy \neq x,$ by the nondegeneracy.  Therefore
$r(x({}^xy))\neq x({}^xy),$ and $\omega_3 \neq
\omega_2.$ Furthermore, $\omega_3 \neq \omega_1.$ Indeed, if we
assume $x = {}^{x^2}y = {}^x{({}^xy)}$ then by
(\ref{rightactioneq1}) one has ${}^xy = x$, and therefore $y = x,$
a contradiction. It follows that $\mid\cO\mid \geq 3.$

(\ref{thm:proposition1AS1}) \textbf{(iii)}.
Suppose  $\cO= \cO(xyz)$ is a square-free $\Dcal$-orbit in $X^3$. Consider the set \[ O_1=  \{ v_i \mid 1 \leq i \leq 6 \}\subseteq \cO\] consisting of the first six elements of the  "Yang-Baxter" diagram
\begin{equation}
\label{ybediagram3}
\begin{CD}
v_1= xyz \quad\quad\quad @>r^{12}>> \quad\quad\quad ({{}^xy}x^y)z = v_2\\
@V  r^{23} VV @VV r^{23} V\\
  v_3=x({{}^yz}y^z)@. ({{}^xy}) ({}^{x^y}z)(x^y)^z= v_5\\
@V r^{12} VV @VV r^{12} V\\
\kern -80pt v_4={}^{x}{({}^yz)}(x^{{}^yz})(y^z) @.
\kern-100pt \quad\quad \quad\quad\quad\quad\quad\quad\quad\quad    [{}^{{}^xy}{({}^{x^y}z)}][({}^xy)^{({}^{x^y}z)}][(x^y)^z]= v_6.
 \end{CD}
\end{equation}
Clearly,
\[
O_1 = U_1\bigcup U_3\bigcup U_5, \quad\text{where}\quad U_j = \{v_j, \;v_{j+1}  = r^{12}(v_j) \},\quad j = 1,3,5.
\]
We claim that $U_1, U_3, U_5$ are pairwise disjoint sets, and each of them has order 2.
Note first that since  $v_j$ is  a square-free monomial, for each $j = 1, 3, 5$, one has
$v_j\neq r_{12}(v_j)= v_{j+1}$,
therefore
$|U_j| = 2, \;\; j = 1,3,5.$
The monomials in each $U_j$ have the same "tail". More precisely,
$v_1 = (xy)z, v_2=r(xy)z$, have a "tail" $z$, the tail of $v_3$,
and $v_4$ is $y^z$, and the  tail of $v_5$, and $v_6$ is
$(x^y)^z$. It will be enough to show that the three elements $z,
y^z,  (x^y)^z \in X$ are pairwise distinct. But $\cO(xyz)$ is
square-free, so $y\neq z$ and by (\ref{rightactioneq1}) $y^z \neq
z$.
 Furthermore $v_2 = ({}^xy)(x^y)z \in \cO(xyz)$, so $x^y \neq y$ and $x^y \neq z$. Now by the nondegeneracy
one has
\[\begin{array}{l}
x^y \neq z \Longrightarrow (x^y)^z \neq z\quad\quad
x^y \neq y \Longrightarrow  (x^y)^z \neq y^z.
\end{array}
\]
Therefore the three elements $z, y^z, (x^y)^z \in X$
occurring as tails in $U_1, U_3, U_5$, respectively,  are pairwise
distinct, so the three sets are pairwise disjoint. This implies
$6= |O_1| \leq |\cO|.$

(\ref{thm:proposition1AS2}). (a) $\Rightarrow$ (b).
Suppose $(X,r)$ is involutive and satisfies \textbf{cl1} and \textbf{cr1}. Let $\cO$ be an orbit of type \textbf{(ii)}. Without loss of generality we may assume
$\cO = \cO(xxy)$. Then (since $r$ is involutive)
each arrow in the diagram (\ref{ybediagram4})
is pointed in both directions, i.e. the arrows have the shape $\longleftrightarrow$, or $\updownarrow$.
\begin{equation}
\label{ybediagram4}
\begin{CD}
v_1= xxy \quad @>r^{12}>> \quad xxy = v_1  \\
@V  r^{23} VV @VV r^{23} V\\
  v_2=x({{}^xy}x^y)@. x({{}^xy}x^y)= v_2\\
@V r^{12} VV @VV r^{12} V\\
\kern -80pt v_3=({}^{x}{({}^xy)}) (x^{{}^xy}) (x^y)= ({}^{xx}y) x^yx^y \quad @>r^{23}>>\quad
({}^{x}{({}^xy)}) (x^{{}^xy}) (x^y)= ({}^{xx}y) (x^y) (x^y) =v_3
 \end{CD}
\end{equation}
It is clear that the diagram contains
all elements of $\cO$, hence $|\cO| = 3.$

(b) $\Rightarrow$ (a). Suppose every orbit $\cO$ of type \textbf{(ii)} contains exactly three elements.
The diagram
\begin{equation}
\label{ybediagram3a}
\begin{CD}
v_1= xxy \quad\quad\quad @>r^{12}>> \quad\quad\quad xxy = v_1\\
@V  r^{23} VV @VV r^{23} V\\
  v_2=x({{}^xy}x^y)@.  x({{}^xy}x^y)=v_2\\
@V r^{12} VV @VV r^{12} V\\
\kern -80pt v_3={}^{x}{({}^xy)}(x^{{}^xy})(x^y) @.
\kern-100pt \quad\quad \quad\quad\quad\quad\quad\quad\quad\quad    {}^{x}{({}^xy)}(x^{{}^xy})(x^y)= v_3.
 \end{CD}
\end{equation}
contains three distinct elements of $\cO$, and therefore it contains the whole orbit $\cO$.

The element
$r^{23}(v_3) = {}^{x}{({}^xy)} r((x^{{}^xy})(x^y)) \in \cO =\{v_1, v_2, v_3\}$. It is clear that $r^{23}(v_3)\neq v_1$ and $r^{23}(v_3)\neq v_2$, so
$r^{23}(v_3) =v_3$. This
implies that $ r((x^{{}^xy})(x^y))= (x^{{}^xy})(x^y)$, hence $(x^{{}^xy})(x^y) \in \Fcal(X,r)= \Delta_2$. Therefore
 \[x^{{}^xy}=x^y,  \forall x, y \in X,\]
 that is \textbf{cr1} holds.
An analogous argument proves \textbf{cl1} (in this case we work with a "YB" diagram with a left top element $v_1 = xyy$).

Notice  that if there exists a pair $(x,y)$ with $r^2(xy) \neq xy$ then the orbit $\cO(xxy)$ contains (but is not limited to) the following four distinct elements
\[
\begin{array}{l}
v_1 = xxy, \quad v_2 = r^{23}(v_1) = xr(xy) = x({{}^xy}x^y), \\
v_3 = r^{12}(v_2)= ({}^{xx}y) (x^{{}^xy}) (x^y), \quad v_4 = r^{23}(v_2)= x r^2(xy)
\end{array}\]
which contradicts (b).
It follows that $r$ is involutive. We have proven (b) $\Rightarrow$ (a).
\end{proof}

\subsection{More on square-free quadratic sets with cyclic conditions}
We end up the section with new results on square-free quadratic
sets which will be used to describe the contrast between
involutive and noninvolutive solutions of YBE in the next section.

\begin{proposition}
\label{prop:A!_and_YBE}
Suppose $(X,r)$ is a finite nondegenerate square-free quadratic set with  $|X|=n$ that satisfies
\textbf{cl1} and \textbf{cr1}.
Then $(X,r)$ is a symmetric set if and only if
 $\dim_{\textbf{k}} A^{!}_3 = \binom{n}{3}.$
\end{proposition}
\begin{proof}
By hypothesis \textbf{cl1} and \textbf{cr1} hold.
Assume $\dim_{\textbf{k}} A^{!}_3 = \binom{n}{3}$. We have to show that $(X,r)$ is a symmetric set.
We shall prove first that $(X,r)$ is involutive, and therefore it is a quantum binomial set, see
Definition \ref{def:quadraticsets_All} (iv).

As usual, we study the $\Dcal_3(r)$-orbits $\cO$. Our assumption
implies that $X^3$ contains exactly $\binom{n}{3}$ square-free orbits, $\cO^{(s)}, 1 \leq s \leq \binom{n}{3}$.
By Theorem \ref{thm:proposition1AS} part (\ref{thm:proposition1AS1}.iii)
the length of each square-free orbit satisfies
\begin{equation}
\label{lengthorbit}
|\cO^{(s)}|= l_s \geq 6, \;\;  \forall\; \; 1 \leq s \leq \binom{n}{3}.
\end{equation}

Denote by $W$ the set of all words $w \in X^3\setminus \Delta_3$ which vanish  in $A^{!}_3$.
Note first that if $y,b \in X, y\neq b$, the orbit $\cO (yyb) \subset W$ contains the three distinct monomials occurring in the following diagram

\[u=yyb \longrightarrow _{r^{23}} y ({}^yb y^b) \longrightarrow _{{r^{12}}} ({}^{yy}b) (y^{{}^yb}.y^b) =({}^{yy}b) (y^b.y^b).\]

 We shall call the word $r^{23}(yyb)= y ({}^yb y^b)$ "\emph{the transition element} \emph{for the pair of words} $u= yyb, czz=r^{12}\circ r^{23}(u)\in W$". It is clear that each pair $y,b\in X, y \neq b$ determines uniquely the three elements $yyb, r^{23}(yyb), czz = r^{12}\circ r^{23} (yyb)\in W$.

Note that if $(y,b)\neq (t,c)$ then the transition elements $y ({}^yb y^b) \neq t ({}^tc t^c)$.
Indeed, the inequality is straightforward if $t \neq y$. If $t =y, b \neq c,$ then by the nondegeneracy, one has
${}^yb \neq  {}^yc = {}^tc.$
So $W$ contains $n(n-1)$ disjoint triples $yyb, r^{23}(yyb), r^{12}\circ r^{23} (yyb)= czz$,
therefore $|W|\geq 3n(n-1)$.

Assume that $(X,r)$ is not involutive, we shall prove that $|W|> 3n(n-1)$.

Clearly, there exist a pair $x,a \in X, x\neq a$, such that $r^2(x,a) \neq (x,a)$
so the words $xa, r(xa), r^2(xa)$ are distinct elements of $X^2$.
Then the orbit $\cO= \cO(xxa)$ contains at least the set $\textbf{O}$  of 4 distinct monomials given below:

\begin{equation}
\label{orbiteq1}
\begin{array}{ll}
\textbf{O}= \{&v_1 = xxa, \; v_2 = r^{23}(xxa)= x ({}^xa x^a), \\
              &v_3 = r^{12}\circ r^{23}(v_1) = ({}^{xx}a) (x^ax^a), \\
              &v_4= (r^{23})^2(v_1) = xr^2(xa)\}.
\end{array}
\end{equation}

Moreover, the set $\textbf{O}$ contains the word $v_4= xr^2(xa)$ which is square-free, but is not a transition element
for any triple $yyb, r^{23}(yyb), r^{12}\circ r^{23}(yyb) = czz.$
This implies that
\begin{equation}
\label{orderW}
|W| > 3n(n-1).
 \end{equation}

 The set $X^3$ splits into the following disjoint subsets
\[X^3 = \Delta_3 \bigcup W \bigcup (\bigcup_{1 \leq s \leq \binom{n}{3}} \cO^{(s)}).  \]
This, together with  (\ref{lengthorbit}) and  (\ref{orderW}) imply
\begin{equation}
\label{ordereq}
n^3 = |X^3| = |\Delta_3|+ |W |+ \sum _{1 \leq s \leq \binom{n}{3}} |\cO^{(s)}|>  n+ 3n(n-1) +  6\binom{n}{3} = n^3,
\end{equation}
which gives a contradiction.
It follows that $r$ is involutive, hence $(X,r)$ is a quantum binomial set. Now our result \cite[Theorem 2]{GI12}, implies that $(X,r)$ is a solution of YBE, therefore,
it is a symmetric set.
The inverse implication follows again from \cite[Theorem 2]{GI12}.
\end{proof}

\begin{lemma}
 \label{lem:Vipprop}
 Let $(X,r)$ be a finite square-free nondegenerate quadratic set, $|X|= n$, and let $S=S(X,r)$.
Suppose $(X,r)$ satisfies the cyclic conditions \textbf{cl1} and \textbf{cr1}.
 The following conditions are equivalent:
\begin{enumerate}
\item
\label{lem:Vipprop1}
$(X,r)$ is involutive.
\item
\label{lem:Vipprop2}
$S$ satisfies the following conditions:
\begin{equation}
\label{vipcondition}
\begin{array}{l}
 axx = byy\;\text{holds in $S$}, \;a, b, x, y \in X \Longrightarrow a = b, x =y;\\
 xxc = yyd\;\text{holds in $S$}, \; c, d, x, y \in X  \Longrightarrow c = d, x =y.
\end{array}
\end{equation}
\item
\label{lem:Vipprop3}
\[\begin{array}{l}
 byy\in \cO(axx) \;\;a, b, x, y \in X \Longrightarrow  a = b, x =y;\\
yyd \in \cO(xxc) \;\; c, d, x, y \in X  \Longrightarrow c = d, x =y.
\end{array}\]
\end{enumerate}
\end{lemma}
\begin{proof}

The equivalence (\ref{lem:Vipprop2}) $\Leftrightarrow$ (\ref{lem:Vipprop3}) is clear.

(\ref{lem:Vipprop1}) $\Rightarrow$ (\ref{lem:Vipprop2}).
Suppose $(X,r)$ is involutive. Theorem \ref{thm:proposition1AS} implies that
 for each $a\neq x$ the orbit $\cO=\cO(axx)$ is of type \textbf{(ii)}
and $\cO\bigcap (X \times \Delta_2) =\{axx\}$, in other words there is no element of the shape $byy\neq axx$ such that $byy\in \cO$ which gives the first
implication in (\ref{vipcondition}). Analogous argument gives the second implication in (\ref{vipcondition}).

(\ref{lem:Vipprop2}) $\Rightarrow$ (\ref{lem:Vipprop1})
Conversely, assume that conditions (\ref{vipcondition}) are in force. We have to show that $r$ is involutive.
By Fact \ref{remarkvip}
it will be enough to prove that $(X,r)$ satisfies condition \textbf{lri}, that is
\begin{equation}
\label{eq:lri}
({}^tx)^t= x\quad\text{and}\quad  {}^t{(x^t)} = x, \forall x,t \in X.
\end{equation}

Let $a, x \in X$. We consider the elements in the $\Dcal$-orbit $\cO(axx)$ in $X^3$ and deduce the following equalities of elements in $S=S(X,r)$:
\begin{equation}
\label{vipcond1}
\begin{array}{ll}
a.xx   &=({}^ax)(a^x) x \\
       &=({}^ax)({}^{a^x}x)a^{xx}= ({}^ax)({}^{a}x)(a^{xx})   \\
             &=({}^ax)({}^{{}^ax}{(a^{xx})})    ({({}^ax)}^{a^{xx}})\\
             &= ({}^({}^{{}^ax}{({}^{{}^ax}{(a^{xx})})})(({}^ax)^{({}^{{}^ax}{(a^{xx})})}({({}^ax)}^{(a^{xx})})\\
                &=b ({({}^ax)}^{a^{xx}})({({}^ax)}^{a^{xx}}) \quad\text{where}\quad b= {}^{({}^ax)}{({}^{({}^ax)}{(a^{xx})})}\\
         &=b yy,\quad\text{where}\quad y= [({}^ax)]^{(a^{xx})}.
\end{array}
\end{equation}
We have obtained that for $a\neq x$ the following equalities holds in $S$
\[axx = b yy, \quad \text{where} \quad y =[({}^ax)]^{(a^{xx})}. \]
Now the first condition in  (\ref{vipcondition}) implies
\begin{equation}
\label{vipcond2}
y = [({}^ax)]^{(a^{xx})} = x,
\end{equation}
and
\begin{equation}
\label{vipcond4}
\begin{array}{l}
{}^ax ={}^{a^x}x =  {}^{(a^x)^x}x= {}^{(a^{(xx)})}x\\
x = [({}^ax)]^{(a^{xx})} =  [{}^{(a^{(xx)})}x]^{(a^{xx})} = ({}^tx)^t,  \quad\text{where}\quad t=a^{(xx)}.
\end{array}
\end{equation}

Suppose $t,x \in X$. By the nondegeneracy there exists $a_1 \in X$, such that $t = a_1^x$, and, similarly, there exists $a \in X$ with
 $a^x = a_1,$ hence $t = a^{xx}$, for some $a\in X$.
Then (\ref{vipcond4}) implies $({}^tx)^t= x$.
The second equality  ${}^t{(x^t)}= x$ in (\ref{eq:lri}) is proven by an analogous argument. Therefore $(X,r)$
satisfies condition \textbf{lri}.
\end{proof}

\section{Square-free braided sets and the contrast between the involutive and noninvolutive cases}
\label{sec:contrast}

Braided monoids were introduced and studied in \cite{GIM08}. For convenience of the reader we recall basic definitions and results in Section
\ref{BraidedMonoidSec}.
Recall that if $(X,r)$ is a braided set then its monoid $S(X,r)$ is a graded braided
{\bf M3}-monoid, we denote it by $(S, r_S)$, see Definitions \ref{MLaxioms} and \ref{braidedmonoiddef}, in particular, $S$ satisfies condition \textbf{ML2}. More details can be found in Section \ref{BraidedMonoidSec}.

\begin{notation}
\label{not:actions}
Let $(X,r)$ be a nondegenerate quadratic set.
Let $a,x \in X$ and let $m$ be a positive integer. We shall use the following notation
\[{}^{(x^m)}a:= (\Lcal_{x}^m)(a) \quad  a^{(x^m)}:= (\Rcal_{x}^m)(a).\]

This formulae agree with the natural left and right actions of $S$ upon itself.
\end{notation}
\begin{remark} Suppose $(X,r)$ is a quadratic set with \textbf{cl1} and \textbf{cr1}.
Then  the following equalities hold in $X$:
 \begin{equation}
\label{eq:new_cyclic}
{}^{a^{(x^m)}}x = {}^ax, \quad x^{{}^{(x^m)}a} = x^a, \quad \text{for all} \; a,x \in X\; \text{and all positive integers }\; m.
\end{equation}
 \end{remark}
The formulae (\ref{eq:new_cyclic}) are easy to prove using induction on $m$.

\begin{proposition}
 \label{prop2}
 Let $(X,r)$ be a square-free nondegenerate quadratic set satisfying the cyclic conditions \textbf{cl1} and \textbf{cr1},
 and let $S=S(X,r)$.
 Then the following conditions hold.
 \begin{enumerate}
\item
\label{prop21}
For every pair $a,x \in X$ and every positive integer $m$ the following equalities hold in $S$:
\begin{equation}
\label{newrel2}
a.(x^m) = (({}^{a}x)^m).(a^{(x^m)}), \quad (x^m).a=({}^{(x^m)}a)((x^a)^m).
\end{equation}

\item
\label{prop22}
Assume that $(X,r)$ is a braided set.
Then the following are equalities in the braided monoid $S$:
 \begin{equation}
\label{newrel1}
{}^a{(x^m)}= ({}^ax)^m,\; (x^m)^a = (x^a)^m,\; \text{for all } \; a,x \in X\;
\text{and all positive integers }\; m.
\end{equation}
 \item
\label{prop23}
 Suppose that $(X, r)$ is a finite braided set, and let $p$ be the least common multiple of the orders of all actions, $\Lcal_x$ and
 $\Rcal _x, x \in X$, so
${}^{(x^p)}a = a =a^{(x^p)}, \; \forall a, x \in X$.
 Then the following equalities hold in $S:$
 \begin{equation}
\label{newrel3}
a.(x^p)= a.(({}^ax)^a)^p, \quad\quad  (x^p).a  = ({}^a{(x^a)})^p.a, \forall a,x \in X.
\end{equation}
 \begin{equation}
\label{newrel4}
(x^p)(y^p)=  (y^p)(x^p),  \forall x, y \in X.
\end{equation}
\end{enumerate}
 \end{proposition}
 \begin{proof}
(\ref{prop21}).
We shall use induction on $m$ to prove the first equality in (\ref{newrel2}).
Clearly, for $m=1$, one has $ax = {}^ax.a^x.$ Assume $a.(x^k) = (({}^{a}x)^k).(a^{(x^k)}), \forall 1 \leq k \leq m$, and all $a,x \in X$.
Let $a, x \in X$. Then
\[\begin{array}{ll}
a.(x^{m+1}) &= (a.x^{m})x=  (({}^{a}x)^m).[(a^{(x^m)})x] =({}^{a}x)^m.({}^{a^{(x^m)}}x).(a^{(x^m)})^x\\
&= ({}^ax)^{m+1})(a^{(x^{m+1})}),
\end{array}
\]
as claimed. For the last equality we have used (\ref{eq:new_cyclic}).

This verifies the first equality in (\ref{newrel2}). Analogous argument verifies the second equality in (\ref{newrel2})

(\ref{prop22}).
Assume that $(X,r)$ is a braided set. Then $(X,r)$ satisfies \textbf{cl1} and \textbf{cr1}, see Proposition \ref{prop:cc}.
We shall prove  (\ref{newrel1}) using induction on $m$. The base for induction is clear. Assume the formula is true for all $k \leq m-1$, where $m\geq 2$ . We use the inductive assumption,  \textbf{ML2} and (\ref{eq:new_cyclic}) to yield
\[{}^a{(x^m)}= {}^a{((x^{m-1}).x)}= ({}^a{(x^{m-1})}). ({}^{(a^{(x^{m-1})})}x) = ({}^ax)^{m-1})({}^ax)=
 ({}^a{x})^m.\]
This proves the first equality in (\ref{newrel1}). Analogous argument verifies the second equality in (\ref{newrel1}).

(\ref{prop23}).
Assume now that $(X, r)$ is a finite braided set and $p$ is the least common multiple of the orders of all actions $\Lcal_x$ and $\Rcal _x, x \in X$.
We use successively \textbf{M3}, (\ref{newrel1}) and  \textbf{M3} again to yield
 \begin{equation}
\label{newrel7}
a.(x^p)= {}^a{(x^p)}.(a^{(x^p)}) = ({}^ax)^p. a =  ({}^{(({}^ax)^p)}a ).(({}^ax)^p)^a = a.(({}^ax)^a)^p.
\end{equation}
This gives the first equality in (\ref{newrel3}). Analogous argument proves the second equality in (\ref{newrel3}).
The equality (\ref{newrel4}) is straightforward.
\end{proof}

\begin{proposition}
 \label{prop3}
 Let $(X,r)$ be a square-free nondegenerate braided set of finite order $|X|=n$, let $S=S(X,r) = (S, r_S)$ be the associated braided monoid, and let $A = A(\textbf{k},X,r)$ .
Let $p$ be the least common multiple of the orders of all actions, $\Lcal_x$ and
 $\Rcal_x, x \in X$.
 The following conditions are equivalent
 \begin{enumerate}
 \item
 \label{prop31}
 The equality $ax^p = ay^p$ in $S$, for $a,x, y \in X$, implies $x = y$;
\item
\label{prop31a}
The equality
 $(x^p)a = (y^p)a$ in $S$, for $a,x, y \in X,$ implies $x = y$.
 \item
 \label{prop32}
 The monoid $S$ is cancellative.
\item
 \label{mainthGKa}
The quadratic algebra $A$ has \emph{Gelfand-Kirillov dimension}   $\gkdim A = n$.
 \item
 \label{prop33}
 The solution $(X,r)$ is involutive, that is  $(X,r)$ is a symmetric set.

 In this case $S(X,r)$ is embedded in the braided group $G(X,r)= (G, r_G)$. Moreover,
 both $(S, r_S)$ and $(G, r_G)$ are also (nondegenerate) involutive solutions.
 \end{enumerate}
 \end{proposition}

\begin{proof}
The implication (\ref{prop32}) $\Longrightarrow$ (\ref{prop31}), and (\ref{prop31a}) is clear.

(\ref{prop31}) $\Longrightarrow$  (\ref{prop31a}).
Assume $(x^p)a = (y^p)a$, where $a,x, y \in X.$
Then we use (\ref{newrel2}) to obtain
\begin{equation}
\label{eq:newrel2a}
\begin{array}{l}
 (x^p)a = a(x^a)^p, \quad \quad    (y^p)a = a (y^a)^p.
\end{array}
\end{equation}
It follows that $a(x^a)^p=a (y^a)^p$, so by our assumption (\ref{prop31}),  $x^a = y^a$, and therefore, by the non-degeneracy, $x=y$.
The implication (\ref{prop31a}) $\Longrightarrow$  (\ref{prop31}) is proven analogously.

(\ref{prop31}) $\Longrightarrow$ (\ref{prop33}).
Suppose condition (\ref{prop31}) holds (hence (\ref{prop31a}) is also in force).
Proposition \ref{prop2} implies the following equalities in $S$:
\[
a.(x^p)= a.(({}^ax)^a)^p, \quad  (x^p).a  = ({}^a{(x^a)})^p.a,\; \forall a,x \in X.
\]
It follows from (\ref{prop31}) that
\[x= (({}^ax)^a); \quad x=({}^a{(x^a)}), \;\forall a,x \in X,\]
therefore the braided set $(X,r)$ satisfies condition \textbf{lri}.
By Fact \ref{remarkvip} (2), $(X,r)$ is involutive, so $(X,r)$ is a nondegenerate symmetric set.

(\ref{prop33}) $\Longrightarrow$ (\ref{prop32}). It is known, see \cite{ESS}, that if $(X,r)$ is a nondegenerate symmetric set
then its monoid $S(X,r)$ is embedded in the group $G(X,r)$, and therefore $S$ is left and right cancellative.

The implication (\ref{prop33}) $\Longrightarrow$ (\ref{mainthGKa}) is known, see \cite[ Theorem 1.2]{GI12}, or \cite{GIVB}.

(\ref{mainthGKa}) $\Longrightarrow$ (\ref{prop31}).
Suppose $A$ has \emph{Gelfand-Kirillov dimension}   $\gkdim A = n$.
Assume, on the contrary, that condition (\ref{prop31})
is not satisfied.
Then there exist  three elements $a,x, y \in X, a \neq x, y$, such that
\[ax^p = ay^p, \; x \neq y.\]
This implies that $ax^{Mp} = ay^{Mp}$, for all positive integers $M$, hence $\gkdim A < n$, a contradiction.
\end{proof}
The following result shows the close relations between various algebraic and combinatorial properties of a finite square-free solution $(X,r)$, the YB-algebra $A= A(\textbf{k}, X, r)$ and its braided monoid, $S=S(X,r)$. Each of these conditions describes the contrast between a square-free symmetric set and a square-free noninvolutive braided set.

As we mentioned above, under
the hypothesis that  $(X,r)$ is a finite nondegenerate square-free braided set,  the result  \cite[Theorem 5.5]{GI21} gives a  list of 14 equivalent conditions each of which is necessary and sufficient so that $r$ is involutive. We remark that the original statement contains conditions (4) "The quadratic algebra $A$ is Koszul" and (5) "There is an enumeration of $X$ such that the set of quadratic relations $R(r)$ is a Gr\"{o}bner basis, that is $A$ is a PBW algebra",  
which are only necessary, but none of them implies that $r$ is imvolutive, so they should be excluded from the list.
Theorem \ref{mainth} below is a corrected and improved version of the original
\cite[Theorem 5.5]{GI21}. 
Our new proof given below is an easy and correct modification of the old one and does not involve the wrong \cite[Proposition 3.12]{GI21}.

\begin{theorem}
  \label{mainth}
 Let $(X,r)$ be a square-free nondegenerate braided set of order $|X| =n$.
  Let $S=S(X,r)$ be its braided monoid, let $A = A(\textbf{k}, X, r)$
 be its graded $\textbf{k}$- algebra over a field $\textbf{k}$, and
 let  $A^{!}$ be the Koszul dual
algebra of $A$.
The following conditions are equivalent:
 \begin{enumerate}
  \item
 \label{mainth1}
 The solution $(X,r)$ is involutive, that is  $(X,r)$ is a symmetric set.
\item
 \label{mainth2}
 $(X,r)$ satisfies \textbf{lri}.
\item
\label{mainth3} The Hilbert series of $A$ is
$H_A(z)= \frac{1}{(1-z)^n}.$
 \item
\label{mainth6}
$A$ is a binomial skew polynomial ring, (in the sense of \cite{GI96})  with
respect to an enumeration of $X$.
 \item
 \label{mainth7}
$\dim_{\textbf{\textbf{k}}} A_2 =\binom{n+1}{2}$. 
\item
\label{mainth8} $\dim_{\textbf{k}} A_3 = \binom{n+2}{3}$. 
\item
\label{mainth9} $\dim_{\textbf{k}} A^{!}_3 = \binom{n}{3}$. 
\item
\label{mainth10}
 The algebra $A$ has \emph{Gelfand-Kirillov dimension} $\gkdim A = n$.  (Meaning that the
integer-valued function $i\mapsto\dim_{\textbf{k}}A_i$ is bounded by
a polynomial in $i$ of degree $n$).
 \item
  \label{mainth11}
    If $ax^p = ay^p$ holds in $S$, where $a,x, y \in X$,
$p$ is the least common multiple of the orders of all actions $\Lcal_x$ and $\Rcal_x$, $x\in X$, then $x = y$,.
  \item
  \label{mainth12}
  The monoid $S$ satisfies conditions (\ref{vipcondition})
\item
 \label{mainth13}
 The monoid $S$ is cancellative.
 \item
 \label{mainth14}
$A$ is a domain.
 \end{enumerate}
 Each of these conditions implies that $A$ is Koszul, a Noetherian domain and a PBW 
 Artin-Schelter regular algebra of global dimension $n$.
 \end{theorem}

\begin{proof}
Note first that
 $(X,r)$ satisfies \textbf{cl1}, and \textbf{cr1}, by \cite[Proposition 4.4]{GI21}. Moreover, $|\Fcal(X,r)| = n$.
It is known that a finite square-free nondegenerate symmetric set $(X,r)$
satisfies all conditions (2) through (14) in the theorem,
so (1) implies all conditions (2) through (12). These implications have been published in various works of the author,
but one can find them all in \cite[Theorem 1.2]{GI12}.
The remaining implications with references to the corresponding results follow.
\[
\begin{array}{llll}
&(\ref{mainth1}) \Leftrightarrow (\ref{mainth2}) \quad &:& \text{\cite[Proposition 2.25]{GIM08}}\\
&(\ref{mainth1}) \Leftrightarrow (\ref{mainth3}) \quad &:& \text{Theor.\ref{thm:new_theorem}}\\
&(\ref{mainth6})\Longrightarrow (\ref{mainth1}) \quad &:& \text{see, \cite{GI04}, or \cite[Theorem 1.2]{GI12}}\\
&(\ref{mainth7}) \Longrightarrow  (\ref{mainth1}) \quad &:& \text{Proposition \ref{prop:mainthlemma1}, part (\ref{prop:mainthlemma1_2})}\\
&(\ref{mainth8}) \Leftrightarrow  (\ref{mainth9}) \quad &:& \text{Easy to prove, we leave it to the reader}\\
&(\ref{mainth9}) \Longrightarrow  (\ref{mainth1}) \quad &:& \text{\cite[Proposition 4.8]{GI21}} \\ 
&(\ref{mainth10}) \Leftrightarrow (\ref{mainth11})\Leftrightarrow (\ref{mainth13})\Leftrightarrow (\ref{mainth1}) \quad &:& \text{\cite[Proposition 5.4]{GI21}}\\
&(\ref{mainth12}) \Longrightarrow  (\ref{mainth1}) \quad &:& \text{by \cite[Lemma 4.9]{GI21}}\\
&(\ref{mainth14}) \Longrightarrow  (\ref{mainth13}) \quad &:& \text{ Clear}\\
&(\ref{mainth1}) \Longrightarrow  (\ref{mainth14}) \quad &:& \text{see \cite{GIVB} }\\
\end{array}
\]
\end{proof}
Artin-Schelter regular algebras were introduced and studied first in \cite{AS}. The study of Artin–Schelter regular algebras,
their classification, and finding new classes of such algebras is one of the central problems in
noncommutative algebraic geometry.

 Follows a corrected version of \cite[Corollary 5.6]{GI21}. Note that parts (2) and (3) in the original version of this corollary have been erased.   
 \begin{corollary} (Characterization of \emph{noninvolutive} square-free braided sets )
  \label{mainthCor}
Under the hypothesis of Theorem \ref{mainth}.
Let $(X,r)$ be a square-free nondegenerate braided set of order $|X| =n$, $S(X,r)$, $A = A(\textbf{k}, X, r)$, and $A^{!}$ in usual notation.
Suppose $r^2 \neq \id_{X\times X}$.
Then the following conditions hold.
 \begin{enumerate}
 \item $(X,r)$ does not satisfy condition \textbf{lri}.
 \item $A$ is a not a binomial skew polynomial ring, with
respect to any enumeration of $X.$
 \item
$2n-1 \leq  \dim_{\textbf{k}} A_2 \leq \binom{n+1}{2}-1$.
 \item
$\dim_{\textbf{k}} A_3 < \binom{n+2}{3}.$
\item
$0 \leq \dim_{\textbf{k}} A^{!}_3 < \binom{n}{3}$,
and $A_3^{!}= 0$, whenever $\dim_{\textbf{k}} A_2 =2n-1$.
\item
$GK\dim A < n$.
 \item
 Suppose $p$ is the least common multiple of the orders of all actions, $\Lcal_x$ and $\Rcal_x$, $x \in X$.
Then there exist pairwise distinct elements $a,x, y \in X$, such that $ax^p = ay^p$ holds in $S$.
 \item There exist $x,y \in X$, such that $x\neq y$, and $x^p = y^p$ holds in the group $G(X,r)$.
  \item There exist $a,b,x, y\in X$, such that $x\neq y, x \neq a, y \neq b$, and the equality $axx =byy$ holds in $S$.
  \item The monoid  $S=S(X,r)$ is not cancellative.
 In particular, $S(X,r)$ is not embedded in the group $G(X,r)$.
 \item
The algebra $A$ is not a domain.
 \end{enumerate}
 \end{corollary}

  \begin{remark} The lower bound
 $2|X|-1 \leq  \dim_{\textbf{\textbf{k}}} A_2$ is exact, whenever $|X|= p$, $p$ is a prime number. More precisely, a Dihedral quandle
 $(X,\la)$ of prime order $|X|=p$ satisfies condition \textbf{M}, see
Lemma \ref{Mremark}.
  \end{remark}

\section{Square-free braided sets $(X,r)$ satisfying the minimality condition}
\label{sec:minimalitycond}
\subsection{Square-free 2-cancellative quadratic sets $(X,r)$ with  minimality condition}

\begin{definition}
 \label{minimality_def}
 We say that a finite quadratic set $(X,r)$ \emph{satisfies the minimality condition} \textbf{M}
 if
 \begin{equation}
    \label{mindim_eq1a}
 \begin{array}{lclc}
 {\bf M :} \quad
& \dim_{\textbf{\textbf{k}}} A_2 =2|X|-1.&\quad&.
\end{array}
    \end{equation}
\end{definition}

\begin{example}
\label{Mexample}
Every square-free self distributive solution,  $(X,r)$
corresponding to a dihedral quandle of prime order $|X|= p>2$
satisfies the minimality condition \textbf{M}, see Corollary \ref{Mcorollary}.
\end{example}

Let $(X,r)$ be a square-free nondegenerate quadratic set set of order $|X| =n$.
Assume that $(X,r)$ is 2-cancellative.
Let $S=S(X,r)$ be its graded monoid, let $A = A(\textbf{k}, X, r)$
 be its graded $\textbf{k}$- algebra over a field $\textbf{k}$, and let $A^{!}$ be the Koszul dual algebra.
Consider the action of $\Dcal_3 (r)=  \: _{\rm{gr}} \langle r^{12}, r^{23}\rangle$ on $X^3$.
 The following useful remarks are straightforward.
\begin{remark}
\label{A!-zeroremark}
In assumption as above the following are equivalent:
\begin{enumerate}
\item
Each $\Dcal_3(r)$-orbit in $X^3$ contains a word of the type $xxy, x,y \in X$ (and $ztt, z,t \in X$).
\item
\label{McondA3prop3}
$A_3^{!}= 0$.
\end{enumerate}
\end{remark}

\begin{remark}
\label{Morbitsrem}
 Let $(X,r)$ be a square-free nondegenerate quadratic set of order $|X| =n$,
  assume that $(X,r)$ is 2-cancellative.
 Suppose $\cO$ is a nontrivial $r$-orbit in $X^2$ of order $|\cO|= n$.
 Then (i) for every $x\in X$ there exists $y \in X$ such that $xy \in \cO$; (ii) for every  $y \in X$
  there exists $x \in X$, such that $xy \in \cO$.
\end{remark}

\begin{proposition}
\label{prop:minimalityprop}
 Let $(X,r)$ be a square-free nondegenerate quadratic set  of order $|X| =n$, $X = \{x_1, \cdots, x_n\}$, $S=S(X,r)$,
$A = A(\textbf{k}, X, r)$, and $A^{!}$ in usual notation.   Assume that $(X,r)$ is 2-cancellative.
 Let $\cO_i, 1 \leq i \leq q$ be the set of all nontrivial $r$- orbits in $X^2$ (these are exactly the square-free $r$-orbits in $X^2$).
 \begin{enumerate}
 \item
 The following three conditions are equivalent.
 \begin{enumerate}
 \item[(a)]
 \label{prop:minimalityprop1a}
 $(X,r)$ satisfies the minimality condition \textbf{M}, (\ref{mindim_eq1a});
 \item[(b)]
 \label{prop:minimalityprop1b}
 Each non-trivial orbit $\cO_i$ has order $|\cO_i|= n.$
 \item[(c)]
 \label{prop:minimalityprop1c}
 The algebra $A$ has a finite presentation
 $A \cong \textbf{k}\langle X\rangle/(\textbf{R}_0)$, where
 $\textbf{R}_0$ is a set of exactly  $(n-1)^2$  quadratic square-free binomial relations:
 \begin{equation}
    \label{mindim_eq1ab}
  \textbf{R}_0 = \{x_{in}y_{in}- x_1x_i, \;  x_{in-1}y_{in-1}- x_1x_i,\;  \cdots, \;  x_{i2}y_{i2}- x_1x_i \mid 2 \leq i \leq n \},
    \end{equation}
where $x_{ij}\neq y_{ij}, 1 \leq i, j \leq n,$ and the following two conditions hold for every $2 \leq i \leq n$:
 \begin{enumerate}
 \item[(c1)] $x_{in}y_{in} >  x_{in-1}y_{in-1}> \cdots  > x_{i2}y_{i2} >  x_{i1}y_{i1}=x_1x_i$;
\item[(c2)]
there are equalities of sets
\[
 \begin{array}{lcl}
\{x_{ij}\mid 2 \leq j \leq n\} = X \setminus \{x_1\}, & &
\{y_{ij}\mid 2 \leq j \leq n\} = X\setminus \{x_i\}.
\end{array}
\]
\end{enumerate}
In this case, after a possible re-enumeration of the orbits, one has
\[\cO_i =\cO(x_1x_i) =\{x_1x_i := x_{i1}y_{i1} <  x_{i2}y_{i2} <  \cdots   < x_{in-1}y_{in-1}<x_{in}y_{in} \}, 2 \leq i \leq n.\]
 \end{enumerate}
 \item
 Moreover, each of conditions (1a), (1b), (1c) implies that
 \begin{enumerate}
 \item[(i)]
\label{McondA3prop3}
$A_3^{!}= 0$.
In particular,  $X^3$ does not contain square-free $\Dcal_3(r)$-orbits.
 \item[(ii)]
 $\gkdim A \leq 2.$
 \end{enumerate}
 \end{enumerate}
 \end{proposition}
\begin{proof}
Recall first that for arbitrary quadratic set $(X,r)$ the number of distinct words of length 2 in $S$ is exactly
 the number of $\Dcal_2(r)$-orbits in $X^2$,
 so one has
 \begin{equation}
    \label{orderS2eq}
  \dim A_2=  |S_2| = \sharp (\Dcal_2\text{-orbits in}\; X^2).
    \end{equation}
(1).
 (a) $\Longleftrightarrow$ (b).
 Suppose the minimality condition (\ref{mindim_eq1a}) holds. Then  $X^2$ splits into exactly
 $2n-1$ orbits, more precisely there are $n$ one element orbits, these are the elements of the diagonal $\Delta_2$
  and $n-1$ square-free orbit $\cO_i, 1\leq i \leq n-1$. Due to 2-cancellativity  one has $|\cO_i|\leq n$. At the same one has:
  \[n^2-n = |\bigcup_{1 \leq i \leq n-1} \cO_i | = \sum_{1 \leq i \leq n-1}|\cO_i| \leq n(n-1).\]
  Therefore $|\cO_i|=n$, for all $1 \leq i \leq n-1.$

  Conversely, suppose each nontrivial orbit $\cO$ has length $n$, and let $q$ be the number of square-free orbits $\cO.$
  There are exactly $(n-1)n$ square-free words in $X^2$, each is contained in some $\cO$, so
  $n(n-1) = n q$, and therefore $q=n-1$. Thus the total number of disjoint orbits in $X^2$ is
  $n + (n-1) = 2n-1$. It follows that $|S_2| = 2n-1$, and $\dim_{\textbf{k}} A_2= |S_2| = 2n-1$, so the minimality condition
  \textbf{M} holds.

  (b) $\Longrightarrow$ (c).
Suppose each non-trivial $\Dcal_2(r)$- orbit $\cO_i$ in $X^2$ has order $|\cO_i| = n, 1 \leq i \leq n-1$.
  It follows from Remark \ref{Morbitsrem}
  that for each $1 \leq i \leq n-1$ there exist unique $x \in X,$ such that $x_1x \in \cO_i.$  We re-enumerate the orbits (if necessary), so that
  $x_1x_i \in \cO_i.$
  Let $1 \leq i \leq n-1.$
  We order lexicographically the $n$ (distinct) words in $\cO_i$:
  \[\cO_i = \{ x_{in}y_{in} > x_{in-1}y_{in-1}> \cdots x_{i2}y_{i2} >x_{i1}y_{i1}=: x_1x_i\}.\]
  Each two monomials in $\cO_i$, considered as elements of $S$, are equal. This information is encoded by the set $\textbf{R}_i$
  of exactly $n-1$ "reduced" relations determined by $\cO_i$:
  \begin{equation}
    \label{mindim_eqR_i}
 \textbf{R}_i  : \quad x_{in}y_{in} = x_1x_i; \;  x_{in-1}y_{in-1}=  x_1x_i;\;  \cdots, \;x_{i2}y_{i2}=x_1x_i.
    \end{equation}
 As discussed in Section \ref{sec:quadratic} the set of defining relations $\Re(r)$ is equivalent to the set of reduced relations
 \[\textbf{R} = \bigcup_i \textbf{R}_i\]
 and the corresponding "algebra-type" relations are exactly the $n(n-1)$ reduced relations $\textbf{R}_0$ given in
 (\ref{mindim_eq1ab}).
It follows from the properties of the orbits $\cO_i$ that the relations in $\textbf{R}_0$ satisfy all additional conditions in part (c).

(c) $\; \Longrightarrow\;$   (b).
 The set of relations $\textbf{R}_0$ splits into $n-1$ disjoint subsets $\textbf{R}_i$, $1 \leq i \leq n-1.$
 Note that the properties of the relations given in part (3) imply that $(X,r)$ is 2-cancellative.
 It is clear that a relation $a = b \in \textbf{R}_i$ implies that $a, b $ belong to the same orbit $\cO$ in $X^2$.
We denote this orbit by $\cO_i$, one has $x_1x_i \in \cO_i$).
 One can also read off from the properties of $\textbf{R}_i$ that $\cO_i$ has exactly $n$-elemnts.
 Note that the sets $\Re_0(r)$ and $\textbf{R}_0$ generate the same two-sided ideal $I$ of $\textbf{k}\langle I\rangle$, so we get
    $A \cong \textbf{k}\langle X\rangle/ (\textbf{R}_0)$.
    It follows from the theory of Gr\"{o}bner bases that the ideal $I$ has unique reduced Gr\"{o}bner basis basis \textbf{GR}(I) (w.r.t $\leq$). Moreover, $\textbf{R}_0$ is a proper subset of
    $\textbf{GR}(I),$ (we assume $n \geq 3$) and all additional elements of $\textbf{GR}(I)\setminus \textbf{R}_0$ are homogeneous polynomials of degree $\geq 3.$
    Therefore, the set of normal monomials of length $2$:
    \[\Ncal_2 = \{x_1x_2, x_1x_3, \cdots , x_1x_n\}\bigcup\{x_ix_i\mid 1 \leq i \leq n\}\]
    projects to a $\textbf{k}$-basis of $A_2 \cong \textbf{k}S_2$, so this again implies $\dim_{\textbf{k}} A_2 = 2n-1$.

(2) Suppose $(X,r)$ satisfies the minimality condition \textbf{M}.
It follows from the argument in part (1) that
the normal basis $\Ncal$ of $A$ satisfies
\begin{equation}
    \label{Ncaleq}
 \Ncal \subseteq  \; \{x_1^{\alpha}x_i^{\beta}\mid 2 \leq i \leq n, \alpha \geq 0, \; \beta \geq 0 \}.
    \end{equation}
This implies that GK$\dim A \leq 2.$

We shall prove that  $A^{!}_3= 0$.
By Remark \ref{A!-zeroremark} it will be enough to show that each $\Dcal_3(r)$-orbit in $X^3$ contains a word of the type $xxy, x,y \in X$ (and $ztt, z,t \in X$).
Let $a,b,c \in X$. Without loss of generality, we may assume that $a\neq b$, and $b\neq c$.
Clearly, $bc \in \cO (bc) \subset X^2$, and
by Remark \ref{Morbitsrem}
the (square-free) orbit $\cO(bc)$ contains an element of the shape $at, t \in X$, thus $bc=at$ is an equality in $S_2$, so $abc = aat$ holds in $S_3$.
This implies that the $\Dcal_3$-orbit $\cO(abc)$ in $X^3$
contains the monomial $aat$.
It follows then that there are no square-free orbits in $X^3$, hence $A^{!}_3= 0$.
\end{proof}

Let $(X,r)$ be a quadratic set. A subset $Y\subset X$ is \emph{$r$-invariant} if $r(x,y) \in Y\times Y, \forall x,y \in Y$. In this case the restriction
$(Y, r_Y)$, where $r_Y:= r_{|Y\times Y}$  is a quadratic set.
A quadratic set $(X,r)$ is \emph{decomposable} if $X = Y\bigcup Z$ is a decomposition into nonempty disjoint $r$-invariant subsets.
Clearly, if $|Y|\geq 2$, then the restriction $(Y, r_Y)$ inherits from $(X,r)$ properties like nondegeneracy, 2-cancellativity, or being square-free.

\begin{definition}
\label{SDdef}
We call a quadratic set $(X,r)$ \emph{(left) self distributive}, or shortly \textbf{SD}, if it satisfies
\[{\bf SD  :} \quad r(x,y) = ({}^xy,x), \forall x,y \in X.\]
\end{definition}
\begin{lemma}
\label{minimaldim_lemma}
Suppose $(X,r)$ is a finite square-free nondegenerate quadratic set which is 2-cancellative, and satisfies the  minimality condition
\textbf{M}, let $|X|=n\geq 3$. Then $(X,r)$ is indecomposable. Moreover, if $(X,r)$ is a self distributive quadratic set, then for every $x\in X$ the permutation
$\Lcal_x$ has exactly one fixed point, so
\[\Lcal_x(x) = x, \quad \Lcal_x(y) \neq y, \forall x, y \in X, y \neq x.\]
\end{lemma}
\begin{proof}
 By Proposition \ref{prop:minimalityprop} $(X,r)$ satisfies the minimality conditions \emph{iff} the set of square-free words of length 2, $X^2 \setminus \Delta_2, \;$,
splits into $n-1$ disjoint  $\Dcal_2$-orbits  $\cO_i, 1 \leq i \leq n-1,$
each of which contains $n$ distinct monomials.

We shall prove first that $(X,r)$ is indecomposable.
Suppose $X = Y\bigcup Z$ is a decomposition into nonempty disjoint $r$-invariant subsets, say $|Y|=k \geq 2, |Z|= s \geq 1, k+s = n$.
The restriction  $(Y,r_Y)$
 is a nondegenerate, square-free and 2-cancellative quadratic set of order $k <n.$
 Let $x,y \in Y, x\neq y$, then the $\Dcal_2$-orbit $\cO (xy)$ in $X^2$ is
 contained entirely in $Y^2$, and by the 2-cancellativity of $(Y,r_Y)$, $|\cO (xy)| \leq k < n$. At the same time  $\cO (xy)$ is a $\Dcal_2$-orbit in $X^2$, so $\cO (xy) = \cO_i,$ for some $1\leq i \leq n-1$, and by Proposition \ref{prop:minimalityprop},  $\; |\cO (xy)|= |\cO_i|=n$, a contradiction.

 Suppose now that $(X,r)$ is a self distributive quadratic set with minimality condition. Let $x,y\in X, x\neq y.$ If we assume that $\Lcal_x(y) = y,$ then $r(xy) = yx$, and $r^2(xy) = r(yx) = {}^yxy$, hence (due to the -cancellativity) ${}^yx = x$. It follows that $\cO (xy)= \{xy, yx\}$, so  $|\cO (xy)| = 2 < n$, a contradiction with Proposition \ref{prop:minimalityprop}.
 \end{proof}

It follows from Corollary \ref{Mcorollary} that every square-free self distributive solution,  $(X,r)$
corresponding to a dihedral quandle of prime order $|X|= p>2$
satisfies the minimality condition \textbf{M}.
We do not know examples, where $(X,r)$ is a nondegenerate, square-free and 2-cancellative quadratic set of order $n \geq 3$, which satisfies the minimality condition \textbf{M},
but $(X,r)$ is not a solution of YBE.

Example \ref{order4} gives  a square-free braided set $(X,r)$ which is indecomposable (and injective), but does not satisfy the minimality condition \textbf{M}.

\begin{lemma}
\label{minimaldim_lemma2}
Suppose $(X,r)$ is a square-free self distributive quadratic set of finite order $|X|\leq 5$.
If $(X,r)$ is 2-cancellative and satisfies the minimality condition \textbf{M} then $(X,r)$ is a braided set.

More precisely, (up to isomorphism) either (a) $(X,r)$ is the quadratic set corresponding to the dihedral quandle of order $3$; or (b)
$(X,r)$ is the quadratic set corresponding to the dihedral quandle of order $5$.
\end{lemma}
\begin{sketchofproof}
Our assumptions imply that $(X, r)$ is nondegenerate and $\Lcal_x(y) \neq y, \forall x,y, \in X, x\neq y$, see Lemma \ref{minimaldim_lemma}.
The statement is straightforward for $|X|=3$. If $|X|=4$, then, $\forall x \in X$, the map $\Lcal_x = (y_1\; y_2\;y_3)$ is a cycle of length $3$.
Then a single r-orbit of length $4$ determines each map $\Lcal_x$ uniquely, which on its turn determines $r$, and all $r$-orbits in $X^2$ uniquely. A combinatorial argument shows that a 2-cancellative square-free quadratic set  $(X,r)$ with $|X|=4$ and the minimality condition does not exist.
Assume $|X|=5$. Then each map $\Lcal_x$ is either of the shape (a) $\Lcal_x = (y_1\; y_2\;y_3\;y_4)$, a cycle of length 4, where $y_i \neq x, 1 \leq i \leq 4$, or
(b) $\Lcal_x = (y_1\; y_2)(y_3\;y_4)$ is a product of disjoint transpositions where $y_i \neq x, 1 \leq i \leq 4$.
 Using a combinatorial argument one shows that if some $\Lcal_x = (y_1\; y_2\;y_3\;y_4)$, then $(X,r)$ is not 2-cancellative. It follows that only case (b) is possible.
 Then using an argument similar to the proof of Proposition
 \ref{quandle5} one shows that $(X,r)$ is a braided set isomorphic to the dihedral quandle of order $5$.
 \end{sketchofproof}

\subsection{Some basics on indecomposable injective racks}

\begin{lemma}
\label{2cancelprop2}
Suppose $(X,r)$ is an \textbf{SD} quadratic set.
\begin{enumerate}
\item
If $(X,r)$ is 2-cancellative then
\begin{enumerate}
\item
$(X,r)$ is nondegenerate;
\item
${}^xx = x, \; \forall \;x \in X$, so $(X,r)$ is square-free;
\item
${}^yx = x \Longleftrightarrow {}^xy = y, \; x, y \in X.$
\end{enumerate}
\item
 $(X,r)$ is involutive \emph{iff}  $(X,r)$ is the trivial solution.
\item $(X,r)$ is a braided set \emph{iff} the condition \textbf{laut} holds:
\[{\bf laut(x, y, z)  :}\quad {}^{x}{({}^yz)}={}^{{}^xy}{({}^xz)}, \; \forall x,y,z \in X.\]
\end{enumerate}
\end{lemma}
Self distributive braided sets are closely related to racks.
We recall some basics on racks,  we follow \cite{AnGra}.

\begin{definition}
\label{rackdef}
\cite{AnGra}
 \emph{A rack} is a pair $(X, \la)$, where $X$ is a nonempty set,  and $\la: X\times X\longrightarrow X$ is a map (a binary operation on $X$) such that
\begin{enumerate}
\item[(R1)] The map $\varphi_i: X\longrightarrow X, \; x \mapsto i\la x $ is bijective for all $i \in X$, and
\item[(R2)] $i\la(j\la k) = (i\la j)\la(i\la k)$.
\end{enumerate}
\emph{A quandle} is a rack which also satisfies $i\la i = i$, for all $i \in X.$

\emph{A crossed set} is a quandle such that $j \la i =i,$ whenever $i \la j =j.$
\end{definition}

\begin{remark}
\label{rackremark}
Suppose $(X,r)$ is an \textbf{SD} quadratic set.
Define $\la: X\times X\longrightarrow X$ as  $x \la y := {}^xy$. It is clear that
$(X, \la)$ is a rack
\emph{iff}  $(X,r)$ is a nondegenerate braided set. Moreover,
$(X,r)$ is a square-free braided set \emph{iff}  $(X, \la)$ is a quandle.

Conversely, every rack $(X, \la)$ defines canonically a nondegenerate SD braided set $(X,r)$, where $r(x,y) = (x\la y, y)$.
 Moreover $(X,r)$ is square-free \emph{iff} $(X, \la)$ is a quandle.
 It follows from Lemma \ref{2cancelprop2} that every rack $(X,r)$ which is 2-cancellative is a quandle, moreover $(X,r)$ is a crossed set.

 To simplify notation and terminology,  "\emph{a self distributive nondegenerate  braided set $(X,r)$ }" will be referred to as  "\emph{a rack}" and if in addition $(X,r)$ is square-free it will be also referred to as
 "a quandle".
 Under this convention we shall keep our usual notation and shall write "${}^xy$", instead of "$x \la y$".
 In this case there is an equality of maps:
\[\varphi_x = \Lcal_x, \; \forall  x \in X.\]
 \end{remark}

\emph{The inner group}, Inn(X) of a rack $X$ is the subgroup of $\Sym(X)$ generated by all permutations $\Lcal_x, x \in X$.
A rack $(X,r)$ is \emph{faithful} if the map $X \longrightarrow$ Inn(X), $x \longmapsto \Lcal_x$ is injective.
In fact, $X$ is indecomposable if and only if Inn(X) acts
transitively on $X$.

\begin{example} (Dihedral quandles)
\label{dihedralquandledef}
Let $n$ be a positive integer. Over the ring $\mathbb{Z}/n\mathbb{Z}$ of integers mod $n$
define $x\la y=2x-y$. This is a quandle known as \emph{the dihedral quandle} of order $n$. This is an Alexander quandle, see for example \cite{Nel03}. If
we assume that n is odd, we can identify the elements of this
quandle with the conjugacy class of involutions of $D_n$, the dihedral group of order $2n$.
Classification of Alexander quandles of prime order $p$ can be found for example in \cite{Nel03}.
These are particular cases of affine racks. Let $X$ be an abelian group
and let $g$ be an automorphism of $X$. Then $x\la y=(1-g)(x)+g(y)$ is a rack, \emph{an
affine rack over $X$.}
\end{example}
The following results are extracted from \cite{GrHVendramin11}

\begin{fact}
\label{factquandles}
\begin{enumerate}
Suppose $(X,r)$ is a finite \textbf{SD} braided set, and assume that the corresponding rack  $(X,\la)$ is indecomposable. Then
\item  $(X,\la)$ is faithful  \emph{iff} it is injective, \cite[Lemma 2.10]{GrHVendramin11}.

Clearly, in this case the solution $(X,r)$ is also injective.
\item
Suppose $X= \{x_1, \cdots , x_n\}$.
Then all permutations $\Lcal_i, i \in X$ have the same order $m$.
Moreover, the equalities $x_i^m = x_j^m$ hold in $G_X$ for all , $1 \leq i, j \leq n$, see \cite[Lemma 2.18]{GrHVendramin11}.
\item
For all $x \in X$ the permutation $\Lcal_x$ has exactly $1+k_2$ fixed points, where $k_2$ is the number of elements
$j \in X$ such that $\cO(1,j)$ has $2$ elements, see \cite[Lemma 2.25.3]{GrHVendramin11}.
\end{enumerate}
\end{fact}
\begin{remark}
\label{Mcorollary}
Suppose that $(X,r)$ is an indecomposable quandle of order $|X|\geq 3$, and every nontrivial $r$-orbit $\cO \subset X^2$ has order $3 \leq |\cO|\leq |X|$.
Then
(i) for every $x\in X$
the permutation $\Lcal_x$ has unique fixed point, namely $x$.
(ii) Moreover, if $\Lcal_x^2 = \id, \forall x \in X$,   then $X$ has an odd order  $|X|=2k+1.$
In this case  $\Lcal_x$ is a product of $k$ disjoint transposition.
\end{remark}

\subsection{Quandles with minimality condition  \textbf{M}}

\begin{corollary}
\label{Mcorollary}
Suppose $(X,r)$ is a 2-cancellative  \textbf{SD}
braided set of finite order $n=|X| \geq 3$, assume that $\Lcal_x^2 = \id, \forall x \in X$.

If $(X,r)$ satisfies the minimality condition \textbf{M} (hence X is indecomposable), then (i) $X$ has an odd order, $n=2k+1$, and
(ii) each $\Lcal_x, x \in X$ is a product of $k$ disjoint transpositions.
\end{corollary}

The following result is well known to the experts.
\begin{lemma}
\label{Mremark}
If (X,r) is a dihedral quandle of prime order  $|X|=p$
then each non-trivial  $r$-orbit $\cO$ in $X^2$ has length exactly $p$.
\end{lemma}
\begin{proof}
Let $x,y \in X$, $x\neq y$. Then, by definition, $r(x,y)=(2x-y, x)$.  One
proves by induction that
\begin{equation}
\label{dihedralorbiteq}
r^k(x,y)=( (k+1)x-ky, kx-(k-1)y)
\end{equation}
But all maps $r^k$ are bijections and $(k+1)x-ky=x$ if and only if $k=0$ mod $p$,
which implies that the $r$-orbit $\cO(x,y)$ in $X^2$ has size $p$.
\end{proof}
Recall that the dihedral quandles, and the Alexander quandles are well known since decades. A classification of Alexander quandles of prime order $p$
 can be found for example in \cite{Nel03}.
The particular proof of Lemma \ref{Mremark} was kindly provided to us by Leandro Vendramin.
\begin{definition}
\label{Mremark2}
If (X,r) is a dihedral quandle of prime order  $|X|=p$, it is called \emph{an Alexander quandle}. It can be identified with the set of reflections of a regular
$p$-gon (elements of the dihedral group $D_{2p}$).
\end{definition}

\begin{corollary}
\label{Mcorollary}
Suppose that  $(X,r)$ is a square-free \textbf{SD} solution, corresponding to a dihedral quandle of order $|X|= p>2$,  where $p$ is a prime number.
Then $(X,r)$ satisfies the minimality condition \textbf{M}.
\end{corollary}

\subsection{Concrete examples of quandles}
We have applied our results on square-free solutions $(X,r)$ to find various examples, as solutions on the following natural problem.
\begin{problem}
\label{problemM}
\textbf{\emph{Given the following data:}}
(a) A set $X$ of odd cardinality $n = 2k+1$;
(b) a cyclic  permutation $r_0 \in \Sym (X^2 \setminus \Delta_2)$ of order $n$
\[\cO: a_1b_1 \longrightarrow_{r_0} a_2b_2 \longrightarrow_{r_0} \cdots \longrightarrow_{r_0}   a_{n} b_{n}\longrightarrow_{r_0}a_1b_1
\]
where $a_i\neq b_i, 1 \leq i \leq n$, $a_i \neq a_j, b_i \neq b_j$, whenever $i \neq j$, $1\leq i,j\leq n$.

\textbf{\emph{Find}} \emph{an extension} $r : X\times X  \longrightarrow X\times X $ of $r_0$
(equivalently, \emph{find all maps $\Lcal_x, x \in X$ explicitly}),
so that

(i) $(X,r)$ is a 2-cancellative square-free \textbf{SD}
\emph{quadratic set} (we do not assume $(X,r)$ is a solution) ;

(ii) $\Lcal_x^2 = \id, \forall  x \in X$.

\emph{Analyze the obtained quadratic set}. In particular, decide (a) whether this data determines an \textbf{SD} solution of YBE?
(b) If moreover, $n$ is a prime number
 and the quadratic set $(X,r)$ satisfies the minimality condition \textbf{M} does this imply that $(X,r)$ is a braided set?
\end{problem}

\begin{remark} In earlier versions of this paper Problem \ref{problemM} was posed for the case when  $(X,r)$ is a 2-cancellative square-free
self distributive \emph{braided set}, with $\Lcal_x^2 = \id$, see \cite[Problem 6.4.1]{GI19}.  Under the strong assumption that $(X,r)$ is an SD braided set this problem has been solved in \cite[Prop.6.2]{CJO19}.
It is interesting and more difficult to consider Problem \ref{problemM} in its present version, see also Open question in Subsec \ref{sec:open_questions}
\end{remark}


We give some concrete examples.
The first illustrates a solution of Problem \ref{problemM} on a quadratic set $(X,r)$ of order $5$.

\begin{proposition}
\label{quandle5}
Let $X$ be a set of order $|X|=5$, to simplify notation we set $X = \{ 1, 2, 3, 4, 5\}$, as it is often used for racks.
Suppose $(X,r)$ is a quadratic \textbf{SD} set, so $r(x,y)= ({}^xy, x), x,y \in X$,
and $\Lcal_x^2 = \id, \forall x \in X$. Suppose $\langle r\rangle$ has an orbit of length $5$,
say:
\begin{equation}
\label{orbit1}
\cO(12):\quad   54\longrightarrow_{r} 35 \longrightarrow_{r} 23 \longrightarrow_{r} 12 \longrightarrow_{r} 41 \longrightarrow_{r} 54.
\end{equation}
Then the following conditions hold.
\begin{enumerate}
\item
The orbit (\ref{orbit1}) determines the maps $\Lcal_i, i\in X$ (and $r$) uniquely, so that  $(X, r)$ is a nondegenerate 2-cancellative quadratic set with minimality condition. In this case the left actions are:
\begin{equation}
\label{orbit2}
\begin{array}{l}
\Lcal_1 = (2\;4)(3\;5),\quad  \Lcal_2 = (1\;3)(4\;5), \quad  \Lcal_3 = (2\;5)(1\;4)\\
\Lcal_4 = (1\;5)(2\;3), \quad \Lcal_5 = (3\;4)(1\;2).
\end{array}
\end{equation}
$X^2$ splits into $4$ $r$-orbits of length $5$: $\cO(1\;i), 2 \leq i \leq 5$, and 5 one-element orbits for the elements of the diagonal $\Delta_2$.
\item
Consider the degree-lexicographic order $\leq$ on $\asX,$ induced by  $1 < 2 < 3 < 4 < 5$ on $X$.
The set  of defining relations $\Re(r)$
reduces (and is equivalent) to the following set of 16 quadratic relations:
\begin{equation}
\label{defrel1}
\begin{array}{lllll}
\textbf{R}= \{ &54 = 12   &41 = 12 &35 = 12   &23 = 12  \\
          &53 = 14    &45 = 14 &32 = 14  &21 = 14 \\
          &52  =15    &43= 15  &31 = 15   &24 = 15\\
          &51  = 13   &42 = 13 &34 = 13   &25= 13\}.
\end{array}
\end{equation}

\item
Moreover, $(X,r)$ is a braided set isomorphic to the dihedral quandle of order $5$.
The solution $(X,r)$  is also injective.
\item
Let $A =A(\textbf{k}, X, r) = \textbf{k}\langle X ; \textbf{R}\rangle \cong \textbf{k}\asX/(I) $ be the associated  quadratic algebra graded by length. The ideal $I$ is generated by the set  \[\textbf{R}_0  = \{u-v \mid u = v \in \textbf{R} \} \subset \textbf{k}\asX.\]
   It is not difficult to find that the reduced Gr\"{o}bner  basis $\textbf{GB}(\textbf{R}_0)$ contains  $4$ additional relations:
   \[\textbf{GB}(\textbf{R}_0) = \textbf{R}_0 \bigcup \{155-122, \;144-122, \; 133-122, \; 1222-1112\}.\]
  It follows that $A$ is standard finitely presented.
  \item $A$ is left and right Noetherian.
 \item $GK\dim A= 1;$  $A_3^{!}= 0$
\item
The monoid $S$ is not cancellative, $S$ satisfies the relations (\ref{defrel1})
and the  following relations derived from the Gr\"{o}bner basis $\textbf{GB}(\textbf{R}_0)$
\[155=122, \;144=122, \; 133=122,\; 1222=1112. \]
\item
The group $G(X,r)$
satisfies the relations (\ref{defrel1})
which (only in the group case) give rise to the following new relations in $G$:
\[55 = 44= 33= 22= 11.\]
We have deduced these relations straightforwardly from the  Gr\"{o}bner basis (without using the theory of racks). Of course, they agree with Fact \ref{factquandles}
\end{enumerate}
\end{proposition}
\begin{sketchofproof}
Our assumption $\Lcal_x^2 = \id, \forall x \in X$, and (\ref{orbit1}) imply that the permutations $\Lcal_i, 1 \leq i \leq 5$ are products of disjoint cycles of the shape
\[
\begin{array}{l}
\Lcal_1 = (2\;4)\sigma_1,\quad  \Lcal_2 = (1\;3)\sigma_2, \quad  \Lcal_3 = (2\;5)\sigma_3\\
\Lcal_4 = (1\;5)\sigma_4, \quad \Lcal_5 = (3\;4)\sigma_5,
\end{array}
\]
where for each $1\leq i \leq 5$, $\sigma_i$ is either a transposition, or $\sigma_i= id_X$.
However, we assume that $(X,r)$ satisfies the minimality condition and therefore by Lemma
\ref{minimaldim_lemma}
$\Lcal_x(x) = x, \quad \Lcal_x(y) \neq y, \forall x, y \in X, y \neq x.$ Therefore the maps $\Lcal_x, x \in X$ are uniquely determined and satisfy
(\ref{orbit2}). It is not difficult to check that $(X,r)$ is a braided set, isomorphic to the dihedral quandle of order $5$.
One uses routine technique of computation with Gr\"{o}bner bases to verify the remaining properties of $(X,r)$ listed above.
\end{sketchofproof}

\begin{corollary}
\label{Corquandle5}
Suppose $(X,r)$ is a 2-cancellative \textbf{SD} quadratic set of order $|X|=5$, and $\Lcal_x^2 = \id, \forall x \in X$. The following conditions are equivalent.
\begin{enumerate}
\item $\Lcal_x(y) \neq y, \forall x, y\in X, x \neq y$; \item $(X,r)$ satisfies the minimality condition; \item $(X,r)$ is a braided set;
\item $(X,r)$ is isomorphic to the dihedral quandle of order $5$.
\end{enumerate}
\end{corollary}

\begin{example}
\label{TypeA_order9}
Suppose $(X,r)$ is an \textbf{SD} quadratic set of order $9$.
Assume that $(\Lcal_x)^2 = 1,  \forall x \in X$, and the map $r$
 has a concrete orbit of order $9$, say:
\begin{equation}
\label{orbit9}
\begin{array}{ll}
\cO(98) : \quad 98  &\rightarrow_{r} 79 \rightarrow_{r} 67 \rightarrow_{r} 56 \rightarrow_{r} 45\\
&\rightarrow_{r} 34\rightarrow_{r} 23
\rightarrow_{r}   12\rightarrow_{r} 81  \rightarrow_{r} 98
\end{array}
\end{equation}
Then the orbit (\ref{orbit9}) determines the permutations $\Lcal_x, x \in X$ uniquely, so that  $(X, r)$ is a nondegenerate 2-cancellative braided set, with $\Lcal_x^2 = 1, x \in X$.
More precisely,
\begin{enumerate}
\item
$(X,r)$ is a braided set \emph{iff}  the left actions satisfy
\begin{equation}
\label{orbit2a}
\begin{array}{l}
\Lcal_{9} = (1\;6)(2\;5)(3\;4)( 7\;8)\quad \Lcal_{8} = (1\;9)(2\;7)(3\;6)(4\;5)\\
\Lcal_{7} = (1\;4)(2\;3)(5\;8)( 6\;9) \quad \Lcal_{6} = (1\;2)(3\;8)(4\;9)(5\;7)\\
\Lcal_{5} = (1\;8)(2 \;9)(3 \;7)(4 \;6)\quad \Lcal_{4} = (1\;7)(2\;6)(3\;5)(8\;9)\\
\Lcal_{3} = (1\;5)(2\;4)(6\;8)(9\;7) \quad \Lcal_{2} = (1\;3)(4\;8)(6\;7)(5\;9)\\
\Lcal_{1} = (2\;8)(3\;9)(5\;6)(4\;7)
\end{array}
\end{equation}
\item
In this case the non-trivial $r$- orbits of $X^2$ are:
\begin{equation}
\label{orbit2a}
\begin{array}{ll}
\cO(12), \cO(13),\cO(15), \cO(16),  \cO(18), \cO(19)  \quad \text{these are $6$ orbits of order 9}\\
 \cO(14), \cO(41), \cO(36), \cO(63), \cO(29), \cO (92) \quad \text{these are $6$ orbits of order 3}.
\end{array}
\end{equation}
Moreover, $(X, r)$ decomposes as a union of three $r$ -invariant subsets
\[X = X_1\bigcup X_2\bigcup X_3, \quad X_1 = \{1, 4, 7\}, X_2 = \{3, 6, 8\}, X_3 = \{2, 5, 9\},\]
and $r$ induces maps
\[X_1\times X_2\longrightarrow X_3\times X_1\longrightarrow X_2\times X_3\longrightarrow X_1\times X_2 . \]
Each $(X_i, r_i), 1 \leq i\leq 3$ is an \textbf{SD} solution whose quandle is isomorphic to the dihedral quandle of order $3$.
 \end{enumerate}
\end{example}

Next we give an example of an indecomposable square-free solution $(X,r)$ of order
 $|X|= 4$ which fails to satisfy the minimality condition \textbf{M}.

 \begin{example}
\label{order4}
Suppose $(X,r)$ is a square-free quadratic \textbf{SD} set of order $|X|=4$, so
$r(x,y)= ({}^xy, x)$. We again simplify notation setting $X = \{ 1, 2, 3, 4\}$.
Suppose $\Lcal_4$ is not involutive, and $\langle r\rangle$ has an orbit of length $3$
say:
\begin{equation}
\label{orbit14}
\cO(24):\quad   43\longrightarrow_{r} 24 \longrightarrow_{r} 32 \longrightarrow_{r} 43.
\end{equation}

 The orbit (\ref{orbit14}) determines the maps $\Lcal_i, i\in X$ (and $r$)  uniquely, so that  $(X, r)$ is a 2-cancellative solution.
More precisely,
\begin{enumerate}
\item
$(X,r)$ is a braided set \emph{iff}  the left actions are:
\begin{equation}
\label{orbit24}
\begin{array}{l}
\Lcal_1 = (2\;3\;4),\quad  \Lcal_2 = (1\;4\;3), \\
\Lcal_3 = (1\;2\;4), \quad \Lcal_4 = (1\;3\;2).
\end{array}
\end{equation}
\item
In this case
 $X^2$ has 4 $r$-orbits of length $3$
 (it is easy to write them explicitly), and 4 one-element orbits for the elements of $\diag(X^2)$.
\item
We consider the degree-lexicographic order $\leq$ on $\asX,$ induced by   $1 < 2 < 3 < 4$ on $X$.
The set  of defining relations $\Re(r)$
reduces (and is equivalent) to the following set of 8 quadratic relations:
\begin{equation}
\label{defrel1a}
\begin{array}{lllll}
\textbf{R}= \{ &43 = 24   &32 = 24 &   42= 14   &21 = 14  \\
          &41 = 13    &34 = 13 &   31 = 12  &23 = 12.
\end{array}
\end{equation}

\item
Let $A =A(\textbf{k}, X, r) = \textbf{k}\langle X ; \textbf{R}\rangle \cong \textbf{k}\asX/(I) $ be the associated  quadratic algebra graded by length. The ideal $I$ is generated by the set  \[\textbf{R}_0  = \{u-v \mid u = v \in \textbf{R} \} \subset \textbf{k}\asX.\]
   It is not difficult to show that the reduced Gr\"{o}bner basis $\textbf{GB}(\textbf{R}_0)$ contains  $4$ additional relations:
   \[\textbf{GB}(\textbf{R}_0) = \textbf{R}_0 \bigcup \{244-133, \;224-122, \; 1444-1222, \; 1333-1222\}.\]
  It follows that $A$ is standard finitely presented.

  \item
  The set $\Ncal$ of normal (mod $I$) monomials, which projects to a $\textbf{k}$-basis of $A$ satisfies:
  \[
    \begin{array}{ll}
    \Ncal \supset &X\bigcup \{12, 13, 14, 24\}\bigcup\{112, 113, 114, 122, 124, 133, 144 \}\\
    &\bigcup\{1^k2^m, k\geq 1, m \geq 3\}\bigcup \{x^k\mid x \in X, k \geq 2\}.
\end{array}\]
\item $\dim_{\textbf{k}} A_2 = 8 > 2|X|-1$.
\item $GK\dim_ A= 2.$
\item $(X,r)$ is indecomposable, and injective, but $(X,r)$ does not satisfy the minimality condition \textbf{M}.
\item
The monoid $S$ satisfies the relations (\ref{defrel1a})
and also the  following relations derived from the Gr\"{o}bner basis
\[244=133, \;224=122, \; 1444=1222,\; 1333=1222. \]
In particular $S$ is 3-cancellative, but $S$ is not cancellative .
\item
The group $G(X,r)$
satisfies the relations (\ref{defrel1a})
which (only in the group case) give rise to the following new relations in $G:$ $444 = 333= 222=111$.
\end{enumerate}
\end{example}

\section{A class of special extensions}
\label{sec:specialext}

\begin{remark} Let $(X,r)$ be a quadratic set.  A permutation $\tau \in \Sym(X)$ is \emph{an automorphism
of $(X, r)$} (or shortly an \emph{r-automorphism}) if
$(\tau \times \tau) \circ  r = r \circ (\tau \times \tau)$. The
group of $r$-automorphisms of $(X, r)$ is denoted by $\Aut(X,
r)$.
\end{remark}

In the hypothesis of the following theorem $(X,r_X), (Y,r_Y)$ are most general disjoint braided sets. No restrictions like nondegeneracy
or 2-cancellativeness are imposed.
\begin{theorem}
\label{theor:irregularext} Let $(X,r_X),$ and $(Y,r_Y)$ be
disjoint braided sets, and  let $Z=X\bigcup Y$. Suppose $\sigma
\in \Sym (X), \;\sigma \neq 1$, $\tau \in \Sym (Y), \;\tau\neq 1$.
 Define a bijective map $r: Z\times Z
\longrightarrow Z\times Z$, as follows
 \[
 \begin{array}{l}
   r(y,x):=(\sigma(x),\tau(y)); \quad r(x,y):=(\tau(y),\sigma(x)),\;  \forall \; x\in X, y \in
Y.\\
r(x_1,x_2):= r_X(x_1,x_2),\; \forall
x_1,x_2\in X,\quad
 r(y_1,y_2):= r_Y(y_1,y_2),\; \forall y_1,y_2\in Y, \\
\end{array}
\]
Then $(Z,r)$ is a quadratic set which satisfies the following conditions.
\begin{enumerate}
\item \label{irregularext1} $(Z,r)$ is nondegenerate \emph{iff}
both $(X,r_X),$ and $(Y,r_Y)$ are nondegenerate. \item
\label{irregularext2} $(Z,r)$ is 2-cancellative \emph{iff} (i)
both $(X,r_X),$ and $(Y,r_Y)$ are 2-cancellative; and (ii) the
maps $\sigma$ and $\tau$ (considered as permutations) are products
of disjoint cycles of the same length $q$. Clearly, in this case
 $|\sigma|=|\tau|= q$.
\item
\label{irregularext3}
Suppose the conditions (\ref{irregularext2}) are satisfied.
For each pair $x\in X, y\in Y$ consider the $r$-orbit $\cO(xy) = \{r^k(xy)\mid k \geq 0\}$ in $Z^2$.
\begin{enumerate}
\item If $q$ is even,  $q = 2m,$ then $|\cO(xy)|= q$. In this case the order $|r|$ of the map $r$ is the least common multiple
of the three orders,  \textbf{LCM}$(|r_X|, |r_Y|, q$).
\item If $q$ is odd,  $q = 2m+1$, then $|\cO(xy)|= 2q.$ In this case the order $|r|$ of $r$ is the least common multiple
\textbf{LCM}$(|r_X|, |r_Y|, 2q)$.
\end{enumerate}
\item
\label{irregularext4}
 The quadratic set $(Z, r)$ is \emph{a regular extension} of $(X,r_X),$ and $(Y,r_Y)$,  in the sense of \cite{GIM08}  if and only if $\sigma^2 = \tau^2 = 1$.
Moreover, $(Z, r)$ is involutive \emph{iff} (i) $\sigma^2 = \tau^2 = 1$, and (ii) $(X,r_X)$ and $(Y, r_Y)$ are involutive.
\item
\label{irregularext5}
$(Z,r)$ obeys the YBE if and only if the following conditions hold:
\begin{enumerate}
\item $\sigma \in \Aut (X,r_X)$ and  $\tau \in \Aut (Y,r_Y)$;
\item The left and the right actions satisfy the following conditions.
\begin{equation}
\label{eqAut01}
\begin{array}{llll}
&\Lcal_{\sigma^2(x)} = \Lcal_{x},\quad& \Rcal_{\sigma^2(x)} = \Rcal_{x}&\text{hold in} \; (X,r_X), \;\forall x \in X.\\
&\Lcal_{\tau^2(y)} = \Lcal_{y}\quad &\Rcal_{\tau^2(y)} = \Rcal_{y}&\text{hold in} \; (Y,r_Y), \;\forall y \in Y.
\end{array}
\end{equation}
\end{enumerate}
In this case $(Z,r) = (X,r_X)\stu^{\ast} (Y,r_Y)$ is
a generalized strong twisted union of $X$ and $Y$,
see Definition \ref{STUgendef}.
\end{enumerate}
\end{theorem}
\begin{proof}
Parts (1), (2), (3) and (4) are easy, and we leave their proof to the reader.
We shall prove part (5).
Assume $(Z,r)$ obeys YBE.
We shall prove condition (a) and (b).
Consider the diagram
(\ref{ybediagramAut}), where $\alpha \in Y, y, z \in X$. This diagram contains elements of the orbit of
the monomial $\alpha yz \in Z^3$, under the action of the group
$\Dcal_3(r).$ All monomials occurring in this orbit are equal
elements of $S$
\begin{equation}
\label{ybediagramAut}
\begin{CD}
 \alpha yz @>r_{12}>> r(\alpha y)z=({{}^{\alpha}y}{\alpha}^y)z=\sigma(y)\tau(\alpha)z\\
@V  r_{23} VV @VV r_{23} V\\
  \alpha r(yz)=\alpha({{}^yz}y^z)@. (\sigma(y))r(\tau(\alpha)z)=(\sigma(y))(\sigma(z))\tau^2(\alpha)\\
@V r_{12} VV @VV r_{12} V\\
\kern -20pt  r(\alpha({{}^yz}))y^z=\sigma({}^yz)\tau(\alpha)y^z@.
\kern-80pt  \quad\quad\quad\quad \quad \quad \quad  ({}^{\sigma(y)}{\sigma(z)})( \sigma(y)^{\sigma(z)})\tau^2(\alpha)\\
  @V r_{23} VV  \\
 \sigma({}^yz)\sigma(y^z)\tau^2(\alpha)@.
 \end{CD}
\end{equation}
Therefore,
\begin{equation}
\label{eqAut}
\begin{array}{ll}
r_{12}r_{23}r_{12}(\alpha yz) &= r_{23}r_{12}r_{23}(\alpha yz), \quad \forall\; \alpha\; \in Y, \;y, z \in X
\\
&\Longleftrightarrow (\sigma \times \sigma)\circ r_X (yz) =r_X\circ (\sigma \times \sigma)(yz),\; \forall y,z \in X \\
&\Longleftrightarrow \sigma \in \Aut(X, r_X).
\end{array}
\end{equation}

Similarly, a diagram starting with an arbitrary monomial of the shape $x\alpha\beta$, where $x \in X$, $\alpha,\beta \in Y$
shows that
\begin{equation}
\label{eqAut2}
\begin{array}{ll}
r_{12}r_{23}r_{12}(x\alpha\beta) &= r_{23}r_{12}r_{23}(x\alpha\beta)\quad \forall\; x \in X, \alpha,\beta \in Y\\
                                 &\Longleftrightarrow (\tau \times \tau)\circ r_Y (\alpha\beta)
 =r_Y\circ (\tau \times \tau)(\alpha\beta), \; \forall   \alpha,\beta \in Y \\
 &\Longleftrightarrow \tau \in \Aut(Y, r_Y).
\end{array}
\end{equation}

We have proven (a). Next we shall prove  (\ref{eqAut01}). Consider the following diagram:
\begin{equation}
\label{ybediagramAut2}
\begin{CD}
 x\alpha y @>r_{12}>>  \tau (\alpha)\sigma(x)y      \\
@V  r_{23} VV @VV r_{23} V\\
 x \sigma(y)\tau( \alpha) @.    \tau (\alpha)({}^{\sigma(x)}y) (\sigma(x)^y) \\
@V r_{12} VV @VV r_{12} V\\
\kern -20pt {}^x {\sigma(y)} x^{\sigma(y)} \tau( \alpha) @.
\kern-80pt \quad\quad\quad\quad \quad\quad\quad  \sigma(({}^{\sigma(x)}y)) \tau^2 (\alpha)  (\sigma(x)^y)  \\
  @V r_{23} VV  \\
 {}^x {\sigma(y)}\tau^2( \alpha) \sigma(x^{\sigma(y)})  @.
 \end{CD}
\end{equation}

The following implication holds:
\begin{equation}
\label{eqAut3}
\begin{array}{ll}
&r_{12}r_{23}r_{12}(x\alpha y) = r_{23}r_{12}r_{23}(x\alpha y),\quad \forall \; x, y\in X, \alpha \in Y\\
&\Longleftrightarrow
\sigma(({}^{\sigma(x)}y)) = {}^x {(\sigma(y))} \;\text{and}\; \sigma(x^{(\sigma(y))}) = \sigma(x)^y,\;  \forall\; x,y \in X.
\end{array}
\end{equation}

But $\sigma\in \Aut(X, r_X),$ so it follows from  (\ref{eqAut3}) that
\[
 {}^x {(\sigma(y))} = \sigma(({}^{\sigma(x)}y))= {}^{\sigma(\sigma(x))}{\sigma(y)} = {}^{(\sigma^2(x))}{(\sigma(y))}, \forall x,y \in X.
\]
The map $\sigma: X \longrightarrow X$ is bijective, hence
\[
{}^{(\sigma^2(x))} z = {}^xz, \; \forall x,z \in X,
\]
which is equivalent to
\[\Lcal_{\sigma^2(x)} = \Lcal_{x} \;\;\text{holds in} \;\; (X, r_X), \;\forall x \in X. \]
Similarly, the equalities  \[(\sigma(x))^{(\sigma^2(y))} = \sigma(x^{\sigma(y)}) = (\sigma(x))^y, \forall x,y \in X\]
are equivalent to
\[\Rcal_{\sigma^2(y)} = \Rcal_{y} \;\;\text{holds in} \;\; (X,r_X), \;\forall y \in X. \]
This proves the first two equalities in (\ref{eqAut01}).
Analogous argument proves the remaining  equalities in (\ref{eqAut01}).
We have shown that if $(Z, r)$ obeys YBE then conditions (a) and (b) hold.
Conversely, assume that conditions (a) and (b) are satisfied. The above discussion implies easily that $(Z,r)$ is a solution of YBE. This proves part (5).

It is clear that our construction gives a particular case of a generalized strong twisted union
$(Z,r) = (X,r_X)\stu^{\ast} (Y,r_Y)$, see Definition \ref{STUgendef}.
\end{proof}

One  may apply the results of Theorem \ref{stuthm} and get more information on the braided monoid $S(Z,r)$ and the braided group $G(Z,r)$.

To construct  concrete extensions via automorphisms, and also for some kind of classification of this type of extensions
it may be practical to use results from \cite{Chouraqui}.

Clearly, if $(X,r)$ is a trivial solution, then $\Aut (X,r) = \Sym(X)$ and for every $\sigma \in \Sym (X)$ there is an equality
$\Lcal_{\sigma^2(x)} = \Lcal_{x} = \id_X.$ Hence we have at disposal
an easy method to construct nondegenerate 2-cancellative square-free braided sets $(Z,r)$,
$Z = X\bigcup Y$ , where the order of the map $r$ may vary
as $2 \leq |r| \leq |Z|$.

\begin{corollary}
\label{irregularextcor}
Let $(X,r_X),$ $(Y,r_Y)$ be disjoint trivial symmetric sets, $|X|= m, |Y|= n$, say $m \leq n$.  Let $Z=X\bigcup Y.$
Suppose $\sigma \in \Sym (X)$, $\tau \in \Sym (Y).$
 Define $r: Z\times Z
\longrightarrow Z\times Z$, as follows
 \[
 \begin{array}{l}
   r(x_1,x_2):= r_X(x_1,x_2)=(x_2,x_1),\; \forall
x_1,x_2\in X,\quad
 r(y_1,y_2):= r_Y(y_1,y_2)=(y_2,y_1),\; \forall y_1,y_2\in Y, \\
 r(x,y):=(\tau(y),\sigma(x)),\quad  r(y,x):=(\sigma(x),\tau(y)); \;\forall \; x\in X,\; y \in
Y.
\end{array}
\]
\begin{enumerate}
\item $(Z,r)$ is a nondegenerate square-free braided set.
\item
Moreover, $(Z,r)$ is 2-cancellative \emph{iff} the permutations $\sigma$ and  $\tau$ are products of disjoint cycles of the same length
$q \leq m$, in particular,
$|\sigma|= |\tau| = q$; In this case either (a) $q$ is even, and $|r|= q$; or (b)  $q$ is odd  and $|r|= 2q$.
\end{enumerate}
\end{corollary}

\begin{example}
\label{irregularextcor}
Let $X = \{x_1, x_2, x_3\}$, $Y = \{y_1, y_2, y_3\}$, be disjoint sets and let $(X,r_X),$ $(Y,r_Y)$ be trivial solutions.
Set  $\sigma = (x_1\;x_2\;x_3) \in \Sym (X)$, $\tau = (y_1\;y_2\;y_3) \in \Sym (Y).$
 Define $r: Z\times Z
\longrightarrow Z\times Z$, as follows
 \[
 \begin{array}{ll}
 r(y,x)=(\sigma(x),\tau(y)), & r(x,y)=(\tau(y),\sigma(x)),\;  \forall \; x\in X,\; y \in
Y;\\
   r(x_i,x_j)= (x_j, x_i),
 & r(y_i,y_j)= (y_j, y_i),\;
1\leq i, j\leq 3.
\end{array}
\]
Then
 $(Z,r)$ is a nondegenerate square-free braided set of order $|Z|=6$, $Z = X\stu^{\ast}Y$.
$(Z,r)$ is 2-cancellative,  the order of $r$ is $|r|= 6= |Z|.$ The algebra $A= A(\textbf{k}, Z, r)$ satisfies $\dim A_2 = 2\binom{3+1}{2}+3 = 15$.
More detailed computation shows that the associated graded algebra $A$ does not have a finite Gr\"{o}bner basis with respect to any ordering
of $Z$.
\end{example}

\section{The braided monoid $S(X,r)$ and  extensions of solutions}
\label{BraidedMonoidSec}
\subsection{The braided monoid $S(X,r)$ of a braided set $(X,r)$}
\label{BraidedMonoid1}
In \cite{GIM08} we introduced the notion of \emph{a braided monoid} analogously to the term `braided group' in the
 sense of  \cite{Takeuchi}, \cite{LYZ}.
We recall some definitions and results from \cite{GIM08}.

To each braided set $(X,r)$ with $S=S(X,r)$  we  associate a
matched pair $(S,S)$  with left and right actions
uniquely determined by $r,$
which defines a unique `braided monoid'
$(S, r_S)$ associated to $(X,r).$
This is not a surprise given the analogous results for
 the group $G(X,r)$, see \cite{LYZ}, but our approach is necessarily  different. In fact we first construct
the matched pair of monoids which is a
self-contained result and then consider the map $r_S: S\times S \longrightarrow S\times S $, see \cite[Theorem 3.6]{GIM08}. We prove, see
\cite[Theorem 3.14]{GIM08}, that $r_S$ is bijective and obeys the
YBE (as would be true in the group case), moreover, we show that
$(S, r_S)$ is a graded braided monoid.

The reader should  be aware that due to the possible
lack of cancellation in $S$ the proofs of our results for monoids   are
difficult and necessarily involve different computations and combinatorial
arguments. In general, the results can not be extracted from the
already known results for the group case. Nevertheless, the monoid case is the one naturally arising in this context. Both the monoid  $S(X,r)$ and the
quadratic algebra  $A = A(\textbf{k},X,r)$ over a field $\textbf{k}$ are of particular interest.
The theory of general braided monoids $(S, r_S)$ gives interesting classes of braided objects. However it seems that the approach to these is different and  more difficult from the approach to braided groups (equivalently skew braces).
We recall some basic definitions.

\begin{definition} \label{MLaxioms} \cite{GIM08}
The pair $(S,T)$ is a matched pair of monoids if $T$
 acts from the left on $S$ by ${}^{(\ )}\bullet$ and $S$ acts on $T$
 from the right by $\bullet^{(\ )}$ and these two actions obey
\[\begin{array}{lclc}
{\bf ML0}:\quad & {}^a1=1,\quad  {}^1u=u;\quad &{\bf MR0:} \quad &1^u=1,\quad a^1=a \\
 {\rm\bf ML1:}\quad& {}^{(ab)}u={}^a{({}^bu)},\quad& {\rm\bf MR1:}\quad  & a^{(uv)}=(a^u)^v \\
{\rm\bf ML2:}\quad & {}^a{(u.v)}=({}^au)({}^{a^u}v),\quad &{\rm\bf
MR2:}\quad & (a.b)^u=(a^{{}^bu})(b^u),
\end{array}\]
for all $a, b\in T, u, v \in S$.
\end{definition}

 \begin{definition}
 \label{braidedmonoiddef}  \cite{GIM08}
 An \emph{{\bf M3}-monoid} is a monoid $S$ forming part of a matched pair $(S,S)$
 for which the actions are such that
 \[ {\rm\bf M3}:\quad {}^uvu^v=uv\]
holds in $S$ for all $u,v\in S$. We define the \emph{associated map}   $r_S: S\times S\to S\times S$  by  $r_S(u,v)= ({}^uv,u^v).$
 A \emph{braided monoid}  is an {\bf M3}-monoid $S$,  where
 $r_S$ is bijective and obeys the YBE.
 \end{definition}

\begin{fact}(\cite{GIM08}, Theor. 3.6, Theor. 3.14.)
\label{theoremA}
Let  $(X,r)$ be a braided set and
$S=S(X,r)$ the associated monoid.
Then
\begin{enumerate}
\item
The left and the
right actions
$
{}^{(\;\;)}{\bullet}: X\times X  \longrightarrow
 X  , \; \text{and}\;\; \bullet^{(\;\;)}: X
\times X \longrightarrow  X
$
defined via $r$ can be extended in a unique way to a left and a
right action
\[{}^{(\;\;)}{\bullet}: S\times S  \longrightarrow
 S  , \; \text{and}\;\; \bullet^{(\;\;)}: S
\times S \longrightarrow  S.
\]
which make $S$ a strong graded {\bf M3}-monoid, in particular, $(S, r_S)$ is a set-theoretic solution of YBE. The associated
bijective map $r_S$ restricts to $r$.
\item Moreover, the following conditions hold.
\begin{enumerate}
\item
\label{theoremAextra1}
$(S,r_S)$ is a graded braided
monoid, that is the actions agree with the grading of $S$:
$|{}^au|= |u|= |u^a|,$ for all $a,u \in S$.
 \item
 \label{theoremAextranondeg}
$(S, r_S)$ is a nondegenerate solution of YBE \emph{iff} $(X,r)$ is nondegenerate.
\item
\label{theoremAextrainvol}
$(S, r_S)$ is involutive \emph{iff}
$(X,r)$ is involutive.
\end{enumerate}
\end{enumerate}
\end{fact}
Suppose $(X,r)$ is a noninvolutive solution. The set $X$ is always embedded in the braided monoid $(S, r_S)$. Moreover, in contrast with the group $G(X,r)$, the monoid $S$ preserves more detailed information about the solution $(X,r)$. In particular, there is an equality $u=v$ in $S$, if and only if $|u|= |v|= m$, and $u$ and $v$ are in the same $\Dcal_m(r)$-orbit in $X^m$. In general, this is not true in $G(X,r),$ where a great portion of information about $(X,r)$ is lost.

\begin{corollary}
\label{cor_SDmonoids}
Suppose $(X,r)$ is a self distributive braided set, $S=S(X,r)$, $G=G(X,r)$. Then
\begin{enumerate}
\item[(i)]
the braided monoid $(S, r_S)$ is a self distributive solution;
\item[(ii)]
the  braided group $(G, r_G)$ is self distributive.
\end{enumerate}
\end{corollary}

\subsection{General extensions of braided sets}
\begin{definition}
\label{extensiongeneraldef} Let $(X,r_X)$ and
$(Y,r_Y)$ be disjoint quadratic sets. Let $(Z,r)$ be a set with a
bijection $r: Z\times Z\longrightarrow Z\times Z.$ We say that
$(Z,r)$ is \emph{a (general)  extension of} $(X,r_X),(Y,r_Y),$ if
$Z= X\bigcup Y$ as sets, and  $r$ extends the maps $r_X$ and
$r_Y,$ i.e. $r_{\mid X^2}= r_X $, and $r_{\mid Y^2}=r_Y.$ Clearly
in this case $X, Y$ are $r$-invariant subsets of $Z$. $(Z,r)$ is
\emph{a YB-extension of} $(X,r_X)$, and $(Y,r_Y)$ if $r$ obeys YBE.
\end{definition}

\begin{remark}
\label{extensionsrem} In the assumption of the above definition,
suppose $(Z,r)$ is \emph{a nondegenerate}  extension of
$(X,r_X),(Y,r_Y).$ Then the equalities $r(x,y) = ({}^xy,x^y),$
$r(y,x) = ({}^yx,y^x),$ and the nondegeneracy of $r$, $r_X,$
$r_Y$ imply that
\[
{}^yx, x^y \in X, \;\;\text{and }\;\; {}^xy, y^x \in Y,\;\;
\text{for all}\;\; x \in X, y\in Y.
\]
Therefore, $r$ induces bijective maps
\begin{equation}
\label{rhosigma} \rho: Y\times X \longrightarrow X\times Y ,  \;
\text{and} \;\sigma: X\times Y \longrightarrow Y\times X,
\end{equation}
and left and right ``actions"
\begin{equation}
\label{ractions1}
{}^{Y}{\bullet}: Y\times X \longrightarrow
X,\;\;\; {\bullet}^{X}: Y\times X \longrightarrow Y,\;
\text{projected from}\; \rho
\end{equation}
\begin{equation}
\label{ractions2} {}^{X}{\bullet}: X\times Y \longrightarrow
Y,\quad  {\bullet}^{Y}: X\times Y \longrightarrow X, \
\text{projected\ from}\; \sigma.
\end{equation}
Clearly,  the 4-tuple of maps $(r_X, r_Y, \rho, \sigma)$ uniquely
determine the extension $r.$ The map $r$ is also uniquely
determined by $r_X$, $r_Y$, and the maps (\ref{ractions1}),
(\ref{ractions2}).
\end{remark}
We call the actions (\ref{ractions1}) and
(\ref{ractions2})  projected from $r_{|Y\times X}$ and $r_{|X\times Y}$ \emph{the associated ground actions}.

\begin{lemma}
\label{extlemma}
Suppose  $(Z,r)$ is a  nondegenerate braided set
which splits
 as a disjoint union $Z = X\bigcup Y$ of two $r$-invariant subsets $X$ and $Y$. Denote
by $(X,r_1)$ and $(Y, r_2)$ the induced sub-solutions.  The
following conditions hold.
\begin{enumerate}
\item
\label{a)}
The assignment
$\alpha\longrightarrow {}^{\alpha}\bullet= \Lcal_{\alpha|X}$
extends to a left action of the associated monoid $S_Y$ on $X$, and induces a left action of $G_Y$ on $X$.
The assignment $\alpha
\longrightarrow   {\bullet}^{\alpha}= \Rcal_{\alpha|X}$
extends to a right action of  the associated monoid $S_Y$ on $X$, and induces a right action of $G_Y$ on $X$.
\item
The assignment $x
\longrightarrow {}^{x}\bullet= \Lcal_{x|Y}$
extends to a left action of  the associated monoid $S_X$ on $Y$, and induces a left action of $G_X$ on $Y$.
The assignment $x
\longrightarrow   {\bullet}^{x}=\Rcal_{x|Y}$
extends to a right action of the associated monoid $S_X$ on $Y$ and induces a right action of $S_X$ on $Y$.
\item
Moreover, if the braided set $(Z,r)$ is injective (that is the natural map $Z\longrightarrow G_Z$ is embedding) then each of the assignments in part (a) extends to an action
of $G_Y$ on $X$, and each of the assignments in part (b) extends to an action of $G_X$ on $Y$.
\end{enumerate}
\end{lemma}

Recall that in \cite{GIM08}  a (general)  extension $(Z,r)$ of $(X,r_X),(Y,r_Y)$
is called \emph{a regular extension} of $(X,r_X)$, and $(Y,r_Y)$ if $r$ is bijective, and the
restrictions $r_{\mid Y\times X}$ and $r_{\mid X\times Y}$ satisfy
\[(r\circ r)_{\mid Y\times X}= \id_{\mid Y\times X}, \quad (r\circ r)_{\mid
X\times Y}= \id_{\mid X\times Y},\] but $r$ is not necessarily
involutive on $X\times X,$ neither on $Y\times Y.$
Regular extensions of arbitrary braided sets were introduced and studied in \cite{GIM08}, where the theory of matched pairs of monoids
 was applied to characterize regular extensions and their monoids.
A regular extension $(Z,r)$ of two involutive solutions is also involutive.
The extensions constructed in Section \ref{sec:specialext} are not regular.

In this paper we have a
particular interest in \emph{noninvolutive} nondegenerate braided sets $(Z,r)$, and
it is natural to search for methods proposing constructions of "new" solutions using already known braided sets.
 We have shown in
Section \ref{sec:specialext}, see Theorem \ref{theor:irregularext}
that one can construct new noninvolutive solutions $(Z,r)$ with a
prescribed orders $|Z|$, and $|r|$ using general (non-regular) extensions of
well-known involutive solutions.
 So it is natural to study general extensions $(Z,r)$, possibly \emph{not regular} in the sense of \cite{GIM08}).
 In notation and assumptions as above, let $(Z,r)$ be a nondegenerate braided set which is an extension of the disjoint braided sets
 $(X,r_X), (Y,r_Y)$.
 Denote $S = S(X,r_X), T= S(Y,r_Y), U = S(Z,r)$.
 It follows from  Fact \ref{theoremA} that $U = S(Z,r)$ has the structure of a graded braided monoid $(U, r_U)$
 with a braiding operator $r_U$ extending $r$, moreover $(U, r_U)$ is an extension of the disjoint braided monoids $(S, r_S)$ and $(T, r_T)$, and one can apply the theory of matched pairs of monoids to give more detailed description of the behaviour of the matched pairs  (S,T), (T,S), $(U,U)$, etc,  in the spirit of the results in \cite{GIM08}).
We propose an explicit construction- \emph{generalized strong twisted unions of braided sets}.

\subsection{Generalized strong twisted unions of nondegenerate braided sets}
\label{subsec:StuBraidedsets}
Theorem \ref{theor:irregularext} gives a method to construct a new type of extensions of braided sets. The properties of these extensions motivate
our Definition \ref{STUgendef} of \emph{generalized strong twisted unions of solutions}
which is a generalization of the notion of a strong twisted union of solutions, see Definition 5.1. \cite{GIM08}. According the "old" definition, the notion of a strong twisted union is restricted only to \emph{regular extensions}.
Note that a strong twisted union $(Z,r)$ of solutions
$(X,r_X)$ and $(Y, r_Y)$ does not necessarily obey YBE, but if $(Z,r)= X\stu Y$ is (a regular) extension of symmetric sets and obeys the YBE,
then $(Z,r)$ is also a symmetric set ($r^2=1$).

In our new settings if $(X,r_X)$ and $(Y, r_Y)$ are symmetric (or
braided) sets with $|X| > 2, |Y|> 2$, we construct extensions $(Z,
r)$ which are braided sets (satisfy YBE),  but the solution $r$
may have order $>2$, see for example Section \ref{sec:specialext}
and the results therein.

\begin{definition}
\label{STUgendef} Suppose  $(X,r_X)$ and $(Y, r_Y)$ are disjoint
quadratic sets. We call an extension $(Z,r)$
 \emph{a generalized strong twisted union} of $(X,r_X)$ and $(Y, r_Y)$, and write $Z = X\stu^{\ast} Y$ if the
ground actions satisfy
 \begin{equation}
\label{stueq1}
 \begin{array}{lll}
 {\rm\bf stu 1 :} &{}^{{\alpha}^y}x  =  {}^{\alpha}x;
 \quad\quad
 {\rm\bf stu 2 :} & x^{{}^y{\alpha}} = x^{\alpha};\\
 &&\\
 {\rm\bf stu 3 :} &{}^{x^{\beta}}{\alpha}  =  {}^x{\alpha}; \quad\quad
 {\rm\bf stu 4 :} & {\alpha}^{{}^{\beta}x}  =  {\alpha}^x, \;
\end{array}
\end{equation}
for all   $x, y \in X, \; \alpha,\beta \in Y$.

 We define a generalized strong
 twisted union of more than two quadratic sets, analogously to
\cite{GIC}, Definition 3.5.
Let $(Z, r )$ be a nondegenerate quadratic set of
arbitrary cardinality, let $X_i$ , $i \in I$, be a set of
pairwise disjoint $r$-invariant proper subsets of $Z$, where $I$ is a set of indices, $|I|\geq 2$.
We say that $(Z, r )$ is \emph{a generalized strong twisted union of
$X_i, i \in I$}, and write $Z = \stu^{\ast}_{i\in I} X_i$ if
$Z = \bigcup_{i\in I} X_i,$
and for each pair $i, j\in I , i\neq j,$ the $r$-invariant subset $X_{i j} = X_i\bigcup X_j$ is a generalized
strong twisted union, $X_{i j}=X_i\stu^{\ast} X_j$. In the
particular case, when $I$ is a finite set, $I= \{1 \leq i \leq  m\}$,
we write $X = X_1 \stu^{\ast} X_2 \stu^{\ast}\cdots
\stu^{\ast}X_m$.
 \end{definition}

 \begin{lemma}
 \label{SD_stu_lemma}
 Suppose
 $(X,r)$ is an \textbf{SD} nondegenerate braided set i.e. $r(x,y)= ({}^xy, x), \forall x,y \in X$ (so $(X,\la)$ is a rack).
 If $(X, r)$ decomposes as a union of disjoint $r$-invariant subsets $X= \bigcup_{1 \leq i \leq m} X_i,$ then $X$ is a generalized strong twisted union of racks, $X = X_1 \stu^{\ast} X_2 \stu^{\ast}\cdots
\stu^{\ast}X_m$, where for $1 \leq i \leq m$, $(X_i, r_i)$ is the corresponding subsolution.
\end{lemma}

\begin{lemma}
\label{lemmixedactions}
  Let $(Z,r)$ be a nondegenerate quadratic set, which splits as a disjoint union $Z=X\bigcup Y$ of its $r$-invarint subsets $(X,r_X), (Y, r_Y)$, so $Z$ is an extension of $X$ and $Y$. Suppose $x,y \in X, \ \alpha \in Y.$ Each two of the following conditions imply the third.
\[
\begin{array}{lll}
(1)&  {\bf l1(\alpha, y, x)    :}\quad& {}^{\alpha}{({}^yx)}={}^{{}^{\alpha}y}{({}^{{\alpha}^y}{x})}   \\
&&\\
(2)&  {\bf laut(\alpha, y, x)  :}\quad& {}^{\alpha}{({}^yx)}={}^{{}^{\alpha}y}{({}^{\alpha}x)} \\
&&\\
(3)&  {\bf stu1 (\alpha, y, x) :}\quad&  {}^{{\alpha}^y}x  =  {}^{\alpha}x.
\end{array}
\]
\end{lemma}

Let  $(Z,r)$ be a  nondegenerate braided set
which split as a disjoint union $Z = X\bigcup Y$ of two $r$-invariant subsets $X$ and $Y$, $G_Z=G(Z,r).$
 Denote by $(X,r_X)$ and $(Y, r_Y)$ the induced
 subsolutions. Due to the nondegeneracy of $r$, each of the sets $X$ and $Y$ is invariant under the left action of
 $G_Z$ on $Z$, similarly,
 $X$ and $Y$ are invariant under the right action of $G_Z$ on $Z$. (We call such sets  \emph{$G$-invriant}).
 Let $\alpha \in Y$, and let $\Lcal_{\alpha}$ be the corresponding left action on $Z$. Denote by
 $\Lcal_{\alpha|X}, \alpha \in Y$  the restriction of $\Lcal_{\alpha}$ on $X$.
 The restrictions $\Rcal_{\alpha|X}$, $\Lcal_{x|Y}, \Rcal_{x|Y}$ are defined analogously for $x \in X$, and $\alpha \in Y$.
\begin{proposition}
\label{stuprop}
Suppose  $(Z,r)$ is a  nondegenerate braided set
which splits as a disjoint union $Z = X\bigcup Y$ of two $r$-invariant subsets $X$ and $Y$, denote
by $(X,r_1)$ and $(Y, r_2)$ the induced subsolutions. The following conditions hold.
\begin{enumerate}
\item
 $\Lcal_{\alpha|X} \in \Aut (X, r_1)$ if and only if
\[ {}^{\alpha}{({}^yx)}={}^{{}^{\alpha}y}{({}^{\alpha}x)},\;\;\text{and}\;\;
{}^{\alpha}{(x^y)}= ({}^{\alpha}x)^{({}^{\alpha}y)},\;\forall x,y \in X.\]
\item
$\Rcal_{\alpha|X} \in \Aut (X, r_1)$ if and only if
\[{({}^yx)}^{\alpha}={}^{(y^{\alpha})}{(x^{\alpha})}, \;\;\text{and}\;\;
{(x^y)}^{\alpha}= (x^{\alpha})^{(y^{\alpha})}, \forall  x,y \in X.
\]
\item  The following implications hold
\begin{equation}
\label{stueq2}
 \begin{array}{lll}
 {\rm\bf stu 1 :} &{}^{{\alpha}^y}x  =  {}^{\alpha}x,\; \forall \alpha\in Y, x,y \in X \Longleftrightarrow & \Lcal_{\alpha|X} \in \Aut (X, r_1),  \forall \alpha\in Y\\
{\rm\bf stu 2 :} & x^{{}^y{\alpha}} = x^{\alpha}, \;\forall \alpha\in Y, x,y \in X \Longleftrightarrow & \Rcal_{\alpha|X} \in \Aut (X, r_1),  \forall \alpha\in Y\\
 {\rm\bf stu 3 :} &{}^{x^{\beta}}{\alpha}  =  {}^x{\alpha}, \; \forall x \in X, \alpha, \beta\in Y  \Longleftrightarrow & \Lcal_{x|Y} \in \Aut (Y, r_2),  \;\forall x\in X\\
 {\rm\bf stu 4 :} & {\alpha}^{{}^{\beta}x}  =  {\alpha}^x,\;\forall x \in X, \alpha, \beta\in Y  \Longleftrightarrow & \Rcal_{x|Y} \in \Aut (Y, r_2),  \;\forall x\in X.
\end{array}
\end{equation}
\end{enumerate}
\end{proposition}
\begin{proof}
(1). By definition $\Lcal_{\alpha|X} \in \Aut (X, r_X)$ \emph{iff} \[(\Lcal_{\alpha|X}\times \Lcal_{\alpha|X})\circ r =
r\circ (\Lcal_{\alpha|X}\times \Lcal_{\alpha|X}),\] so part (1) follows straightforwardly  from the equalities in $X^2$ given below:
\[
\begin{array}{ll}
(\Lcal_{\alpha|X}\times \Lcal_{\alpha|X}\circ r)(x,y) &= ({}^{\alpha}{({}^xy)},{}^{\alpha}{(x^y)}) \\
r\circ (\Lcal_{\alpha|X}\times \Lcal_{\alpha|X})(x,y) &=({}^{{}^{\alpha}x}{({}^{\alpha}y)},({}^{\alpha}x)^{({}^{\alpha}y)}), \; \alpha \in Y, \; x,y \in X.
\end{array}
\]
Part (2) is analogous.

(3). We shall prove the first implication
\begin{equation}
\label{stueq3}
{\rm\bf stu 1 :}\quad {}^{{\alpha}^y}x  =  {}^{\alpha}x,\; \forall \alpha\in Y, x,y \in X \;\Longleftrightarrow \;\Lcal_{\alpha|X} \in \Aut (X, r_1), \; \forall \alpha\in Y.
\end{equation}
Recall first that the braided set $(Z,r)$ satisfies conditions \textbf{l1}, \textbf{lr3}, see Remark \ref{ybe}.

\textbf{stu 1} $\Longrightarrow \Lcal_{\alpha|X} \in \Aut (X, r_1)$.
Assume \textbf{stu 1} holds in $Z$. This is condition (3) of Lemma \ref{lemmixedactions}. Note that $(Z,r)$ satisfies \textbf{l1}, and therefore condition (1) in Lemma \ref{lemmixedactions} is also satisfied. Hence by Lemma \ref{lemmixedactions} the remaining condition (2) is also in force. This gives
\begin{equation}
\label{stueq4}
{}^{\alpha}{({}^yx)} = {}^{{}^{\alpha}y}{({}^{\alpha}x)},
\quad \forall \alpha\in Y, x,y \in X.
\end{equation}
We shall prove
\begin{equation}
\label{stueq5}
{}^{\alpha}{(x^y)}= ({}^{\alpha}x)^{({}^{\alpha}y)}, \quad
 \forall \alpha\in Y, x,y \in X.
\end{equation}
We use \textbf{lr3} and \textbf{stu1} to deduce the following equalities
\[
 \begin{array}{ll}
{({}^{\alpha}x)}^{({}^{{\alpha}^x}{y})} \ = \ {}^{({\alpha}^{{}^xy})}{(x^y)} \quad &:\text{by \textbf{lr3}}\\
{}^{({\alpha}^{{}^xy})}{(x^y)} = {}^{\alpha}{(x^y)}\quad &: \text{by \textbf{stu1}}\\
{({}^{\alpha}x)}^{({}^{{\alpha}^x}{y})} = {({}^{\alpha}x)}^{({}^{\alpha}{y})}\quad &:\text{by \textbf{stu1}},
\end{array}
\]
which imply  (\ref{stueq5}). Hence  $\Lcal_{\alpha|X} \in \Aut (X, r_1), \;\; \forall \alpha \in Y.$

$\Lcal_{\alpha|X} \in \Aut (X, r_1) \Longrightarrow \textbf{stu 1}$. Suppose $\Lcal_{\alpha|X} \in \Aut (X, r_1)$,
so by part (1) of our proposition
\[
{}^{\alpha}{({}^yx)} = {}^{{}^{\alpha}y}{({}^{\alpha}x)}, \forall x,y \in X,\]
which is exactly condition (2)
of Lemma \ref{lemmixedactions}. Condition (1)
 of Lemma \ref{lemmixedactions} holds, this is \textbf{l1},
and therefore, the remaining condition (3) of Lemma \ref{lemmixedactions}  is also satisfied, but this is exactly \textbf{stu 1}.
We have proven the equivalence (\ref{stueq3}).
Analogous argument proves the remaining three equivalences in (\ref{stueq2}).
\end{proof}

Lemma \ref{extlemma} and Proposition \ref{stuprop} imply straightforwardly the following.
\begin{corollary}
\label{stucor}
Suppose  $(Z,r)$ is a  nondegenerate injective braided set
which splits as a disjoint union $Z = X\bigcup Y$ of its $r$-invariant subsets $X$ and $Y$. Let
 $(X,r_1)$ and $(Y, r_2)$ be the induced subsolutions (so $(X, r_1)$ and $(Y, r_2)$ are also injective).
Then $(Z,r) = X\stu^{\ast} Y$ is a generalized strong twisted union if and only if the following four conditions hold.
\begin{enumerate}
\item[(i)]
\label{stutheor1}
The assignment $x
\longmapsto \Lcal_{x|Y}$
extends to a  group homomorphism
\[
\Lcal_{X|Y}:  G_X \longrightarrow \Aut(Y, r_Y).
  \]
\item[(ii)] The assignment $x \longrightarrow    \Rcal_{x|Y}$
extends to a  group homomorphism
\[\Rcal_{X|Y}:  G_X \longrightarrow \Aut(Y,r_Y).  \]
\item[(iii)]
\label{stutheor2}
The assignment $\alpha
\longrightarrow \Lcal_{\alpha|X}$
extends to a  group homomorphism
\[\Lcal_{Y|X}:  G_Y \longrightarrow \Aut(X,r).  \]
\item[(iv)] The assignment $\alpha
\longrightarrow   \Rcal_{\alpha|X}$
extends to
a  group homomorphism
\[\Rcal_{Y|X}:  G_Y \longrightarrow \Aut(X,r).  \]
\end{enumerate}
\end{corollary}

\begin{theorem}
\label{stuthm}
Suppose  $(Z,r)$ is a  nondegenerate 2-cancellative braided set
which splits as a generalized strong twisted union $Z = X\stu^{\ast} Y$ of its $r$-invariant subsets $X$ and $Y$. Let
 $(X,r_X)$ and $(Y, r_Y)$ be the induced subsolutions, $S=S(X,r_X)$, $T = S(Y, r_Y)$, $U= S(Z,r)$ in usual notation.
 Let $(S,r_S)$,  $(T,r_T)$, $(U,r_U)$ be the corresponding braided monoids, see Fact \ref{theoremA}.
 Then the following conditions hold.

 \begin{enumerate}
\item
The braided monoid $(U,r_U)$ has a canonical structure of a generalized strong twisted union
\[ (U,r_U) =  (S,r_S)\stu^{\ast}(T,r_T),\]
extending the ground actions  of the generalized strong twisted union $Z = X\stu ^{\ast} Y.$
\item
 Let
$(G_Z,r_{G_Z})$,  be the associated braided group. Suppose furthermore that $(Z, r)$ is injective, so
$X$ and $Y$ are also embedded in $G_Z$, and
let $G_1$ and $G_2$ be the subgroups of $G_Z$ generated by $X$ and $Y$, respectively. Then $G_1$ and $G_2$ are $r_{G_Z}$-invariant
and the braided group $(G_Z,r_{G_Z})$ has a canonical structure of a generalized strong twisted union
\[ (G_Z,r_{G_Z}) =  (G_1,r_1)\stu^{\ast}(G_2,r_2),\]
where $r_1$ is the restriction of $r_{G_Z}$ on $G_1\times G_1$ and $r_2$ is the restriction of $r_{G_Z}$ on $G_2\times G_2$.
\end{enumerate}
\end{theorem}

\begin{proof}
(1). It follows from  Fact \ref{theoremA} that $U = S(Z,r_Z)$ has the structure of a graded braided monoid $(U, r_U)$
 with a braiding operator $r_U$ extending $r$, moreover $(U, r_U)$ is a (general) extension of the disjoint braided monoids $(S, r_S)$ and $(T, r_T)$. We have to show that the four \textbf{stu} conditions are satisfied, see (\ref{stueq1}).
  We shall use induction on lengths of words to prove
\begin{equation}
\label{stueq7}
 \begin{array}{l}
 {\rm\bf stu 1 :} \quad {}^{u^b}a  =  {}^{u}a,\; \forall u\in T, a,b \in S.
\end{array}
\end{equation}
\emph{Step 1.} First we prove  (\ref{stueq7}) for all $a \in S$, $b=y\in X$, $u = \alpha \in Y$ by induction on the length $|a|$ of $a$.
Condition \textbf{stu 1} on $Z$ gives the base for the induction. Assume (\ref{stueq7}) is true for all $u \in Y, b \in X$ and all
 $a\in S$, with $|a| \leq n$. Suppose $a \in S, |a| = n+1$, $u = \alpha \in Y, b = y \in X$. Then $a =tc$, where $c\in S, |c|= n, t \in X$, and
 the following equalities hold in $U$.
 \begin{equation}
\label{stueq8}
 \begin{array}{lll}
 {}^{{\alpha}^y}a  = {}^{{\alpha}^y}{(tc)} &=({}^{{\alpha}^y}t) {}^{({\alpha}^y)^t}c \quad &:\text{by \textbf{ML2}}\\
                                           &=({}^{\alpha}t) {}^{({\alpha}^y)}c \quad &:\text{by \textbf{stu 1} and IH}\\
                                           &=({}^{\alpha}t) ({}^{\alpha}c) \quad &:\text{by \textbf{stu 1} and IH}.
   \end{array}
\end{equation}
where IH is the inductive assumption.
\begin{equation}
\label{stueq9}
 \begin{array}{lll}
 {}^{\alpha}a  = {}^{\alpha}{(tc)} &=({}^{\alpha}t) {}^{({\alpha}^t)}c \quad &:\text{by \textbf{ML2}}\\
                                       &=({}^{\alpha}t) ({}^{\alpha}c) \quad &:\text{by \textbf{stu 1} and IH}. \\
   \end{array}
\end{equation}
Equalities (\ref{stueq8}) and (\ref{stueq9}) imply ${}^{{\alpha}^y}a = {}^{{\alpha}}a$, and therefore
\begin{equation}
\label{stueq10}
 {}^{{\alpha}^y}a = {}^{{\alpha}}a \quad \forall a \in S, \forall y\in X, \alpha \in Y.
 \end{equation}

\emph{Step 2.} We use induction on the length $|u|$ of $u \in T$ to prove
\begin{equation}
\label{stueq11}
 {}^{{u}^y}a = {}^{{u}}a \quad \forall a \in S, u \in T, y\in X.
 \end{equation}

Condition (\ref{stueq10}) gives the base for the induction.
Assume (\ref{stueq11}) holds for all $a \in S, y \in X$, and all $u \in T,$ with $|u|\leq n$.
Let $a \in S, y \in X$, and  $u \in T, |u|= n+1$. Then $u = \alpha v, v \in T, |v|= n, \alpha \in Y$ and the following equalities hold
in $U$.

\begin{equation}
\label{stueq12}
 \begin{array}{lll}
{}^{{u}^y}a= {}^{{(\alpha v)}^y}a  &= {}^{({\alpha}^{{}^vy}) (v^y)}a \quad &:\text{by \textbf{MR2}}\\
                                  &={}^{({\alpha}^{{}^vy})} {({}^{(v^y)}a)} &\\
                                  &={}^{\alpha}{({}^va)} \quad &:\text{by \textbf{stu 1} and IH}\\
                                  &=  {}^{(\alpha v)}{a}  = {}^ua. \quad &
   \end{array}
\end{equation}
This proves (\ref{stueq11})

\emph{Step 3.} Finally, we prove (\ref{stueq7}), for all $a,b \in S, u \in T,$  by induction on the length $|b|$ of $b$.
The base of the induction is given by (\ref{stueq11}). Assume (\ref{stueq7}) holds for all $b \in S$, with $|b|\leq n$.
Let $b = cy, c \in S, |c| = n,  y \in X$.

\begin{equation}
\label{stueq13}
 \begin{array}{lll}
{}^{{u}^b}a= {}^{({u}^{cy})}a  &= {}^{({u}^c)^y}a&\\
                         &= {}^{({u}^c)}a \quad &:\text{since ${u}^c\in T$ and by  IH}\\
                         &= {}^ua \quad &:\text{by  IH}.\\

   \end{array}
\end{equation}
The remaining \textbf{stu} conditions, see (\ref{stueq1}), are proven by a similar argument.
We have proven part (1).

Each of the parts (1) and (2)  should be proved separately, although we use similar arguments, since, in general, the braided monoids $U, S, T$ are not embedded in the corresponding braided groups.

Sketch of proof of (2).
Note that every element $a \in G$ can be presented as a monomial
\begin{equation}
\label{keyeq2} a = \zeta_1\zeta_2 \cdots \zeta_n,\quad \zeta_i \in Z
\bigcup Z^{-1}.
\end{equation}
By convention we consider  \emph{a reduced form of} $a$, that is a
presentation \ref{keyeq2} with minimal length $n$.
Bearing this in mind, we prove (\ref{stueq1}) in $G_Z$,
using an argument similar to our argument for monoids, but at each step we use
induction on the length $n$ of \emph{the reduced form} of the corresponding words $a$, $u$, $b$.
\end{proof}

 \begin{corollary}
 \label{SD_stu_lemmaCOR}
In notation as in Theorem \ref{stuthm}.
  Suppose $(Z,r)$ is a 2-cancellative \textbf{SD}
 braided set (that is $(X,\la)$ is a rack),
 which decomposes as a union of disjoint $r$-invariant subsets $Z= X \bigcup Y.$ Then
 $Z$ is a generalized strong twisted union of racks, $Z = X \stu^{\ast} Y$.
 Moreover, (i) the braided monoids $(U,r_U),  (S,r_S),  (T, r_T)$, are self distributive
 and $U$ is a generalized strong twisted union
 \[ (U,r_U) =  (S,r_S)\stu^{\ast}(T,r_T);\]
 (ii) Let
$(G_Z,r_{G_Z})$,  be the associated braided group, and suppose $(Z, r)$ is injective,
so $X$ and $Y$ are also embedded in $G_Z$.
Let $G_1$ and $G_2$ be the subgroups of $G_Z$ generated by $X$ and $Y$, respectively. Then $G_1$ and $G_2$ are $r_{G_Z}$-invariant
and the braided group $(G_Z,r_{G_Z})$ has a canonical structure of a generalized strong twisted union
\[ (G_Z,r_{G_Z}) =  (G_1,r_1)\stu^{\ast}(G_2,r_2),\]
where $r_1$ is the restriction of $r_{G_Z}$ on $G_1\times G_1$ and $r_2$ is the restriction of $r_{G_Z}$ on $G_2\times G_2$.
\end{corollary}

\subsection{"Local" conditions  sufficient for a generalized strong twisted unions of nondegenerate braided sets to be also a braided set}
\label{subsec:StuBraidedsets}
\begin{definition}\label{extendedleftaction} \cite{GIM08}
Given a quadratic set $(X,r)$ we extend the  actions ${}^x\bullet$ and  $\bullet
^x$ on $X$ to left and right actions on $X\times X$ as follows.
For $x,y,z \in X$ we define:
\[ {}^x{(y,z)}:=({}^xy,{}^{x^y}z), \quad
\text{and} \quad (x,y)^z:= (x^{{}^yz}, y^z).\]
The map $r$ is called, respectively, \emph{left and right invariant} if
\[\begin{array}{lclc}
 {\bf l2:}\quad&  r({}^x{(y,z)})={}^x{(r(y,z))},
 \quad\quad\quad
 & {\bf r2:}\quad&
r((x,y)^z)={(r(x,y))}^z
\end{array}
\]
hold for all $x,y,z\in Z$.
\end{definition}
Conditions \textbf{l2} and \textbf{r2} give a more compact way to express \textbf{l1}, \textbf{r1}, \textbf{lr3}, since
the following implications hold:
\begin{equation}
\label{l2l1}
{\bf l2} \Longleftrightarrow {\bf l1,lr3};\quad\quad\quad\quad\quad\quad
{\bf r2} \Longleftrightarrow {\bf r1,lr3}.
\end{equation}

\begin{remark}
\label{leml1r2} \cite{GIM08}
Let $(X,r)$ be a quadratic set.  Then the following three conditions are equivalent:
(a) $(X,r)$ is a braided set;
(b) $(X,r)$ satisfies {\bf l1} and {\bf r2};
(c) $(X,r)$ satisfies {\bf r1} and {\bf l2}.
\end{remark}

\begin{notation}
\label{mlnotation}
\cite{GIM08}
 When we study extensions it is convenient to have a `local' notation for some of our conditions, in which the specific elements for which the condition is being imposed will be explicitly indicated in lexicographical order of first appearance. Thus for example {\bf l1(x,y,z)} means the condition exactly as written in Remark~\ref{ybe} for the specific elements
 $x,y,z$. Similarly {\bf r2(x,y,z)} means for the elements $x,y,z$ exactly  as appearing as in Definition~\ref{extendedleftaction}.

In this section we consider triples in the set $Z^3$ so for example
\[ {\bf l1(x, \alpha,y)}:\quad {}^x{({}^{\alpha}y)}={}^{{}^x{\alpha}}{({}^{x^{\alpha}}y)}, \; \alpha,x,y\in Z.\]

Finally, we use this notation to specify the restrictions of any of our conditions to subsets of interest. For example
\[\begin{array}{l}
{\bf l1(X,Y,X)}:=\{{\bf l1(x, \alpha,y)},\quad \forall
 \; x,y\in X, \;\alpha\in Y\}.\\
{\bf r2(X,Y,X)}:=
(r(x, \alpha))^y =r({(x, \alpha)}^y),  \quad \forall x, y\in X, \alpha\in Y.
\end{array}
\]
  \end{notation}

The following result gives a necessary and sufficient condition so that a (general) quadratic set which is a generalized strong twisted union
$(Z, r) =(X,r_1)\stu^{\ast}(Y,r_2)$ of two disjoint braided sets is also a braided set.

\begin{proposition}
\label{stutheor}
Suppose a nondegenerate and injective quadratic set $(Z,r)$ is a generalized strong twisted union of two disjoint 2-cancellative braided sets
$(X,r_X)$ and $(Y, r_Y)$.
 Then $(Z,r)$ obeys YBE \emph{iff} the following hold.
\begin{enumerate}
\item
\label{stutheor1}
Conditions (i) through (iv) in Corollary \ref{stucor} are satisfied;
\item
\label{stutheor3}
The actions satisfy the following four "mixed" conditions
\begin{equation}
\label{mixedcond}
{\rm\bf l1(X,Y,X)},\; \;{\rm\bf r2(X,Y,X)}, \;\; {\rm\bf l1(Y,X,Y)},\; \; {\rm\bf r2(Y,X,Y)}.
\end{equation}
\end{enumerate}
\end{proposition}
\begin{proof}
The proof is routine and an experienced reader may skip it.

Assume $(Z,r)$ obeys YBE. Then, by Remark \ref{leml1r2}, conditions {\bf l1} and {\bf r2}, (and {\bf r1} and {\bf l2}) are satisfied for any triple $(a,b,c) \in Z^3$. In particular, the mixed conditions (\ref{mixedcond}) hold, which proves (2).
By assumption the braided set $(Z,r)$ is a strong twisted union $Z = X\stu^{\ast} Y$, so the hypothesis of Corollary \ref{stucor} is satisfied, which implies (1).

Assume now that (1) and (2) are satisfied. We have to show that $(Z,r)$ is a braided set.
Recall that the YB-diagram starting with the triple $(a, b, c)\in Z^3$ shows that
\[\begin{array}{lll}
r^{12}r^{23}r^{12}(a,b,c)&=r^{23}r^{12}r^{23}(a,b,c)& \Longleftrightarrow {\rm\bf r1(a,b,c), l2(a,b,c)} \\
                                     & &\Longleftrightarrow {\rm\bf l1(a,b,c), r2(a,b,c)},\\
                                     &\forall a,b,c \in Z.&
\end{array}
 \]
 There is nothing to prove if $(a,b,c) \in X^3,$ or $(a,b,c) \in Y^3$, since by hypothesis $(X,r_X)$ and $(Y, r_Y)$ are braided sets.

Our argument uses the presentation of the set $Z^3 \setminus (X^3\bigcup Y^3)$ as a union of six disjoint subsets
\[\begin{array}{ll}Z^3 \setminus (X^3\bigcup Y^3)= &
                    (X\times X\times Y) \bigcup (Y\times X\times X)\bigcup(X\times Y\times Y)\\& \bigcup(Y\times Y\times X)\bigcup(X\times Y\times X) \bigcup (Y\times X\times Y).
                    \end{array}
                    \]
Clearly, $(Z,r)$ obeys YBE \emph{iff} each of the sets on the right-hand side of the above equality satisfies simultaneously the "mixed" conditions \textbf{l1} and  \textbf{r2} (or equivalently, \textbf{r1} and  \textbf{l2}).
Analyzing with details each of the corresponding six cases we note that

Condition (1) implies (a) \textbf{l1}(X,X,Y) and \textbf{r2}(X,X,Y); (b) \textbf{l1}(Y,X,X), and \textbf{r2}(Y,X,X);
(c) \textbf{l1}(X,Y,Y), and \textbf{r2}(X,Y,X); (d) \textbf{l1}(Y,Y,X), and \textbf{r2}(Y,Y,X).
(In fact  (1) encodes exactly these eight (mixed) conditions).

 Condition (2) gives the missing  "mixed" conditions (\ref{mixedcond}) not encoded in  (1).
\end{proof}

\section{Questions}
\label{sec:questions}
\subsection{Some open questions}
\label{sec:open_questions}

\begin{question}
\label{new_question1}
Let $(X,r)$ be a square-free nondegenerate
   \emph{ quadratic set} of finite order $|X|=n$.
  Suppose its associated algebra $A=  A(\textbf{k}, X, r)$ is a \emph{PBW} algebra.

  (We know that these assumptions imply that  $r^2 = 1$, and $(X,r)$ is 2-cancellative, see Sec 3).
\begin{enumerate}
 \item
  Is it true that the algebra $A$ has polynomial growth?

An equivalent question is:
 \item Is it true that the algebra $A$ has finite global dimension?
\end{enumerate}
\end{question}
This is so for $|X|=3,$ see Lemma \ref{lem:n3}.

For each $n \geq 3$ an affirmative answer of (1), or (2) would imply that $(X,r)$ is a solution of YBE, and all conditions (1) through (8) in Theorem
\ref{thm:new_theorem} are satisfied. A counterexample would also be interesting.

\begin{question}
\label{new_question2}
Suppose $(X, r)$ is a square-free 2-cancellative \emph{quadratic set} of finite
order $|X| \geq 3$.

\begin{enumerate}
\item
\label{new_question2_1}
Is it true that if $(X, r)$ is self distributive and satisfies the minimality condition, $\dim A_2 = 2|X|-1$, then  $(X, r)$ is a braided set?

Our assumptions imply that $(X, r)$ is nondegenerate and $\Lcal_x(y)\neq y, \forall  x,y\in X, x\neq y$, see Lemma \ref{minimaldim_lemma}.
\item
\label{new_question2_2}
In particular,  is it true that if $(X, r)$ is a self distributive quadratic set, of prime
order $|X| = p$ and satisfies the minimality condition, $\dim A_2 = 2|X|-1$, then $\Lcal_x^2= \id_X, \forall x \in X$?


\item  What can be said about a (general) square-free 2-cancellative \emph{quadratic set} $(X, r)$ if its Koszul dual algebra satisfies $A^{!}_3 =0$?
In
particular, study the braided sets $(X, r)$ for which $A^{!}_3 =0$.
\end{enumerate}
\end{question}

It follows from Corollary \ref{Corquandle5}, and Lemma \ref{minimaldim_lemma2} that the answers to (\ref{new_question2_1}) and (\ref{new_question2_2})  are affirmative, whenever $3 \leq |X|\leq 5$. In this case  (up to isomorphism) there are two \textbf{SD} quadratic sets with 2-cancellation and satisfying the minimality condition, namely:
(a) $(X,r)$ is the quadratic set corresponding to the dihedral quandle of order $3$; and (b)
$(X,r)$ is the quadratic set corresponding to the dihedral quandle of order $5$. Clearly, each of those is a braided set.


\begin{problem}
\label{problemM1}
\textbf{\emph{Given the following data:}}
(a) A set $X$ of odd cardinality $n = 2k+1$;
(b) a cyclic  permutation $r_0 \in \Sym (X^2 \setminus \Delta_2)$ of order $n$
\[\cO: a_1b_1 \longrightarrow_{r_0} a_2b_2 \longrightarrow_{r_0} \cdots \longrightarrow_{r_0}   a_{n} b_{n}\longrightarrow_{r_0}a_1b_1,
\]
where $a_i\neq b_i, 1 \leq i \leq n$, $a_i \neq a_j, b_i \neq b_j$, whenever $i \neq j$, $1\leq i,j\leq n$.

\textbf{\emph{Find}} \emph{an extension} $r : X\times X  \longrightarrow X\times X $ of $r_0$
(equivalently, \emph{find all maps $\Lcal_x, x \in X$ explicitly}),
so that

(i) $(X,r)$ is a 2-cancellative square-free \textbf{SD}
\emph{quadratic set} (\emph{we do not assume that $(X,r)$ is a solution}) ;

(ii) $\Lcal_x^2 = \id, \forall  x \in X$.

\emph{Analyze the obtained quadratic set}. In particular, decide (a) whether this data determines an \textbf{SD} solution of YBE?
(b) If moreover, $n=p$ is a prime number
 and the quadratic set $(X,r)$ satisfies the minimality condition \textbf{M} does this imply that $(X,r)$ is a braided set?
\end{problem}

\subsection{Questions posed in a previous version of this work which have been recently answered}

Various questions on braided sets posed in \cite{GI19}v3, and v1  were recently answered in \cite{CJO19}. We give an account of some of our previous questions .

\textbf{Question 5.8.}, \cite{GI19}v3.  (1) For which integers n this lower bound is attainable, that is
there exists a braided set $(X, r)$, $|X| = n$ satisfying the minimality condintion \textbf{M}?
(2) Classify the square-free solutions $(X, r)$ satisfying the minimality condition \textbf{M}. - A complete answer is given in \cite{CJO19}.

\textbf{Conjecture} \textbf{5.10},\cite{GI19}v3.
Let $(X,r)$ be an arbitrary finite nondegenerate braided set with 2-cancellation. Then the monoid $S(X,r)$ is cancellative if and only if $r$ is involutive.

Theorem 5.5, \cite{GI19}v3 (here Theorem \ref{mainth}) confirms this conjecture in the case when $(X,r)$ is an arbitrary square-free nondegenerate braided set of order $|X| =n$.
It was shown in \cite{JespersKuAntwerp}, Theorem 4.5,  that the Conjecture is true for arbitrary finite nondegenerate set-theoretic
solution $(X,r)$ of the Yang-Baxter equation.

\textbf{Questions 6.3.1.} \cite{GI19}v3.
The following questions refer to finite square-free solutions $(X, r)$
which are 2-cancellative.
\begin{enumerate}
\item
Is it true that if a dihedral quandle (X, r) satisfies the minimality condition
M then its order |X| is a prime number?
- \emph{Confirmed in  }\cite{CJO19}.

\item Suppose $(X, r)$ is an indecomposable quandle such that the corresponding
solution $(X, r)$ satisfies the minimality condition M. Does this imply that
the quandle $(X, r)$ is simple?
 -\emph{Yes}, \cite{CJO19}
\item Which of the known simple quandles satisfy the minimality condition M?
\emph{Answer-the dihedral quandles of prime order $p$, see} \cite{CJO19}

\item Study general square-free noninvolutive, braided sets (X, r) which are not
self distributive.
\emph{-This is an ongoing project.}

 Our results in Section \ref{sec:specialext}, see Theorem \ref{theor:irregularext}, and Corollary \ref{irregularextcor} give a method for constructions
  of 'new' noninvolutive solutions $(Z,r)$ with a
prescribed orders $|Z|$, and $|r|$. In this case $(Z,r)$ is a generalized strong twisted union $Z = X\stu^{\ast} Y$  of
 involutive (or noninvolutive) disjoint solutions $(X,r_X)$, $(Y, r_Y)$.

\item Classify the square-free noninvolutive, braided sets (X, r) whose quadratic
algebra satisfy $\gkdim A(k, X, r) = 1$.

Some answers are given in \cite{CJO19}, Example 5.1

\item Classify the square-free noninvolutive braided sets of small orders. In particular, classify the square-free noninvolutive, and not \textbf{SD} braided sets
$(X, r)$ of small order.

\item
 Find examples of indecomposable (not \textbf{SD}) finite square-free solutions.

\item Find examples of indecomposable (not SD) square-free solutions which
satisfy the minimality conditions M.

A complete classification of (general) square-free nondegenerate solutions $(X,r)$ satisfying the minimality condition \textbf{M} is given by  Ced\'{o}, Jespers, and Okninski,
see \cite{CJO19}, Theorem 5.5, and Cor. 5.6. The classification
is made in terms of the so called \emph{derived solution} $(X,r^{\prime}).$
\end{enumerate}

 \begin{remark}
 We have shown that if $(X,r)$ is a finite nondegenerate square-free braided set, where $r$ is not involutive, then the monoid $S=S(X,r)$ is not cancellative (even if $(X,r)$ is 2-cancellative).
 This gives a negative answer
 to Open Question 3.24 in \cite{GIM08}:
 "\emph{Is it true that if $(X,r)$ is a 2-cancellative braided set, then
 the associated monoid $S(X,r)$ is cancellative?}"
 \end{remark}

{\bf Acknowledgments}. This paper was written during my visits at the Abdus Salam
International Centre for Theoretical Physics (ICTP), Trieste, Summer
2018, and at Max Planck Institute for Mathematics, Bonn in 2019. It is my pleasant duty to thank MPIM, Bonn, and the Mathematics group of ICTP
for the inspiring
atmosphere. I thank the referees for their comments and very useful suggestions.

\end{document}